\documentclass[final]{siamart0516}
\usepackage{amssymb}
\usepackage{graphicx}
\usepackage{url}
\usepackage{hyperref}
\usepackage{algorithm}
\usepackage{algpseudocode}
\usepackage{pdflscape}
\usepackage{xcolor}

\usepackage{booktabs} % Horizontal rules in tables
\usepackage{float} % Required for tables and figures in the multi-column environment - they need to be placed in specific locations with the [H] (e.g. \begin{table}[H])
\usepackage{multicol, multirow} % Used for the two-column layout of the document
\usepackage{caption} % Custom captions under/above floats in tables or figures
\usepackage{colortbl}

\definecolor{blue}{rgb}{0,0,0.9}
\definecolor{red}{rgb}{0.9,0,0}
\definecolor{green}{rgb}{0,0.9,0}

\usepackage[left=3.25cm,top=3.1cm,right=3.0cm,bottom=3.25cm,bindingoffset=0.5cm]{geometry}
\newtheorem{assumption}{Assumption \!\!}
\theoremstyle{remark}
\newtheorem*{remark}{Remark}
\newtheorem{fact}{Fact}[section]

\newcommand{\R}{\mathbb{R}}
\newcommand{\Rext}{\mathbb{R}\cup\{+\infty\}}
\newcommand{\Spd}{\mathbb{S}}
\newcommand{\Id}{\mathbb{I}}
\newcommand{\abs}[1]{\left\vert#1\right\vert}
\newcommand{\set}[1]{\left\{#1\right\}}
\newcommand{\norm}[1]{\left\Vert#1\right\Vert}
\newcommand{\norms}[1]{\Vert#1\Vert}

\newcommand{\prox}{\mathrm{prox}}

\newcommand{\argmin}{\mathrm{arg}\!\displaystyle\min}
\newcommand{\dom}[1]{\mathrm{dom}(#1)}
\newcommand{\iprods}[1]{\langle #1\rangle}

\newcommand{\trace}[1]{\mathrm{trace}\left(#1\right)}

\newcommand{\intx}[1]{\mathrm{int}\left(#1\right)}

\newcommand{\dist}[1]{\mathrm{dist}\left(#1\right)}

\newcommand{\BigO}[1]{\mathcal{O}\left(#1\right)}
\newcommand{\Eproof}{\hfill $\square$}

\newcommand{\Xc}{\mathcal{X}}

\newcommand{\Sc}{\mathcal{S}}

\newcommand{\Lc}{\mathcal{L}}

\newcommand{\Qc}{\mathcal{Q}}

\newcommand{\Cc}{\mathcal{C}}

\newcommand{\Pc}{\mathcal{P}}
\newcommand{\Kc}{\mathcal{K}}

\newcommand{\myeq}[2]{\vspace{-0.75ex}
\begin{equation}\label{#1}
{#2}
\vspace{-0.75ex}
\end{equation}
}
\newcommand{\myeqn}[1]{\vspace{-0.75ex}
\begin{equation*}
{#1}
\vspace{-0.75ex}
\end{equation*}
}

\newcommand{\TheTitle}{A New Homotopy Proximal Variable-Metric Framework for Composite Convex Minimization} 
\newcommand{\TheRunningTitle}{A New Homotopy Proximal Variable-Metric Framework}
\newcommand{\TheAuthors}{Quoc Tran-Dinh, Liang Ling, and Kim-Chuan Toh}
\newcommand{\TheAbAuthors}{Q. Tran-Dinh, Liang Ling, and Kim-Chuan Toh}

\headers{\TheRunningTitle}{\TheAbAuthors}

\title{{\TheTitle}}
\author{
Quoc Tran-Dinh\thanks{Department of Statistics and Operations Research, 
University of North Carolina at Chapel Hill (UNC), 333-Hanes Hall, Chapel Hill, NC27599-3260, USA. 
\newline Email: {\tt quoctd@email.unc.edu}.}
\and
Liang Ling$^{\dagger}$
\and
Kim-Chuan Toh\thanks{Department of Mathematics, and Institute of Operations Research and Analytics, National University of Singapore, 10 Lower Kent Ridge Road, Singapore 119076. 
\newline Email:{\tt \{liang.ling, mattohkc\}@nus.edu.sg)}} 
}

\ifpdf
\hypersetup{
  pdftitle={\TheTitle},
  pdfauthor={\TheAuthors}
}
\fi

\begin{document}
\maketitle

\begin{abstract}
This paper suggests two novel ideas to develop new proximal variable-metric methods for solving a class of composite convex optimization problems.
The first idea is a new parameterization of the optimality condition which allows us to develop a class of homotopy proximal variable-metric methods.
We show that under appropriate assumptions such as strong convexity-type and smoothness, or self-concordance, our new schemes can achieve \textit{finite} global iteration-complexity bounds.
Our second idea is a primal-dual-primal framework for proximal-Newton methods which 
{can lead to} some useful computational  features for a subclass of nonsmooth composite convex optimization problems.
Starting from the primal problem, we formulate its dual problem, and use our homotopy proximal Newton method to solve this dual problem.
Instead of solving the subproblem directly in the dual space, we suggest to dualize this subproblem to go back to the primal space. 
The resulting subproblem shares some similarity promoted by the regularizer of the original problem and leads to some computational advantages.
As a byproduct, we specialize the proposed algorithm to solve covariance 
{estimation} problems.
Surprisingly, our new algorithm does not require any matrix inversion or Cholesky factorization, and function evaluation, while it works {in the primal space with 
sparsity structures that are} promoted by the regularizer.
Numerical examples on several applications are given to illustrate our theoretical development and to compare with state-of-the-arts.

\vspace{1ex}
\noindent\textbf{Keywords:}
Homotopy method; proximal variable-metric algorithm; global convergence rate; finite iteration-complexity; primal-dual-primal framework; composite convex minimization.
\end{abstract}

% REQUIRED
\noindent
\begin{AMS}
90C25, 90C06, 90-08
\end{AMS}

%%%%%%%%%%%%%%%%%%%%%%%%%%%%%%%%%%%%%%%%%%%%%%%%%
%% 1. Introduction.
%%%%%%%%%%%%%%%%%%%%%%%%%%%%%%%%%%%%%%%%%%%%%%%%%
\section{Introduction}\label{sec:intro}
%%%%%%%%%%%%%%%%%%%%%%%%%%%%%%
%%% 1.1. Problem statement.
%%%%%%%%%%%%%%%%%%%%%%%%%%%%%%
\paragraph{\textbf{Problem statement}}
We are interested in the following  composite convex minimization template that covers 
{various} of applications in different fields including statistics, machine learning, image and signal processing, and engineering \cite{Bauschke2011,Beck2009,Boyd2011,Chambolle2011,Esser2010a,Parikh2013}:
\myeq{eq:composite_cvx}{
F^{\star} := \min_{x\in\R^p} \Big\{ {F(x)} := f(x) + g(x) \Big\},
}
where $f : \R^p\to\Rext$ and $g : \R^p\to\Rext$ are proper, closed, and convex functions.
Here, $f$ often represents a loss function or a data fidelity term, while $g$ is considered as a regularizer or a penalty to promote some desired structures of the final solutions.
%\fbox{{I change to $F$ since the proofs  later used $F$.}}

%%%%%%%%%%%%%%%%%%%%%%%%%%%%%%
%%% 1.2. Research questions.
%%%%%%%%%%%%%%%%%%%%%%%%%%%%%%
\paragraph{\textbf{Motivation}}

This paper aims at addressing two questions arisen from numerical methods for solving \eqref{eq:composite_cvx}.
The first question {concerns} the global iteration-complexity of second-order-type methods.
It is well-known that second-order methods such as Newton-type algorithms have fast local  convergence rates under certain assumptions.
In particular, the classical Newton method can achieve a local quadratic convergence rate under the local Lipschitz continuity of the Hessian around  an optimal solution and the regularity of such an optimal solution \cite{Deuflhard2006}.
However, global convergence behaviors as well as global convergence rates and iteration-complexity estimates of second-order-type methods have not yet been well understood.
Recent attempts {to address the aforementioned issues} have been made for Newton-type methods \cite{Nesterov2008b,Nesterov1994,Nesterov2006a},  but they are still limited to some subclasses of problems such as self-concordant and global Lipschitz Hessian functions.
%\fbox{given settings and are also insufficient. {This phrase is not clear.}}
In the first part of this paper, we address the following question.
\begin{itemize}
\item \textit{When can we design second-order-type methods that achieve global iteration-complexity?}
\end{itemize}
Unfortunately, we do not have a complete answer for this question.
However, we identify three different subclasses of \eqref{eq:composite_cvx} where we can develop new proximal variable-metric methods to achieve a global iteration-complexity.
Our algorithms can solve nonsmooth instances of \eqref{eq:composite_cvx}, but require $f$ to be smooth and satisfy other additional  {mild} conditions that are different from existing methods.

In the second part of this paper, we address another situation of \eqref{eq:composite_cvx}.
We observe that existing methods for solving nonsmooth instances of \eqref{eq:composite_cvx} can be classified into the following categories:
\begin{enumerate}
\item[(a)] If $f$ is smooth with the Lipschitz gradient and $g$ is nonsmooth but proximally tractable as defined in Subsection \ref{subsec:prox_oper}, then accelerated proximal-gradient methods achieve optimal convergence rate of $\BigO{\frac{1}{k^2}}$, where $k$ is the iteration counter.
If $f$ is twice differentiable, and its Hessian is Lipschitz continuous \cite{Lee2014} or self-concordant \cite{Nesterov2004}, then we can apply proximal-Newton methods \cite{Lee2014,Tran-Dinh2013b,Tran-Dinh2013a} to efficiently solve \eqref{eq:composite_cvx}.
\item[(b)] If both $f$ and $g$ are proximally tractable, then operator splitting schemes such as Douglas-Rachford's methods can be used to efficiently solve \eqref{eq:composite_cvx} but with a sublinear rate \cite{Bauschke2011,Davis2014}.
\item[(c)] If $g(x) = \psi(Dx)$ for a given linear operator $D$, and both $f$ and $\psi$ are proximally tractable, then primal-dual methods such as Chambolle-Pock's and primal-dual hybrid gradient methods, and alternating direction methods of multipliers (ADMM) can be applied to \eqref{eq:composite_cvx}.
These methods also achieve a sublinear rate in general.
\end{enumerate}
We instead consider the following subclass of \eqref{eq:composite_cvx}, where
{
\begin{enumerate}
\item[(d)]  $f$ is self-concordant as defined in Definition \ref{de:gen_sel_con_def}; and
$g$ is given by $g(x) = \psi(Dx)$, where $D$ is a linear operator, and $\psi$ is nonsmooth and convex, but proximally tractable.
\end{enumerate}
}
Under this setting, existing methods such as proximal-gradient-type schemes are often not efficient {for solving \eqref{eq:composite_cvx} due to the 
expensive evaluation of the} proximal operator of $g$.
We address the following research question:
\begin{itemize}
\item \textit{What is an appropriate solution method to solve 
 \eqref{eq:composite_cvx} under the conditions {stated in the subclass $(\mathrm{d})$}?}
\end{itemize}
This question may have multiple answers. One can apply some primal-dual methods to solve it.
However, these methods only have a sublinear convergence rate.
We instead propose a primal-dual-primal approach to solve  \eqref{eq:composite_cvx}  which consists of the following steps:
\begin{enumerate}
\item Construct the Fenchel dual problem of \eqref{eq:composite_cvx} when $g(x) = \psi(Dx)$.
\item Apply our homotopy proximal-Newton method in the first part to solve the dual problem.
\item Instead of solving the dual subproblem, we dualize it to go back to the primal space.
\item Construct an approximate primal solution of \eqref{eq:composite_cvx} from its dual approximate {solution.}
\end{enumerate}
The idea of using primal-dual approach is classical, but our primal-dual-primal method has various computational advantages as well as a linear convergence rate when it is applied to the  subclass {(d)} of \eqref{eq:composite_cvx}.
As a motivating example, we will show in Section \ref{sec:GL_app} that this approach is very suitable for covariance estimation problem \eqref{eq:GL_prob} below.

%%%%%%%%%%%%%%%%%%%%%%%%%%%%%%
%%% 1.3. Examples
%%%%%%%%%%%%%%%%%%%%%%%%%%%%%%
\paragraph{\textbf{Examples}}
Apart from two research questions above, our paper is also motivated by several prominent applications.
Let us recall  a few concrete examples of \eqref{eq:composite_cvx}:
\begin{enumerate}
\item \textit{Covariance estimation models: } 
If $f(X) := -\log\det(X) + \trace{\Sigma X}$ in \eqref{eq:composite_cvx}, where $\Sigma$ is a given symmetric matrix, then \eqref{eq:composite_cvx} covers both covariance and inverse covariance estimation problems in the literature depending on the choice of $g$ \cite{Banerjee2008,Friedman2008,Kyrillidis2014}:
\myeq{eq:GL_prob}{
\phi^{\star} := \min_{X\succ 0}\Big\{ \phi(X) :=  \trace{\Sigma X} - \log\det(X)  +  g(X) \Big\},
}
\item \textit{Poisson log-likelihood models:}
If we choose $f(x) := \sum_{i=1}^n\left( a_i^{\top}x - y_i\log(a_i^{\top}x)\right)$, where $\set{(a_i, y_i)}_{i=1}^n$ is a given dataset, then \eqref{eq:composite_cvx} covers Poisson log-likelihood models {used} in medical imaging, see, e.g., \cite{Lefkimmiatis2013}.
\item \textit{Regularized logistic regression:} 
If we choose $f(x) :=  \frac{1}{n}\sum_{i=1}^n\log( 1 + \exp(-y_i(a_i^{\top}x))) + \frac{\mu_f}{2}\norms{x}_2^2$, where $\set{(a_i, y_i)}_{i=1}^n$ is a given dataset, and $\mu_f > 0$ is a regularization parameter, then \eqref{eq:composite_cvx} covers {the} well-known logistic models including both sparse and group sparse settings under an appropriate choice of $g$.
\item \textit{Poisson regression:} 
If $f(x) := \frac{1}{n}\sum_{i=1}^n\left(-y_i\exp(-\frac{1}{2}a_i^{\top}x) + \exp(\frac{1}{2}a_i^{\top}x)\right) + \frac{\mu_f}{2}\norms{x}^2$, where $\set{(a_i, y_i)}_{i=1}^n$ is given, then we obtain a Poisson regression problem as studied in \cite{ivanoff2016adaptive,jia2017sparse}.
\item \textit{Distance-weighted discrimination} (DWD): If $f(x) := \frac{1}{n}\sum_{i=1}^n\frac{1}{(a_i^{\top}x + \mu_i)^q} + \frac{\mu_f}{2}\norms{x}_2^2$, for some fixed order $q > 0$, then this model can be considered as a slight modification of the distance-weighted discrimination (DWD) for binary classification studied in \cite{lam2018fast,marron2007distance}.
\end{enumerate}
Many other applications of \eqref{eq:composite_cvx}  that fit our assumptions can be found, {for example} in \cite{Dalalyan2013,Ostrovskii2018,Tran-Dinh2013a}.

%%%%%%%%%%%%%%%%%%%%%%%%%%%%%%
%%% 1.4. Literature review.
%%%%%%%%%%%%%%%%%%%%%%%%%%%%%%
\paragraph{\textbf{Literature review}}
Problem \eqref{eq:composite_cvx} is well studied in the literature under different assumptions on $f$ and $g$.
Hitherto, several methods have been proposed for solving \eqref{eq:composite_cvx}.
Such methods include disciplined convex programming \cite{Grant2006,Wright2009}, proximal gradient and accelerated proximal gradient \cite{Beck2009,Nesterov2004,Nesterov2007}, proximal Newton-type \cite{Becker2012a,Lee2014,Tran-Dinh2013a}, splitting and alternating optimization \cite{Boyd2011,Esser2010a,Goldfarb2012,Yuan2012}, primal-dual \cite{Chambolle2011,TranDinh2015b}, coordinate descent \cite{fercoq2015accelerated,Richtarik2012,nesterov2017efficiency,wright2015coordinate}, conditional gradient \cite{harchaoui2012conditional,Jaggi2013}, stochastic gradient-type methods \cite{allen2016katyusha,defazio2014saga,nitanda2014stochastic,shalev2013stochastic,xiao2014proximal}, and incremental proximal gradient schemes \cite{Bertsekas2011b}.

Existing first-order methods for solving \eqref{eq:composite_cvx} heavily rely on the
{assumption that $f$ has
 Lipschitz gradient} \cite{Nesterov2004} and the proximal tractability of $g$ \cite{Bauschke2011,Parikh2013} as defined in Subsection~\ref{subsec:prox_oper}.
Another common subclass of \eqref{eq:composite_cvx} is that $g(x) = \psi(Dx)$ for a given linear operator $D$, and both $f$ and $\psi$ are proximally tractable.
Under this setting, operator splitting and primal-dual approaches can be applied to solve \eqref{eq:composite_cvx}.
Notable works  in this direction include primal-dual hybrid gradient schemes, Chambolle-Pock's methods, Douglas-Rachford and Vu-Condat splitting algorithms, and alternating direction methods of multipliers \cite{Bauschke2011,Boyd2011,Chambolle2011,Esser2010a,Goldstein2009}.
While first-order methods offer a low per-iteration {computational} complexity, they often require {a large number of} iterations and have a sublinear convergence rate.
In addition, their efficiency also depends {sensitively on the scaling and conditioning} of the problem \cite{giselsson2014diagonal}.

Proximal second-order methods such as proximal quasi-Newton \cite{Becker2012a,karimi2017imro,stella2017forward} and proximal-Newton methods \cite{Lee2014,Tran-Dinh2013a} often achieve a high accuracy solution and have good local convergence rate {but} they usually have high per-iteration {computational} complexity.
In proximal second-order-type methods, the trade-off between iteration-complexity and per-iteration {computational} complexity is crucial to obtain a good performance. 
Some existing works such as \cite{Becker2012a,Hsieh2011,hsieh2013big,karimi2017imro,Tran-Dinh2013b,Tran-Dinh2013a} have provided evidence showing that second-order methods outperform first-order methods for {some important}
 subclasses of \eqref{eq:composite_cvx}.
{The} recent work \cite{karimireddy2018global} also studied {the} global linear convergence of Newton methods, but using a different concept {called} ``$c$-Hessian stable''.
Nevertheless, it is completely different from our approach.

%%%%%%%%%%%%%%%%%%%%%%%%%%%%%%
%%%% 1.5. Our approach.
%%%%%%%%%%%%%%%%%%%%%%%%%%%%%%
\paragraph{\textbf{Our approach}}
Our approach {here} relies on a combination of different ideas.
The first idea is the homotopy method, which has been used in interior-point methods \cite{Nesterov2004} and recently in path-following proximal Newton algorithms \cite{TranDinh2013e}, where the main iterations rely on a scaled proximal Newton scheme \cite{TranDinh2013e}.
The second idea is a new parameterization of the optimality condition of \eqref{eq:composite_cvx} as presented in Subsection \ref{subsec:re-parameterization}.
Our third idea is inspired by {the} generalized self-concordance concept introduced in \cite{SunTran2017gsc}.
The last one is a primal-dual-primal framework {that we have mentioned} above.

%%%%%%%%%%%%%%%%%%%%%%%%%%%%%%
%%%% 1.5. Our contribution.
%%%%%%%%%%%%%%%%%%%%%%%%%%%%%%
\paragraph{\textbf{Our contribution}}
Our contribution can be summarized as follows.
\begin{itemize}
\item[$\mathrm{(a)}$] 
We suggest a new parameterization for the optimality condition of \eqref{eq:composite_cvx} as a framework to study homotopy proximal variable-metric methods for solving different subclasses of \eqref{eq:composite_cvx}. This framework covers homotopy proximal-gradient, proximal quasi-Newton, and proximal-Newton methods, and their inexact variants as special cases.

\item[$\mathrm{(b)}$] 
We propose a homotopy proximal variable-metric scheme, Algorithm~\ref{alg:A1}, to solve \eqref{eq:composite_cvx} based on our new parameterization strategy.
We show that this scheme achieves a global linear convergence rate under the strong convexity and Lipschitz gradient assumptions w.r.t. a local norm, and the Lipschitz continuity of $g$.
We also propose an inexact homotopy proximal-Newton method to solve \eqref{eq:composite_cvx}.
Under the self-concordant {property} of $f$, and either the Lipschitz continuity of $g$ or the barrier property of $f$, our algorithm can also achieve a {finite} global iteration-complexity estimate.
With an appropriate choice of initial points or {suitable assumptions on $f$ and/or $g$}, 
%\fbox{{Please check that the last phrase is correct}} 
our method can achieve a linear convergence rate.

\item[$\mathrm{(c)}$] 
We propose a primal-dual-primal approach for a subclass of \eqref{eq:composite_cvx} where $f$ is self-concordant.
This approach {produces} a new homotopy primal-dual proximal-Newton algorithm which can also achieve a linear convergence rate under given assumptions.

\item[$\mathrm{(d)}$] 
We {specialize} our algorithm to solve a special case \eqref{eq:GL_prob} of \eqref{eq:composite_cvx} known as a regularized covariance estimation problems studied in \cite{Banerjee2008,Friedman2008,Hsieh2011,hsieh2013big,Kyrillidis2014,Tran-Dinh2013b,Tran-Dinh2013a}. 
This algorithmic variant possesses the following new features compared to existing works \cite{Hsieh2011,hsieh2013big,Tran-Dinh2013b,Tran-Dinh2013a}.
First, it {is applicable} to any regularizer $g$ instead of just the $\ell_1$-norm as in \cite{Hsieh2011,hsieh2013big}.
Second, it deals with the dual form of \eqref{eq:GL_prob}, while {allowing} one to reconstruct an approximate primal solution for \eqref{eq:GL_prob}.
Third, it does not require any Cholesky factorization or matrix inversion as in \cite{Tran-Dinh2013b} {by} working on the dual form. 
Fourth, the subproblem for computing proximal Newton directions is in the primal space of $X$
{which has some special structures as} promoted by the regularizer $g$ instead of in the dual space {where the dual variable has structures 
that are correspondingly} promoted by the conjugate $g^{\ast}$.
{The last point is important computationally} when $g$ promotes the sparsity or low-rankness of the solutions.
\end{itemize}

Let us emphasize the following aspects of our contribution.
Firstly, our new parameterization strategy can potentially be used to develop new numerical methods for different subclasses of \eqref{eq:composite_cvx} instead of only {the} four cases 
%\fbox{{Shouldn't it be 4 subclasses??}} 
studied in this paper.
Secondly, our path-following scheme for finding an appropriate initial point of Algorithm~\ref{alg:A1} is independent of a starting point as shown in Theorem~\ref{th:initial_point}.
Thirdly, even for a strongly convex and Lipschitz gradient function $f$, our homotopy scheme has advantages in sparse optimization as discussed in Section \ref{sec:homo_prox_grad_alg}.
Fourthly, \eqref{eq:composite_cvx} is different from the barrier formulation considered in \cite{TranDinh2013e}, where we do not use any penalty parameter for $f$ in \eqref{eq:composite_cvx} {as} compared to \cite{TranDinh2013e}.
In addition, \cite{TranDinh2013e} is aimed at solving constrained convex optimization problems where the barrier is induced from the feasible set.
Finally, for the covariance estimation problem~\eqref{eq:GL_prob}, our method shares some similarity with \cite{Hsieh2011,hsieh2013big,Tran-Dinh2013b,TranDinh2013e}, but it is still fundamentally different.
While \cite{Hsieh2011,hsieh2013big} {focused} on the sparse instance of \eqref{eq:GL_prob}, we consider a more general form in \eqref{eq:composite_cvx} that covers this example as a special case.
Our algorithm and its convergence guarantee are completely different and rely on a different approach compared to \cite{Hsieh2011,hsieh2013big}.
It has an iteration-complexity analysis {for a genera} $g$, while {the analysis in \cite{Hsieh2011,hsieh2013big} rely critically on} the special structure of the $\ell_1$-norm for $g$.

\paragraph{Paper organization}
In Section~\ref{sec:background}, we recall some preliminary results used in this paper.
Section~\ref{sec:methods} presents a new parameterization for the optimality condition of \eqref{eq:composite_cvx} and a conceptual three-stage proximal variable-metric framework, Algorithm~\ref{alg:A1}, for solving \eqref{eq:composite_cvx}.
Section~\ref{sec:three_classes_of_models} analyzes {the} convergence of Algorithm~\ref{alg:A1} under three sets of assumptions.
Section~\ref{subsec:initial_point} proposes some procedures to find an appropriate starting point for Algorithm~\ref{alg:A1}.
Section~\ref{sec:PDP_alg} proposes a primal-dual-primal method for solving a nonsmooth subclass of \eqref{eq:composite_cvx}, and its application to the covariance estimation problem \eqref{eq:GL_prob}.
Section~\ref{sec:num_exp} provides several numerical experiments to illustrate our theoretical results.
All technical proofs are deferred to the appendices.

%%%%%%%%%%%%%%%%%%%%%%%%%%%%%%%%%%%%%%%%%%
%%% 2. Preliminaries: self-concordance and scaled proximity
\section{Preliminaries: Scaled proximal operators and optimality condition}\label{sec:background}
In this section, we recall some basic concepts which will be used in the sequel.

%%% 2.1. Basic notation and concepts.
\subsection{Basic notation and concepts}
We work on the {vector space} $\R^p$  equipped with the standard inner product $\iprods{\cdot, \cdot}$ and the corresponding Euclidean norm $\norm{\cdot}_2$.
We use $\Spd^p_{++}$ to denote the set of all symmetric positive definite matrices in $\R^{p\times p}$.
For a given  $H \in\Spd^p_{++}$, we use $\norm{x}_H := \iprods{Hx, x}^{1/2}$ to denote the weighted norm.
The corresponding dual norm is $\norms{y}_H^{\ast} = \iprods{H^{-1}y, y}^{1/2}$.

For a subset $\Xc$, $\intx{\Xc}$ denotes the interior of $\Xc$, and $\partial{\Xc}$ denotes its boundary.
Let $f : \R^p\to\Rext$ be a convex function. As usual, $\dom{f}$ denotes the effective domain of $f$, and $\partial{f}$ denotes its subdifferential \cite{Rockafellar1970}.
If $f$ is twice differentiable, then $\nabla{f}$ and $\nabla^2{f}$ denote its gradient and Hessian, respectively.
For a given twice differentiable convex function $f$, if $x\in\dom{f}$ such that $\nabla^2{f}(x) \succ 0$, we define a local norm and its dual norm associated with $f$ as in \cite{Nesterov1994}:
\myeq{eq:local_norm}{
\norm{u}_x := \iprods{\nabla^2{f}(x)u, u}^{1/2} ~\text{and}~\norm{v}_{x}^{\ast} := \iprods{\nabla^2{f}(x)^{-1}v, v}^{1/2},
}
for any $u, v\in\R^p$. 
Clearly, $\iprods{u, v} \leq \norm{u}_x\norm{v}_x^{\ast}$.
This is the weighted norm with $H = \nabla^2{f}(x)$.
For a real number $a$, we use $\lfloor a\rfloor$ to denote the integer less than or equal to $a$. 
We use $[a]_{+} := \max\set{0, a}$ for any real number $a$.

Given a nonempty convex set $\Xc$ in $\R^p$, and a point $x\in\R^p$, the distance from $x$ to $\Xc$ corresponding to the weighted-norm $\norm{\cdot}_H$ is defined as $\dist{x,\Xc} := \inf_{y\in\Xc}\norm{x - y}_H$.
For a given convex function $f$, we say that $f$ is $\mu_f$-strongly convex if $f(\cdot) - \tfrac{\mu_f}{2}\norm{\cdot}_2^2$ remains convex, where $\mu_f > 0$ is called the strong convexity parameter of $f$.
We say that $f$ is $L_f$-smooth (i.e., Lipschitz gradient continuous) if $f$ is differentiable on $\dom{f}$ and $\nabla{f}$ is Lipschitz continuous with a Lipschitz constant $L_f \in [0, +\infty)$, i.e., $\norms{\nabla{f}(x) - \nabla{f}(y)}_2 \leq L_f\norms{x - y}_2$ for all $x, y\in\dom{f}$.
We denote the class of $\mu_f$-strongly convex and $L_f$-smooth functions by $\mathcal{F}_{L,\mu}^{1,1}$.
A convex function $g$ is Lipschitz continuous on {$\dom{g}$} with a Lipschitz constant $L_g \in [0, +\infty)$ if $\vert g(x) - g(y)\vert \leq L_g\Vert y - x\Vert_2$ for all $x,y\in\dom{g}$.

%%% 2.3. Scaled proximal operators
\subsection{Scaled proximal operators}\label{subsec:prox_oper}
Let  $g : \R^p\to\Rext$ be a proper, closed, and convex function, and $H \in \Spd^p_{++}$. 
We define the following \textit{scaled} proximal operator \cite{friedlander2016efficient} of $g$:
\myeq{eq:scaled_prox_oper2}{
\prox^{H}_{g}(x) := \argmin_{u\in\R^p}\big\{ g(u) + \tfrac{1}{2}\Vert u - x\Vert_{H}^2 \big\}.
}
The optimality condition of this minimization problem is $0 \in H(\prox^{H}_{g}(x) - x) + \partial{g}(\prox^{H}_{g}(x))$, which can be written as $x \in (\Id + H^{-1}\partial{g})(\prox^{H}_{g}(x))$, or $\prox^{H}_{g}(x) = (\Id + H^{-1}\partial{g})^{-1}(x)$.
When $H = \frac{1}{\gamma}\Id$, where $\gamma > 0$ and $\Id$ is  the identity matrix, $\prox^{H}_{g}(\cdot)$ becomes a classical proximal operator \cite{Bauschke2011,Parikh2013}, and is usually denoted by $\prox_{\gamma g}(\cdot)$.
{An important} property of $\prox_{g}^{H}$ is {its nonexpansiveness}
\myeq{eq:nonexpansiveness}{
\Vert \prox^{H}_{g}(x) - \prox^{H}_{g}(y)\Vert_{H} \leq \Vert x - y\Vert_{H},
}
for any $x, y$ in $\R^p$.
We say that $g$ is \textit{proximally tractable} if $\prox^{H}_{g}(\cdot)$ can be efficiently evaluated, e.g., in a closed form or by a low-order polynomial time algorithm (e.g., $\BigO{p\log(p)}$).
Computational methods for evaluating this scaled proximal operator and its classical forms can be easily found in the literature including \cite{friedlander2016efficient,Parikh2013}.

%%% 2.4. Lipschitz continuity w.r.t. local norm of g
\subsection{Lipschitz continuity w.r.t. local norm}
Let $\norms{\cdot}_x$ and its dual norm $\norm{\cdot}_x^{\ast}$ be 
defined by a strictly smooth convex function $f : \R^p\to\R$, and $g : \R^p\to\Rext$ be a proper, closed, and convex function. 

%% Definition 4.
\begin{definition}\label{de:f_lipschitz}
We say that $g$ is $L_g$-Lipschitz continuous w.r.t. $\norms{\cdot}_x$ with a Lipschitz constant $L_g \in [0, +\infty)$, if for any $x, y, z\in\dom{f}\cap\dom{g}$, we have $\vert g(y) - g(z)\vert \leq L_g\Vert y - z\Vert_x$.  
\end{definition}
As a concrete example, assuming that $f(x) = \frac{1}{2}x^{\top}Qx + q^{\top}x$ is a strongly convex quadratic function, then $g$ is Lipschitz continuous in $\ell_2$-norm if and only if $g$ is Lipschitz continuous w.r.t. the local norm defined by $f$.

%%% Lemma 5.
\begin{lemma}\label{le:bound_subgrad}
A proper, closed, and convex function $g$ is $L_g$-Lipschitz continuous w.r.t. $\norms{\cdot}_x$ with a Lipschitz constant $L_g$ on $\dom{f}\cap\dom{g}$ if and only if $\Vert\nabla{g}(y)\Vert^{\ast}_{x} \leq L_g$ for any $x, y\in\dom{f}\cap\dom{g}$ and $\nabla{g}(y)\in\partial{g}(y)$.

In particular, if $f$ is strongly convex with a strong convexity parameter $\mu_f > 0$ and $g$ is Lipschitz continuous in $\ell_2$-norm with a Lipschitz constant $\bar{L}_g \geq 0$ $($i.e. $\vert{g(y) - g(z)}\vert \leq \bar{L}_g\norms{y-z}_2$ for any $y, z\in\dom{f}\cap\dom{g}$$)$, then $g$ is $L_g$-Lipschitz continuous w.r.t. $\norms{\cdot}_x$ on $\dom{f}\cap\dom{g}$ with the Lipschitz constant $L_g := \frac{\bar{L}_g}{\sqrt{\mu_f}}$. 
However, the converse statement does not hold in general.
\end{lemma}

%% Proof of Lemma 5.
\begin{proof}
For any $x, y\in\dom{f}\cap\dom{g}$ and $\nabla{g}(y)\in\partial{g}(y)$, we have 
\myeqn{\begin{array}{ll}
\norms{\nabla{g}(y)}_x^{\ast} &= \max\set{\iprods{\nabla{g}(y), z - y} \mid \norms{z-y}_x \leq 1} \vspace{1ex}\\
&\leq \max\set{ \vert g(z) - g(y)\vert \mid \norm{z-y}_x \leq 1} \vspace{1ex}\\
& \leq L_g\max\set{\norm{z-y}_x \mid \norm{z-y}_x\leq 1} = L_g.
\end{array}}
Conversely, by convexity of $g$, we have $g(y) - g(z) \leq \iprods{\nabla{g}(y), y - z} \leq \norm{\nabla{g}(y)}_{x}^{\ast}\norms{y - z}_x \leq L_g\norms{y-z}_x$.
By exchanging $y$ and $z$, we finally get $\vert g(z) - g(y)\vert \leq L_g\norms{z-y}_x$.

If $f$ is strongly convex with a strong convexity parameter $\mu_f > 0$, then we have $\nabla^2{f}(x) \succeq \mu_f\Id$ for any $x\in\dom{f}$.
Therefore, we have $\norms{y-z}_2 \leq \frac{1}{\sqrt{\mu_f}}\norm{y-z}_x$.
This shows that $\vert g(y) - g(z)\vert \leq \bar{L}_g\norms{y-z}_2 \leq \frac{\bar{L}_g}{\sqrt{\mu_f}}\norm{y-z}_x$.
Hence, $g$ is $L_g$-Lipschitz continuous w.r.t. $\norms{\cdot}_x$ with $L_g := \frac{\bar{L}_g}{\sqrt{\mu_f}}$.
\end{proof}
%% End of the proof.

As an example, if $\dom{g}$ is contained in an affine subspace defined by $\Lc := \set{x \in\R^p \mid Ax = b}$, and $\nabla^2f(x)$ is uniformly positive definite on $\Lc$, then $g$ is $L_g$-Lipschitz and Lemma~\ref{le:bound_subgrad} still holds, but $f$ is still non-strongly convex.
In this case, we say that $f$ is restricted strongly convex.
Lemma~\ref{le:bound_subgrad} shows that the Lipschitz continuity w.r.t. the local norm $\norms{\cdot}_x$ of $f$ is weaker than the global Lipschitz continuity of $g$ since we only require the condition to hold on $\dom{f}\cap\dom{g}$.

%%% 2.4. Fundamental assumption, optimality condition and approximate solution
\subsection{Fundamental assumption and optimality condition}\label{subsec:opt_cond}
Throughout this paper, we rely on the following fundamental assumption:
\begin{assumption}\label{as:A1}
${\dom{F}} \!:=\! \dom{f}\cap\dom{g} \!\neq\!\emptyset$. 
The solution set $\Xc^{\star}$ of \eqref{eq:composite_cvx} is nonempty.
\end{assumption}
Assumption~\ref{as:A1} is {a standard one that is} required in any solution method.
Throughout this paper, we assume that Assumption~\ref{as:A1} holds without recalling it.

The optimality condition associated with \eqref{eq:composite_cvx} becomes
\myeq{eq:opt_cond}{
0 \in \nabla{f}(x^{\star}) +  \partial{g}(x^{\star}).
}
This condition is necessary and sufficient for $x^{\star}$ to be an optimal solution of \eqref{eq:composite_cvx}.
For any $H \in \Spd^p_{++}$, we can reformulate this optimality condition as a fixed-point condition:
\myeqn{
x^{\star} = \prox_{g}^{H}\left(x^{\star}  - H^{-1}\nabla{f}(x^{\star})\right).
}
This formulation shows that $x^{\star}$ is a fixed point of $\mathbb{T}_{g}^H(\cdot) := \prox_{g}^{H}(\cdot - H^{-1}\nabla{f}(\cdot))$.

%%% 3. A new inexact proximal Newton method for inverse covariance estimation.
\section{A Conceptual Homotopy Proximal Variable-Metric Framework}\label{sec:methods}
In this section, we introduce a novel parameterization of the optimality condition \eqref{eq:opt_cond} and propose a conceptual framework for designing homotopy proximal variable-metric methods for solving  \eqref{eq:composite_cvx}.

%%% 3.1. A novel parametrization of the optimality condition.
\subsection{Parametrization of the optimality condition}\label{subsec:re-parameterization}
Given $x^0\in\dom{F}$, we compute a subgradient $\xi^0\in\partial{g}(x^0)$.
Then, we parameterize $f$ as follows:
\myeq{eq:f_tau}{
f_{\tau}(x) :=  \tau f(x) - (1 - \tau)\iprods{\xi^0, x},
}
where $\tau \in [0, 1]$.
Clearly, $f_1(x) = f(x)$,  $\nabla{f_{\tau}}(x) = \tau \nabla{f}(x) - (1- \tau)\xi^0$, and $\nabla^2{f_{\tau}}(x) = \tau\nabla^2{f}(x)$.
In addition, $\dom{f_{\tau}} = \dom{f}$ for any $\tau \in (0, 1]$.
Note that if we can choose $\xi^0\in\partial{g}(x^0)$ such that $\xi^0 = \boldsymbol{0}^p$, then $f_{\tau}(x)$ in \eqref{eq:f_tau} reduces to $f_{\tau}(x) = \tau f(x)$.

Next, we consider the following composite convex optimization problem derived from \eqref{eq:composite_cvx}:
\myeq{eq:auxiliary_prob1}{
x_{\tau}^{\ast} = \argmin_{x\in\R^p}\Big\{ F_{\tau}(x) :=  f_{\tau}(x) + g(x) \Big\}.
}
This problem is similar to  \eqref{eq:composite_cvx} and can be considered as a parametric perturbation instance of \eqref{eq:composite_cvx}.
The optimality condition of this problem {is given by}
\myeq{eq:new_reparam_opt_cond}{
0 \in \nabla{f_{\tau}}(x^{\ast}_{\tau}) +  \partial{g}(x^{\ast}_{\tau}) \equiv \tau \nabla{f}(x^{\ast}_{\tau}) -  (1-\tau)\xi^0 + \partial{g}(x^{\ast}_{\tau}),
}
which is necessary and sufficient for $x^{\ast}_{\tau}$ to be an optimal solution of \eqref{eq:auxiliary_prob1}.
We call this condition a parametric optimality condition of \eqref{eq:composite_cvx}.

From the optimality condition \eqref{eq:new_reparam_opt_cond}, we can show that
\begin{itemize}
\item If $\tau =1$, then \eqref{eq:new_reparam_opt_cond} becomes $0 \in   \nabla{f}(x^{\ast}_1) + \partial{g}(x^{\ast}_1)$, which is exactly the original optimality condition \eqref{eq:opt_cond} of \eqref{eq:composite_cvx}.
Hence, $x^{\ast}_1 = x^{\star}$ is an exact optimal solution of \eqref{eq:composite_cvx}.
\item If $\tau = 0$, then \eqref{eq:new_reparam_opt_cond} reduces to $\xi^0 \in \partial{g}(x^{\ast}_0)$.
Therefore, we can choose $x^{\ast}_0 = x_0$, the initial point, as an optimal solution of 
{\eqref{eq:auxiliary_prob1} at $\tau = 0$.}
\end{itemize}
Our main idea is to start from a small value $\tau_0 \approx 0$ and follow a homotopy path on $\tau$ to {find an} approximate solution of $x^{\ast}_{\tau}$ at $\tau \approx 1$.
As we will show later, we do not start from $\tau_0 = 0$, but from a sufficiently small value $\tau_0 > 0$.

Note that, when $\tau > 0$, we can write \eqref{eq:new_reparam_opt_cond} as
\myeqn{
0 \in \nabla{f}(x^{\ast}_{\tau}) - \left(\tfrac{1}{\tau} - 1\right)\xi^0 + \tfrac{1}{\tau}\partial{g}(x^{\ast}_{\tau}).
}
If $g(\cdot) = \rho\norms{\cdot}_1$, an $\ell_1$-regularizer, for a given regularization parameter $\rho > 0$, then when $\tau$ is close to zero, the weight $\frac{1}{\tau}$ {on $g$ is large and the solution of \eqref{eq:auxiliary_prob1} is expected to be very} sparse.
This potentially {can reduce} the computational complexity of the underlying optimization method by working on sparse vectors or matrices.
This property of the regularizer $g$  is {also expected} in other applications such as low-rank and group sparsity models.

%%% Remark 1.
\begin{remark}\label{re:comparison}
The formulation \eqref{eq:new_reparam_opt_cond} is new and it does not reduce to any existing homotopy formulation including \cite[Formula 4.2.26]{Nesterov2004} to the best of our knowledge. 
This formulation {is expected to lead to more efficient homotopy-type algorithms for
solving} sparse and low-rank convex optimization as explained above.
\end{remark}

%%% 3.2. A fixed-point interpretation of the parametric optimality condition
\subsection{A fixed-point interpretation of the parametric optimality condition}

{Recall
 the optimality condition \eqref{eq:new_reparam_opt_cond} given as} $0 \in \nabla{f_{\tau}}(x^{\ast}_{\tau}) + \partial{g}(x^{\ast}_{\tau})$.
{By} using the scaled proximal operator $\prox_{g}^H$, we can {reformulate} \eqref{eq:new_reparam_opt_cond} into a fixed-point problem:
\myeq{eq:fixed_point1}{
x^{\ast}_{\tau} = \prox^{H}_{\frac{1}{\tau}g}\left( x^{\ast}_{\tau} - H^{-1}\big(\nabla{f}(x^{\ast}_{\tau}) - (\tfrac{1}{\tau}-1)\xi^0\big) \right),
}
for any $H\in\Spd^p_{++}$.
Let us define the following mapping for any $x\in\dom{F}$:
\myeq{eq:generalized_grad}{
G^H_{\tau}(x) = H\Big(x - \prox^{H}_{\tfrac{1}{\tau}g}\left( x - H^{-1}\big(\nabla{f}(x)  - (\tfrac{1}{\tau} - 1)\xi^0\big)\right)\Big).
}
Clearly, \eqref{eq:fixed_point1} is equivalent to $G_{\tau}^H(x^{\ast}_{\tau}) = 0$.
We call $G_{\tau}^H$ the scaled generalized gradient mapping of the parametric problem \eqref{eq:auxiliary_prob1}.
The most common case is $H = \frac{1}{\gamma}\Id$ as mentioned above for some $\gamma > 0$.
Then, $G_{\tau}^H$ reduces to the standard generalized gradient mapping \cite{Nesterov2004}.

%%%% 3. Conceptual framework of homotopy proximal variable-metric methods
\subsection{Conceptual framework of homotopy proximal variable-metric methods}
We first describe our conceptual three-stage proximal variable-metric algorithm as in Algorithm~\ref{alg:A1}.

%%%%%%%%%%%%%%%%%%%%%%%%%%%%%%%%%%%%%%%%%%
%%% + Algorithm 1 - Homotopy proximal variable-metric framework.
%%%%%%%%%%%%%%%%%%%%%%%%%%%%%%%%%%%%%%%%%%
\begin{algorithm}[H]\caption{(\textit{A Conceptual Three-Stage Proximal Variable-Metric Algorithm})}\label{alg:A1}
\begin{algorithmic}[1]
\State\textbf{Stage 1}~(\textbf{Find an initial point}): 
\State\hspace{0.2cm}\label{a1step:p0_step1} Choose $\tau_0 \in (0, 1)$, and an appropriate initial point $x^0\in\dom{F}$. Evaluate $\xi^0\in\partial{g}(x^0)$.
\State\textbf{Stage 2}~(\textbf{Homotopy scheme}): ~\textrm{For $k = 0$ to $k_{\max}$, perform}
\State\hspace{0.2cm}\label{a1step:p1_step2} Update $\tau_{k+1}$ from $\tau_k$ such that $0 < \tau_k < \tau_{k+1} \leq 1$. 
\State\hspace{0.2cm}\label{a1step:p1_step4} Evaluate $\nabla{f}(x^k)$ and $H_k$, and update $x^{k+1}$ by approximately solving
\myeq{eq:homotopy_method}{
x^{k+1}  :\approx \prox_{\frac{1}{\tau_{k+1}}g}^{H_k}\Big(x^k - H_k^{-1}\left(\nabla{f}(x^k) - \big(\tfrac{1}{\tau_{k+1}} - 1\big)\xi^0\right)\Big).
}
\State\hspace{0cm}\textbf{Stage 3}~(\textbf{Solution refinement}): Fix $\tau_{k}$ and perform \eqref{eq:homotopy_method} until a desired solution is achieved.
\end{algorithmic}
\end{algorithm}
%%% End of the algorithm.
%%%%%%%%%%%%%%%%%%%%%%%%%%%%%%%%%%%%%%%%%%

\noindent We will provide the details of each stage in the sequel based on 
{some appropriate assumptions for} \eqref{eq:composite_cvx}.
The main step of Algorithm~\ref{alg:A1} is \eqref{eq:homotopy_method}, where we need to evaluate the scaled proximal operator $\prox_{\frac{1}{\tau}g}^H(\cdot)$.
Depending on the choice of the variable matrix $H_k$, we obtain different methods:
\begin{itemize}
\item If $H_k$ is diagonal, then  we obtain a homotopy proximal gradient method.
\item If $H_k$ approximates $\nabla^2{f}(x^k)$, then we obtain a homotopy proximal quasi-Newton method.
\item If $H_k = \nabla^2{f}(x^k)$, then we obtain a homotopy proximal Newton method.
\end{itemize}
The choice of {an} initial point $x^0$, the initial value $\tau_0$ of the parameter $\tau$, the update rule of $\tau_k$, and the approximation rule of \eqref{eq:homotopy_method} in Algorithm~\ref{alg:A1} will be specified in the sequel.

%%% 3.4. Inexact proximal Newton scheme.
\subsection{Inexact proximal Newton scheme}
Let  $H_k = \nabla^2{f}(x^k)$.
The exact evaluation of the scaled proximal operator in \eqref{eq:homotopy_method} is equivalent to solving the following convex subproblem:
\myeq{eq:cvx_subprob_k}{
\bar{x}^{k+1} := \argmin_{x\in\R^p}\Big\{ \Pc_k(x) := \iprods{\nabla{f_{k\!+\!1}}(x^k), x - x^k} + \tfrac{1}{2}\iprods{\nabla^2{f}(x^k)(x - x^k), x - x^k} + \tfrac{1}{\tau_{k+1}}g(x) \Big\},{\!\!\!}
}
where $\nabla{f}_{k+1}(x^k) := f(x^k) - \big(\tfrac{1}{\tau_{k+1}}-1\big)\xi^0$.
In this case, we can write $\bar{x}^{k+1}$ as 
\myeqn{
\bar{x}^{k+1} = \prox_{\frac{1}{\tau_{k+1}}g}^{\nabla^2{f}(x^k)}\left(x^k - \nabla^2{f}(x^k)^{-1}\nabla{f_{k+1}}(x^k) \right), 
}
the exact solution of \eqref{eq:cvx_subprob_k}.

When $g$ is nontrivial (e.g., not a linear function), we can only approximate the true solution $\bar{x}^{k+1}$ of \eqref{eq:cvx_subprob_k} by an approximation $x^{k+1}$ {such that}
\myeq{eq:homotopy_method2}{
x^{k+1}  :\approx \prox_{\frac{1}{\tau_{k+1}}g}^{\nabla^2{f}(x^k)}\left(x^k -\nabla^2{f}(x^k)^{-1}\nabla{f_{k+1}}(x^k)\right).
}
Here the approximation ``$:\approx$'' is defined explicitly {next} in Definition~\ref{de:approx_prox}.

%%% Definition 3.1.
\begin{definition}\label{de:approx_prox}
Let $\bar{x}^{k+1}$ be the exact solution of \eqref{eq:cvx_subprob_k}, and $\delta_k \geq 0$ be a given accuracy.
We say that $x^{k+1}$ is a $\delta_k$-approximate solution to $\bar{x}^{k+1}$, denoted by $x^{k+1} :\approx \bar{x}^{k+1}$ as in \eqref{eq:homotopy_method2}, if
\myeq{eq:approx_prox}{
\Pc_k(x^{k+1}) - \Pc_k(\bar{x}^{k+1})  \leq \tfrac{\delta_k^2}{2}.
}
\end{definition}
Using this definition, we have the following result, see \cite[Lemma 3.2.]{TranDinh2013e}.
%%% Fact 1.
\begin{fact}\label{fact:fact_k} $\frac{1}{2}\norms{x^{k+1}  - \bar{x}^{k+1}}_{x^k}^2 \leq \Pc_k(x^{k+1}) - \Pc_k(\bar{x}^{k+1})$.
Consequently, combining this inequality and \eqref{eq:approx_prox}, we can show that if \eqref{eq:approx_prox} holds, then
\myeq{eq:approx_prox2}{
\delta(x^k) := \norms{x^{k+1} - \bar{x}^{k+1}}_{x^k} \leq \delta_k.
}
\end{fact}
The condition \eqref{eq:approx_prox} can be guaranteed by using several optimization methods in the literature such as accelerated proximal-gradient \cite{Beck2009,Nesterov2004,Schmidt2011}, ADMM \cite{Boyd2011}, or semi-smooth Newton-CG augmented Lagrangian methods \cite{yang2015sdpnalp}.
In Subsection~\ref{subsec:dual_approach_for_cvx_subprob3}, we approximately compute \eqref{eq:cvx_subprob_k} {via solving its dual.}

We define the following local distances to measure the distance {of approximations $x^{k+1}$ and ${x}^k$} to the true solution $x^{\ast}_{\tau_{k+1}}$ of the parameterized problem \eqref{eq:new_reparam_opt_cond}:
\myeq{eq:local_distances22}{
\lambda_{k+1} := \Vert x^{k+1} - x^{\ast}_{\tau_{k+1}}\Vert_{x^{\ast}_{\tau_{k+1}}}~~~\text{and}~~~\hat{\lambda}_{k} := \Vert x^{k} - x^{\ast}_{\tau_{k+1}}\Vert_{x^{\ast}_{\tau_{k+1}}}.
}
These metrics will be used to analyze the convergence of our methods.
%In the next two sections, we consider three cases: $f\in\Fc^{1,1}_{L,\mu}$ (i.e. strong convexity and smoothness), $f$ is self-concordant, and $f$ is self-concordant barrier.

%%%% 3. Convergence and iteration-complexity analysis.
\section{Convergence and iteration-complexity analysis}\label{sec:three_classes_of_models}
We analyze the convergence and iteration-complexity of Algorithm~\ref{alg:A1} for solving \eqref{eq:composite_cvx} under three different {subclasses of $f$ and $g$.}

%%% Linear convergence under local strong convexity and Lipschitz gradient assumption
\subsection{Linear convergence for the smooth and strongly convex case}\label{sec:homo_prox_grad_alg} 
The first class of models is when $f$ and $g$ in \eqref{eq:composite_cvx} satisfies the following assumption.

%%% Assumption A3.
\begin{assumption}\label{as:A3}
Assume that $f$ is $\mu_f$-strongly convex and $L_f$-smooth.
The function $g$ is $L_g$-Lipschitz continuous on $\dom{g}$.
\end{assumption}
%% End of Assumption A.3.
Under Assumption~\ref{as:A3}, Algorithm~\ref{alg:A1} only has \textbf{Stage~2} and we  do not need to perform \textbf{Stage~1} and \textbf{Stage~3} of Algorithm~\ref{alg:A1}.
We can start from any starting point $x^0\in\dom{F}$.
We show that Algorithm~\ref{alg:A1} achieves a global linear convergence rate.
%under an additional assumption that $g$ is $L_g$-Lipschitz continuous on $\dom{g}$.
This result is stated in the following theorem, whose proof can be found in Apppendix~\ref{subsec:linear_convergence1}.
%Note that if $f$ satisfies Assumption~\ref{as:A3}, then the standard proximal Newton also achieves a linear convergence rate.

%%% Theorem 4.
\begin{theorem}\label{th:linear_convergence1}
Under Assumption~\ref{as:A3}, let $m$ and $L$ be two constants such that $0 < m \leq L <+\infty$ and $\omega := \frac{1}{m}\sqrt{(L-2\mu_f)m + L_f^2} < 1$.
For any given $\tau_0 \in (0, 1)$ and $x^0\in\dom{F}$, we define
\myeq{eq:quantities1}{
C := \frac{\norms{\nabla{f}(x^0) + \xi^0}_2}{\mu_f}~~\text{and}~~~\sigma := \frac{1-\tau_0 + \tau_0\omega\Gamma}{1-\tau_0 + \tau_0\Gamma} \in (\omega, 1),~~\text{where}~~\Gamma := \frac{\norms{\nabla{f}(x^0) + \xi^0}_2}{\omega(L_g + \norms{\xi^0}_2)}.
}
Let $\set{(x^k, \tau_k)}$ be the sequence generated by the exact scheme \eqref{eq:homotopy_method} in Algorithm~\ref{alg:A1}, where $H_k \in \Spd^p_{++}$ is chosen such that $m\Id \preceq H_k \preceq L\Id$ and $\tau_k$ is updated by
\myeq{eq:sigma_k}{
\tau_k := 1 -  \frac{(1-\tau_0)\sigma^k}{\tau_0 + (1-\tau_0)\sigma^k}.
}
Then, for any $k\geq 0$, we have $\norms{x^k - x^{\ast}_{\tau_k}}_2 \leq C \sigma^k$ and $0 < 1-\tau_k \leq \frac{(1-\tau_0)}{\tau_0}\sigma^k$.
Consequently, both sequences $\set{ \norms{x^k - x^{\ast}_{\tau_k}}_2}$ and $\set{1 - \tau_k}$ globally converge to zero at a linear rate.

The sequence $\set{x^k}$ also satisfies $\norms{x^k - x^{\star}}_2 \leq \hat{C}\sigma^k$, where $\hat{C} := C +  \frac{(1-\tau_0)\omega(L_g + \norms{\xi^0}_2)}{\tau_0^3\mu_f}$.
Hence, $\set{x^k}$ globally converges to { a solution $x^{\star}$} of \eqref{eq:composite_cvx} at a linear rate.
\end{theorem}

Let us make {some remarks} on the result of Theorem~\ref{th:linear_convergence1}.
First, the condition $\omega < 1$ is equivalent to $(L - 2\mu_f) m + L_f^2 < m^2$. 
If we choose $m = L > 0$, then we have $L_f^2 < 2\mu_fL$ which leads to $L > \tfrac{L_f^2}{2\mu_f}$.
In this case, if we define $\gamma := \frac{2}{L + m} = \frac{1}{L}$, then \eqref{eq:homotopy_method} becomes
\myeqn{
x^{k+1} := \prox_{\frac{\gamma}{\tau_{k+1}}g}\left(x^k - \gamma\big(\nabla{f}(x^k) - (\tfrac{1}{\tau_{k+1}} - 1)\xi^0\big)\right),
}
which reduces to a  homotopy proximal gradient method.

To optimize the contraction factor, we need to minimize $1 - 2\mu_ft + L_f^2t^2$ over $t$.
This gives us $t = \frac{\mu_f}{L_f^2}$ showing  that $m = L = \frac{L_f^2}{\mu_f}$.
Hence, we must choose $H_k = \tfrac{L_f^2}{\mu_f}\Id$, and we obtain $\omega = 1- \tfrac{\mu_f}{L_f}$.
%= 1- \frac{1}{\kappa_f}
Another simple choice of $H_k$ is $H_k = \big(\frac{L + m}{2}\big)\Id$. 

Next, note that the convergence rate of $\set{\norms{x^k - x^{\star}}_2}$ in \eqref{eq:homotopy_method} is slower than in the standard proximal variable-metric method.
Its contraction factor is $\sigma$ defined in \eqref{eq:quantities1}.
However, \eqref{eq:homotopy_method}  possesses some {computational advantages  that} the standard proximal variable-metric method does not have as we will discuss in Section \ref{sec:num_exp}.

The linear convergence rate under Assumption~\ref{as:A3} is known from the literature for both gradient and Newton-type methods.
Nevertheless, our method is  new, which works on the {parameterized} function $f_{\tau}$ instead of $f$.
Our method also allows us to flexibly choose the variable matrix $H_k$ as long as it satisfies the condition of Theorem~\ref{th:linear_convergence1}.
Another appropriate choice of $H_k$ is a rank-one update as proposed in \cite{Becker2012a}.

\begin{remark}[Nesterov's accelerated variant]
Note that we can develop Nesterov's accelerated variant for \eqref{eq:homotopy_method} under Assumption \ref{as:A3}.
In this case, the convergence factor in Theorem \ref{th:linear_convergence1} will be improved from $1 - \frac{\mu_f}{L_f}$ to $1 - \sqrt{\frac{\mu_f}{L_f}}$.
However, we skip this modification in this paper.
\end{remark}

%%% 4.2. Linear convergence for self-concordant function $f$ without barrier.
\subsection{Linear convergence for self-concordant function $f$ without barrier}\label{subsec:model2}
We consider the second case where $f$ {and g satisfy the following assumption.} 
%Our convergence guarantee of the inexact homotopy proximal-Newton scheme \eqref{eq:homotopy_method2} requires the following assumption.

%%% Assumption 4.
\begin{assumption}\label{as:A4}
The function $f$ in \eqref{eq:composite_cvx} is standard self-concordant  $($see Definition~\ref{de:gen_sel_con_def}$)$.
The function $g$ is $L_g$-Lipschitz continuous w.r.t. the local norm $\norms{\cdot}_x$ defined by $f$ with a Lipschitz constant $L_g\in [0, +\infty)$ in $\dom{f}\cap\dom{g}$.
\end{assumption}
Note that, in Assumption \ref{as:A4}, we only require $f$ to be self-concordant,  {but}
\textbf{not necessary a self-concordant barrier}.
The class of self-concordant functions is much larger than the class of self-concordant barriers. As indicated in Proposition~\ref{pro:self_con_class}, any generalized self-concordant and strongly convex function is self-concordant.
In particular, it covers  a few representative applications presented in the introduction.
For other examples, we refer the reader to \cite{SunTran2017gsc,Tran-Dinh2013a}.

Under Assumption~\ref{as:A4}, Algorithm~\ref{alg:A1} has two stages: \textbf{Stage~1} finds an initial point, and \textbf{Stage~2} performs a homotopy scheme.
We can skip \textbf{Stage~3}.
For any initial value $\tau_0 \in (0, 1)$ of $\tau$, let us choose $\sigma \in (0.318642, 1)$ that solves the following inequation:
\myeq{eq:param_cond2}{
\frac{2L_g(1-\tau_0)(1 - \sqrt{\sigma})}{\tau_0 - 2L_g(1-\tau_0)(1-\sqrt{\sigma})} \leq \frac{\sqrt{\sigma - 0.01}}{10} - \frac{1}{18}. % ~~\text{and}~~\sigma  \geq \left(1 - \frac{0.0521\tau_0}{2L_g(1-\tau_0)}\right)^2 .
}
Note that there always exists $\sigma \in (0.318642, 1]$ that solves \eqref{eq:param_cond2}.
Theorem \ref{th:linear_convergence2b} below states the convergence of \eqref{eq:homotopy_method2} under the {Assumption \ref{as:A4}. Its}
%self-concordance of $f$, and the $L_g$-Lipschitz continuity of $g$ with respect to a local norm $\norm{\cdot}_x$, 
 proof can be found in Appendix \ref{apdx:th:linear_convergence2b}.

%%% Theorem 5b.
\begin{theorem}\label{th:linear_convergence2b}
Suppose that Assumption~\ref{as:A4} holds for \eqref{eq:composite_cvx}.
Let  $\tau_0\in (0, 1)$ and $\sigma \in (0, 1]$  be two constants satisfying \eqref{eq:param_cond2} and $\set{(x^k, \tau_k)}$ be the sequence generated by \eqref{eq:homotopy_method2}.
Moreover, $x^0$ and $\tau_0\in(0, 1)$ are chosen such that $\lambda_0 := \Vert x^0 - x^{\ast}_{\tau_0}\Vert_{x^{\ast}_{\tau_0}} \leq \beta$ with $\beta := 0.05$.
Let us choose $0 \leq \delta_k \leq \frac{\lambda_k}{113}$, and update the parameter $\tau_k$ as
\myeq{eq:sigma_choice_2b}{
\tau_{k+1} :=
{\left[ 1+  \frac{  \Delta_k\tau_k}{2L_g(1 + \Delta_k) 
-  \Delta_k\tau_k}\right] \tau_k}
\quad \text{where} \quad \Delta_k := \left( \frac{1}{10}\sqrt{\sigma - 0.01} - \frac{1}{18}\sqrt{\sigma}^k \right)\sqrt{\sigma}^k.
} 
Then, $\norms{x^k - x^{\ast}_{\tau_k}}_{x^{\ast}_{\tau_k}} \leq \beta\sigma^k$ for $k\geq 0$ and $0 < 1 - \tau_k  \leq \frac{(1-\tau_0)}{\tau_0}\sqrt{\sigma}^k$.
Therefore,  the sequences  $\{\norms{x^k - x^{\ast}_{\tau_k}}_{x^{\ast}_{\tau_k}}\}$ and $\set{1 - \tau_k}$ both globally {converge to zero at a} linear rate.
%\fbox{{Can change $\bar{\Delta}_k$ to $\Delta_k$?}}

Moreover, there exists $\hat{C} > 0$ such that $\Vert x^k - x^{\star}\Vert_{x^{\star}} \leq \hat{C}\sqrt{\sigma}^k$ for all $k\geq 0$. 
Hence, the sequence $\set{\Vert x^k - x^{\star}\Vert_{x^{\star}}}$ also globally converges to 
{zero} at a linear rate.
\end{theorem}

The following result is a direct consequence of Theorem \ref{th:linear_convergence2b} and Lemma~\ref{le:bound_subgrad} when $g$ is $\bar{L}_g$-Lipschitz continuous in $\ell_2$-norm, and $f$ is strongly convex and generalized self-concordant.

\begin{corollary}
%\fbox{{This part should be stated separately in a remark or as a corollary.}}
Assume that $f$ is generalized self-concordant as defined in Definition~\ref{de:gen_sel_con_def} and $g$ is Lipschitz continuous with a Lipschitz constant $\bar{L}_g \geq 0$ in $\ell_2$-norm instead of $g$ being $L_g$-Lipschitz continuous w.r.t. $\norm{\cdot}_x$.
Assume additionally that $f$ is strongly convex with a strong convexity parameter $\mu_f > 0$.
Then, the conclusion of Theorem~\ref{th:linear_convergence2b} still holds with $L_g := \frac{\bar{L}_g}{\sqrt{\mu_f}}$.
\end{corollary}

\begin{remark}\label{re:weaker_condition_on_g}
\normalfont
(a) 
The $L_g$-Lipschitz continuity of $g$ w.r.t. a local norm $\norm{\cdot}_x$ in Assumption \ref{as:A4} can be replaced by assuming that $\norms{\nabla{g}(x)}_{x}^{\ast} \leq L_g$ for some $\nabla{g}(x)\in\partial{g}(x)$ for any $x\in\dom{f}$ and $\xi^0 = \boldsymbol{0}^p\in\partial{g}(x^0)$.
By Lemma \ref{le:bound_subgrad}, we can easily see that  this condition is weaker than the $L_g$-Lipschitz continuity of $g$ w.r.t. $\norm{\cdot}_x$.
For example, if $f(x) = -\ln(x)$, and $g(x) = x^2$, then $\norms{\nabla{g}(x)}_{x}^2 = 1$ for all $x > 0$.
In this case, the conclusions of Theorem~\ref{th:linear_convergence2b} still hold.

\vspace{1ex}
(b) {Observe that in \eqref{eq:sigma_choice_2b}, if the rate of change
from $\Delta_k$ to $\Delta_{k+1}$ is slow so that $\Delta_{k+1} \approx 
\Delta_k$, then the rate of increment from $\tau_k$ to $\tau_{k+1}$ will 
become faster when $k$ increases. 
}
\end{remark}

%%% 4.3. Linear convergence under the self-concordant barrier of $f$
\subsection{Linear convergence under the self-concordant barrier of $f$}
When $f$ is a self-concordant barrier, we use {a} different analysis, and no longer require $g$ to be Lipschitz continuous as stated in the following assumption:
\begin{assumption}\label{as:A5}
The function $f$ is a $\nu_f$-self-concordant barrier as defined in Definition~\ref{de:gen_sel_con_def}, and $g$ is proper, closed, and convex.
For a given $x^0\in\dom{F}$, either the analytic center $x^{\star}_f$ of $f$ defined by \eqref{eq:analytical_center} on the interior of the {level} set $\Lc_{F}(x^0) := \set{x\in{\dom{F} \mid F(x)} \leq F(x^0)}$ exists or $\xi^0 = \boldsymbol{0}^p\in\partial{g}(x^0)$.
\end{assumption}
Under Assumption~\ref{as:A5}, Algorithm~\ref{alg:A1} also requires \textbf{Stage 1} and \textbf{Stage 2}, while we can skip \textbf{Stage 3}.
For any $\tau_0 \in (0, 1)$, we choose $\sigma\in  (0.318642, 1]$ such that
\myeq{eq:the_choice_of_sigma}{
%\hat{C}_0 := \sqrt{\frac{C}{10}} - \frac{10C}{9} > 0 ~~\text{and}~~
C_0 := \frac{\sqrt{\sigma - 0.01}}{10} - \frac{1}{18} > 0~~\text{and}~~\sigma \geq \left(1 - \frac{\tau_0C_0}{(1-\tau_0)(1 + C_0)(\sqrt{\nu_f} + \bar{c}_0)}\right)^2.
}
Here, $\bar{c}_0 := \theta_f\norms{\xi^0}_{x^{\star}_f}$ ($\theta_f$ is defined in Appendix~\ref{sec:apdx:proofs}) 
%\fbox{{$\theta_f$ should be $\nu_f$???}}
if $x^{\star}_f$ defined by \eqref{eq:analytical_center} exists, and $\bar{c}_0 := 0$, otherwise.
These two constants $\tau_0$ and $C_0$ always exist, and $C_0 \leq \frac{9}{20\sqrt{10}} \approx 0.14230$.
The following theorem states the convergence of 
{Algorithm 1}
%\eqref{eq:homotopy_method2} 
under the self-concordant barrier {assumption on} $f$, whose proof can be found in Appendix \ref{apdx:th:linear_convergence2}.

%%% Theorem 5.
\begin{theorem}\label{th:linear_convergence2}
Let us choose $\sigma \in (0.318642, 1]$ such that   \eqref{eq:the_choice_of_sigma} holds.
Let $\set{(x^k, \tau_k)}$ be the sequence generated by Algorithm~\ref{alg:A1} using \eqref{eq:homotopy_method2} with $\tau_0\in(0, 1)$ such that $\Vert x^0 - x^{\ast}_{\tau_0}\Vert_{x^{\ast}_{\tau_0}} \leq \beta$ with $\beta := 0.05$.
Let us choose $0 \leq \delta_k \leq \frac{\lambda_k}{113}$, and update $\tau_k$ as
\myeq{eq:sigma_choice}{
\tau_{k+1} := \left(1 + \frac{\Delta_k}{(1+ \Delta_k)(\sqrt{\nu_f} + \bar{c}_0)} \right)\tau_k,~~~\text{with}~~\Delta_k := \left( \frac{\sqrt{\sigma - 0.01}}{10} - \frac{\sqrt{\sigma}^k}{18} \right)\sqrt{\sigma}^k,
}
Then, $\norms{x^k - x^{\ast}_{\tau_k}}_{x^{\ast}_{\tau_k}} \leq \beta\sigma^k$ and $1 - \tau_k \leq \frac{C_0\sqrt{\sigma}^k}{(1 ~+~ C_0)(\sqrt{\nu_f} + \bar{c}_0)(1-\sqrt{\sigma})}$.
Therefore, the sequences $\{\norms{x^k - x^{\ast}_{\tau_k}}_{x^{\ast}_{\tau_k}}\}$ and $\set{1 - \tau_k}$ both globally {converge} to zero at a linear rate.

Moreover, if we choose $\sigma \in (0.318642, 1]$ such that  \eqref{eq:the_choice_of_sigma} holds and $\widetilde{C}_1 := \frac{C_0}{\tau_0(1 + C_0)(1-\sqrt{\sigma}) - C_0} > 0$ $($always exist$)$, then $\Vert x^k - x^{\star}\Vert_{x^{\star}} \leq \frac{\beta\sigma^k}{1 - \widetilde{C}_1\sqrt{\sigma}^k} + \widetilde{C}_1\sqrt{\sigma}^k$. Hence, the sequence $\set{x^k}$  converges to an optimal solution $x^{\star}$ of \eqref{eq:composite_cvx} at a linear rate.
\end{theorem}

Theorem \ref{th:linear_convergence2} shows {the} global linear convergence of our inexact proximal-Newton method for solving \eqref{eq:composite_cvx} under Assumption \ref{as:A5}.
Here, the constants $\sigma$ and $\beta$ {balance} between the contraction factor and the step-size of the homotopy parameter $\tau$.
The choice of $\sigma$ from \eqref{eq:the_choice_of_sigma} is conservative due to several rough estimates {in} our proof.
In practice, $\sigma$ can be {chosen to be} much smaller than one as observed in our numerical experiments.

%%% 5.5. Finding an initial point
\section{Stage 1: Finding an appropriate initial point}\label{subsec:initial_point}
While the variant of Algorithm~\ref{alg:A1} in Theorem~\ref{th:linear_convergence1} can start from any initial point $x^0\in\dom{F}$, the variants in  Theorem~\ref{th:linear_convergence2b} and Theorem~\ref{th:linear_convergence2} require an appropriate initial point $x^0$.
More precisely, we need to choose an initial value $\tau_0 \in (0, 1)$ and $x^0$ such that $\norms{x^0 - x^{\ast}_{\tau_0}}_{x^{\ast}_{\tau_0}} \leq \beta$ for a given $\beta := 0.05$.
We consider two cases: $g$ is $\mu_g$-strongly convex and $g$ is non-strongly convex.

\subsection{Inexact damped-step proximal-Newton scheme}
We can apply the following inexact damped-step proximal-Newton scheme proposed in \cite{TranDinh2013e} to find $x^0$.
Let us start from any initial point $\hat{x}^0\in\dom{F}$, compute {a} subgradient $\hat{\xi}^0\in\partial{g}(\hat{x}^0)$, and update:
\myeq{eq:damped_step_PN}{
\left\{\begin{array}{ll}
\hat{s}^{j+1} &:\approx \prox_{\frac{1}{\tau_0}g}^{\nabla^2{f}{(\hat{x}^j)}}\Big(\hat{x}^j - \nabla^2{f}(\hat{x}^j)^{-1}\big(\nabla{f}(\hat{x}^j) - (\frac{1}{\tau_0} - 1)\hat{\xi}^0\big)\Big) \\%\vspace{0.75ex}\\
\hat{x}^{k+1} &:= (1 - \alpha_j)\hat{x}^j + \alpha_j\hat{s}^{j+1},~~\text{with}~~\hat{\zeta}_j := \norms{\hat{s}^{j+1} - \hat{x}^j}_{\hat{x}^j} ~\text{and}~\alpha_j := \frac{\hat{\zeta}_j-\hat{\delta}_j}{(1 + \hat{\zeta}_j - \hat{\delta}_j)\hat{\zeta}_j}.
\end{array}\right.
}
Here, $0 \leq \hat{\delta}_j < \hat{\zeta}_j$ is the accuracy level defined as in Definition~\ref{de:approx_prox} %{with $\hat{\delta}_j = \|\hat{x}_j - x^*_{\tau_0}\|_{x^*_{\tau_0}}$ \fbox{Is this correct??}}, 
and we use {the ``hat'' notation
 for the iterates} to distinguish this procedure from Algorithm~\ref{alg:A1}.
The following proposition provides an estimation {on the number} 
of iterations {needed} to find the initial point $x^0$, whose proof can be found in \cite[Lemma 4.3.]{TranDinh2013e}.

%%% Proposition 7.
\begin{proposition}\label{pro:damped_step_PN}
Let $\set{\hat{x}^j}$ be  generated by \eqref{eq:damped_step_PN} with  $\hat{\delta}_j := \frac{\hat{\zeta}_j}{10}$, then after at most 
\myeqn{
\left\lfloor \frac{F_{\tau_0}(\hat{x}^0) - F_{\tau_0}(x^{\ast}_{\tau_0})}{\omega(0.9\beta)} \right\rfloor
}
iterations, we obtain $\hat{x}^{j_{\max}}$ such that $\norms{\hat{x}^{j_{\max}} - x^{\ast}_{\tau_0}}_{x^{\ast}_{\tau_0}} \leq \beta$, where $F_{\tau_0}(x) := f(x) - (\frac{1}{\tau_0} - 1)\iprods{\hat{\xi}^0, x} + \tfrac{1}{\tau_0}g(x)$, and $\omega(t) := t - \ln(1 + t)$.
\end{proposition}
%% End of Proposition 7.
Proposition~\ref{pro:damped_step_PN} suggests that we can perform a finite number of damped-step {proximal-Newton} scheme \eqref{eq:damped_step_PN} to find $x^0 := \hat{x}^{j_{\max}}$ such that $\norms{x^0 - x^{\ast}_{\tau_0}}_{x^{\ast}_{\tau_0}} \leq \beta$. Hence, $x^0$ is an initial point that satisfies the conditions of Theorem~\ref{th:linear_convergence2b} and Theorem~\ref{th:linear_convergence2}.

\subsection{Strong convexity of $g$}
If $g$ is strongly convex with a strong convexity parameter $\mu_g > 0$, and $x^0$ is not an optimal solution of \eqref{eq:composite_cvx}, then we can choose $\tau_0$ as
\myeq{eq:choice_of_tau0}{
0 < \tau_0 \leq \frac{\beta\mu_g}{(1+\beta)\lambda_{\max}(\nabla^2{f}(x^0))^{1/2}\norms{\nabla{f}(x^0) + \xi^0}_2}.
}
Here, $\lambda_{\max}(\nabla^2{f}(x^0))$ is the maximum eigenvalue of $\nabla^2{f}(x^0)$.
We will show in Appendix \ref{apdx:eq:choice_of_tau0} that $x^0$ satisfies $\norms{x^0 - x^{\ast}_{\tau_0}}_{x^{\ast}_{\tau_0}} \leq \beta$.
Hence, Algorithm~\ref{alg:A1} can start from an arbitrary point $x^0\in\dom{F}$.

%%%%%%%%%%%%%%%%%
\subsection{Non-strong convexity of $g$ and strong convexity of $f$}
We adopt our recent idea in \cite{TranDinh2016c} to develop a homotopy scheme to find this initial point {$x^0$} in a finite number of iterations. 

Starting from any $\hat{x}^0\in\dom{F}$, we consider the following auxiliary optimality condition depending on a new homotopy parameter $t > 0$ and a fixed value $\tau_0 \in (0, 1)$:
\myeq{eq:opt_cond_auxi}{
0 \in \nabla{f}(x^{\ast}_t) - \big(\tfrac{1}{\tau_0} -1\big)\hat{\xi}^0 - t(\nabla{f}(\hat{x}^0) + \hat{\xi}^0) + \tfrac{1}{\tau_0}\partial{g}(x^{\ast}_t),
}
for any $\hat{\xi}^0\in\partial{g}(\hat{x}^0)$.
Clearly, when $t = 0$, \eqref{eq:opt_cond_auxi} reduces to \eqref{eq:new_reparam_opt_cond} at $\tau = \tau_0$ and $x^0 \equiv \hat{x}^0$.
When $t = 1$, \eqref{eq:opt_cond_auxi} becomes $0 \in \nabla{f}(x^{\ast}_0) - \nabla{f}(\hat{x}^0) + \frac{1}{\tau_0}(\partial{g}(x^{\ast}_0) - \hat{\xi}^0)$, which shows that ${x^*_0 = } \hat{x}^0$ is a solution of \eqref{eq:opt_cond_auxi}.
By applying the homotopy method starting from $t_0 \approx 1$, and decreases $t_j$ to zero, we obtain an approximation $\hat{x}^j$ to $x^{\ast}_{\tau_0}$.
The main step of this scheme is given as follows:
\myeq{eq:homotopy_scheme0}{
\hat{x}^{j+1} :\approx \prox_{\frac{1}{\tau_0}g}^{\nabla^2{f}(\hat{x}^j)}\Big(\hat{x}^j - \nabla^2{f}(\hat{x}^j)^{-1}\big(\nabla{f}(\hat{x}^j) - (\tfrac{1}{\tau_0} - 1)\hat{\xi}^0 -  t_{j+1}(\nabla{f}(\hat{x}^0) + \hat{\xi}^0)\big)\Big),
}
where the approximation ``$:\approx$'' is defined as in Definition~\ref{de:approx_prox}, and $t_0 > 0$ is a starting value of $t$.
We also use the {``hat'' notation  for the iterates} to distinguish this procedure from Algorithm~\ref{alg:A1}.

This scheme is slightly different from \eqref{eq:homotopy_method2} {with}
the additional term $-t_{j+1}(\nabla{f}(\hat{x}^0) + \hat{\xi}^0)$.
The following theorem shows us how to choose $t_0$ and update $t$ to guarantee $\norms{\hat{x}^j - x^{\ast}_{\tau_0}}_{x^{\ast}_{\tau_0}} \leq \beta$, whose proof is given in Appendix~\ref{apdx:th:initial_point}.

%%% Theorem 7.
\begin{theorem}\label{th:initial_point}
Assume that $f$ is self-concordant and  $\mu_f$-strongly convex with $\mu_f > 0$.
For any given $\beta \in (0, 0.05]$, we defined $\Theta := \sqrt{\frac{99\beta}{500}} -\frac{10\beta}{9}$ {$> 0.$}
Let $\hat{x}^0\in\dom{F}$ be an arbitrary starting point, $\hat{\xi}^0 \in\partial{g}(\hat{x}^0)$, and $t_0$ be chosen such that 
\myeq{eq:t0_choice}{
t_0 := \begin{cases}
1 - \frac{\beta}{(1+2\beta)\norms{\nabla{f}(\hat{x}^0) + \hat{\xi}^0}_{\hat{x}^0}^{\ast}} &\text{if}~\norms{\nabla{f}(\hat{x}^0) + \hat{\xi}^0}_{\hat{x}^0}^{\ast} > \frac{1+2\beta}{\beta},\\
1 & \text{otherwise}.
\end{cases}
}
Let $\set{(\hat{x}^j, t_j)}$ be the sequence generated by \eqref{eq:homotopy_scheme0} starting from this $\hat{x}^0$ and $t_0$.
Support further that $t_j$ is updated {by} $t_{j+1} := \big[ t_j  - \frac{\Theta}{L_g(1+\Theta)}\big]_{+}$ and $\delta_j$  {satisfies} $\hat{\delta}_j \leq \frac{\lambda_j}{113}$. 
Then, after at most $j_{\max} := \left\lfloor \frac{t_0M_0(1 + \Theta)}{\Theta}\right\rfloor$ iterations with $M_0 := \frac{\norms{\nabla{f}(\hat{x}^0) + \hat{\xi}^0}_2}{\sqrt{\mu_f}}$, we have $t_{j_{\max}} = 0$, and $\norms{\hat{x}^{j_{\max}} - x^{\ast}_{\tau_0}}_{x^{\ast}_{\tau_0}} \leq \beta$.
\end{theorem}
%% End of Theorem 7.
Theorem~\ref{th:initial_point} shows that to find an initial point $x^0 := \hat{x}^{j_{\max}}$ for Algorithm~\ref{alg:A1} such that $\norms{x^0 - x^{\ast}_{\tau_0}}_{x^{\ast}_{\tau_0}} \leq \beta$, we only need a finite number of iterations $j_{\max}$ as defined in Theorem~\ref{th:initial_point}. Moreover, in this case, we can take $\xi^0 := \hat{\xi}^0$ in Algorithm~\ref{alg:A1}.

%%%% 5.4. Implementation remarks of Algorithm 1.
\subsection{Implementation remarks for Algorithm~\ref{alg:A1}}\label{subsec:practical_guide}
Theoretically, the variants of Algorithm~\ref{alg:A1} stated in Theorem~\ref{th:linear_convergence2b} and Theorem~\ref{th:linear_convergence2} require a good starting point $x^0$ such that $\Vert x^0 - x^{\ast}_{\tau_0}\Vert_{x^{\ast}_{\tau_0}} \leq \beta$.
To find this point, we can use either \eqref{eq:damped_step_PN} or \eqref{eq:homotopy_scheme0}.
However, since {we know that when} $\tau_0 = 0$, $x^{\ast}_{\tau_0} \equiv x^{\ast}_0 = x^0$, {in practice we can choose $\tau_0 >0$ to be} sufficiently small such that $x^{\ast}_{\tau_0} \approx x^0$, and skip \textbf{Stage 1}.

Practically, we {only} perform two stages as follows:
\begin{itemize}
\item Skip \textbf{Stage 1} and choose $\tau_0 > 0$ sufficiently small such that $\Vert x^0 - x^{\ast}_{\tau_0}\Vert_{x^{\ast}_{\tau_0}}$ is small.
\item In \textbf{Stage 2}, we {choose} $\sigma = 1$ to guarantee {that} $\norms{x^k - x^{\ast}_{\tau_k}}_{x^{\ast}_{\tau_k}} \leq \beta$ instead of $\norms{x^k - x^{\ast}_{\tau_k}}_{x^{\ast}_{\tau_k}} \leq \beta\sigma^k$.
Then we update $\tau_k$ from $\tau_0$ to $\tau_k\approx 1$.
\item In \textbf{Stage 3}, we fix $\tau_k$ and perform a couple of iterations to reach $\norms{x^k - x^{\ast}_{\tau_k}}_{x^{\ast}_{\tau_k}} \leq \varepsilon$.
\end{itemize}
We only perform \textbf{Stage 3} if we choose $\sigma = 1$.
In this case, we only have $\norms{x^k - x^{\ast}_{\tau_k}}_{x^{\ast}_{\tau_k}} \leq \beta$.
To achieve $\norms{x^k - x^{\ast}_{\tau_k}}_{x^{\ast}_{\tau_k}} \leq \varepsilon$, we need to perform a few proximal-Newton iterations with fixed $\tau_k$.

%%%%% 6. Primal-dual-dual method.
\section{Primal-Dual-Primal Method}\label{sec:PDP_alg}
Our second idea is a primal-dual-primal approach to solve  \eqref{eq:composite_cvx}.
We propose a primal-dual-primal method which consists of the following steps:
\begin{itemize}
\item Construct the Fenchel dual problem \eqref{eq:dual_prob} of \eqref{eq:composite_cvx}.
\item Apply Algorithm~\ref{alg:A1} to solve the dual problem \eqref{eq:dual_prob}.
\item Instead of solving the dual subproblem \eqref{eq:cvx_subprob_k}, we dualize it to go back to the primal space.
\item Construct an approximate primal solution of \eqref{eq:composite_cvx} from its dual {approximate solution.}
\end{itemize}
We will show in Section \ref{sec:GL_app} that this approach is useful for the well-known model \eqref{eq:GL_prob}.
Now, we present this method in detail as follows.

%%% 6.1. The dual problem.
\subsection{The dual problem}
We assume that $g(x) := \psi(Dx)$, where $\psi$ is a proper, closed, and convex function from $\R^n\to\Rext$, and $D : \R^p\to\R^n$ is a linear operator such that $n\leq p$.
The dual problem of \eqref{eq:composite_cvx} in this case becomes
\myeq{eq:dual_prob}{
\Psi^{\star} := \min_{y\in\R^n}\Big\{ \Psi(y) := f^{\ast}(-D^{\top}y) + \psi^{\ast}(y) \Big\},
}
where $f^{\ast}$ and $\psi^{\ast}$ are the Fenchel conjugates of $f$ and $\psi$, respectively.

Let us define $\varphi(y) := f^{\ast}(-D^{\top}y)$.
Then, we can compute the gradient and Hessian of $\varphi$ as 
\myeq{eq:grad_hess_of_varphi}{
\nabla{\varphi}(y) = -D\nabla{f}^{\ast}(-D^{\top}y),~~~\text{and}~~\nabla^2{\varphi}(y) = D\nabla^2{f^{\ast}}(-D^{\top}y)D^{\top}.
}
We impose the following assumption.

%%% Assumption 4.
\begin{assumption}\label{as:A4b}
The function $f$ in \eqref{eq:composite_cvx} is self-concordant as defined in Definition~\ref{de:gen_sel_con_def}, and $g(x) := \psi(Dx)$, where $\psi : \R^n\to\Rext$ is a proper, closed, and convex function and $D : \R^p\to\R^n$ is a linear operator such that $n\leq p$.
In addition, $D$ {has} full-row rank.
\end{assumption}

Under Assumption~\ref{as:A4b}, the function $\varphi$ is still a self-concordant function as stated in \cite[Theorem 2.4.1]{Nesterov1994} for $\dom{\varphi}$ defined as
\myeqn{
\dom{\varphi} = \set{ y\in\R^n \mid {-D^{\top}y } \in \dom{f^{\ast}}}.
}
We define the local norm with respect to $\varphi$ as $\norms{u}_y := (u^{\top}\nabla^2{\varphi}(y)u)^{1/2}$ and its dual norm $\norms{v}_y^{\ast} := (v^{\top}\nabla^2{\varphi}(y)^{-1}v)^{1/2}$.
The optimality condition of the dual problem \eqref{eq:dual_prob} becomes
\myeq{eq:opt_dual_prob}{
0 \in \nabla{\varphi}(y^{\star}) + \partial{\psi^{\ast}}(y^{\star}) \equiv -D\nabla{f}^{\ast}(-D^{\top}y^{\star}) + \partial{\psi^{\ast}}(y^{\star}),
}
which is necessary and sufficient for $y^{\star}$ to be an optimal solution of \eqref{eq:dual_prob} if $\dom{\varphi}\cap{\dom{\psi^*}}\neq\emptyset$.

Let $y^{\star}$ be an optimal solution of \eqref{eq:dual_prob}. Then, from  \eqref{eq:opt_dual_prob}, if we define  
\myeq{eq:sol_of_com_cvx}{
x^{\star} := \nabla{f}^{\ast}(-D^{\top}y^{\star}),
}
then $-D^{\top}y^{\star} \in \partial{f}(x^{\star})$, which leads to $0 \in D^{\top}y^{\star} +  \partial{f}(x^{\star})$.
On the other hand, we have $Dx^{\star} \in  \partial{\psi^{\ast}}(y^{\star})$, which leads to $y^{\star} \in \partial{\psi}(Dx^{\star})$.
Combining both expressions, we have $0 \in D^{\top}\partial{\psi}(Dx^{\star}) + \partial{f}(x^{\star})$.
Therefore, $x^{\star}$ given by \eqref{eq:sol_of_com_cvx} is an exact solution of the primal problem \eqref{eq:composite_cvx}.

%%% 6.2. The homotopy proximal-method methods for the dual problem.
\subsection{The homotopy proximal Newton method methods for the dual problem}
To fulfill the assumptions of Theorems \ref{th:linear_convergence2b} and \ref{th:linear_convergence2}, we assume that one of the following conditions holds:
\begin{itemize}
\item $f$ satisfied Assumption \ref{as:A4b} and $\psi^{\ast}$ is $L_{\psi^{\ast}}$-Lipschitz continuous w.r.t. $\norms{\cdot}_y$ defined by $\varphi$.
\item $f$ satisfied Assumption \ref{as:A4b} and is $\nu_f$-self-concordant barrier.
\end{itemize}
One can show that $\psi^{\ast}$ is $L_{\psi^{\ast}}$-Lipschitz continuous if $\dom{\partial{\psi}}$ is bounded w.r.t. the local norm defined by $f$, i.e. there exists $L_{\psi^{\ast}} > 0$ such that $\norms{u}_x \leq  L_{\psi^{\ast}}$ for any $u\in\dom{\partial{\psi}}$.
Since the dual problem \eqref{eq:dual_prob} has the same property as the primal one \eqref{eq:composite_cvx} under the above assumptions, let us apply Algorithm~\ref{alg:A1} with $H_k := \nabla^2{\varphi}(y^k)$ to solve this problem, which leads to
\myeq{eq:homotopy_method4}{
y^{k+1} :\approx \prox_{\frac{1}{\tau_{k+1}}\psi^{\ast}}^{\nabla^2{\varphi}(y^k)}\left(y^k - \nabla^2{\varphi}(y^k)^{-1}\nabla{\varphi_{\tau_{k+1}}}(y^k)\right),
}
where $\nabla{\varphi_{\tau_{k+1}}}(y^k) := \nabla{\varphi}(y^k) - \big(\frac{1}{\tau_{k+1}} - 1\big)\xi^0$ with $\xi^0 \in \partial{\psi^{\ast}}(y^0)$.
Here, $y^{k+1}$ is an approximation to the true solution $\bar{y}^{k+1}$ as defined in Definition~\ref{de:approx_prox}, where  $\bar{y}^{k+1}$ is {given by}
\myeq{eq:cvx_subprob3}{
{\!\!}\bar{y}^{k+1} {\!\!}:= \argmin_{y\in\R^n}\set{ \Pc_k(y) := \iprods{\nabla{\varphi_{\tau_{k\!+\!1}}}(y^k), y - y^k} + \tfrac{1}{2}\iprods{\nabla^2{\varphi}(y^k)(y - y^k), y - y^k} + \tfrac{1}{\tau_{k+1}}\psi^{\ast}(y) },{\!\!\!\!}
}
{and} both the gradient mapping $\nabla{\varphi}$ and the Hessian mapping $\nabla^2{\varphi}$ of $\varphi$ are given in \eqref{eq:grad_hess_of_varphi}, respectively.
This problem in general does not have a closed form solution.
{But observe that \eqref{eq:cvx_subprob3} is a convex composite quadratic
programming problem for which highly advanced algorithms such as the
semismooth Newton augmented Lagrangian method developed in \cite{li2018highly,zhao2010newton} can be
designed to solve it efficiently, as we shall demonstrate later 
in the numerical experiments.}

%%% 3.5. Solution of the subproblem.
\subsection{The dualization of the subproblem \eqref{eq:cvx_subprob3}}\label{subsec:dual_approach_for_cvx_subprob3}
Instead of solving the dual subproblem \eqref{eq:cvx_subprob3} directly, we dualize it to obtain the following subproblem in the primal space of $Dx$:
\myeq{eq:dual_subprob1c}{
{z^{k+1}  \approx \bar{z}^{k+1}}  := \argmin_{z\in\R^n}\Big\{ \Qc_k(z; y^k) := \tfrac{1}{2}\iprods{H(y^k)z, z} - \iprods{h_{\tau_{k+1}}(y^k), z} + \tfrac{1}{\tau_{k+1}}\psi(\tau_{k+1}z) \Big\},
}
where 
\myeqn{
\begin{array}{ll}
H(y^k) &:= \nabla^2{\varphi}(y^k)^{-1} =  (D\nabla^2{f^{\ast}}(-D^{\top}y^k)D^{\top})^{-1}, ~\text{and} \vspace{1ex}\\
h_{\tau_{k+1}}(y^k) &:= y^k - \nabla^2{\varphi}(y^k)^{-1}\nabla{\varphi_{\tau_{k+1}}}(y^k) = y^k - (D\nabla^2{f^{\ast}}(-D^{\top}y^k)D^{\top})^{-1}\nabla{\varphi_{\tau_{k+1}}}(y^k).
\end{array}
}
Clearly, this problem is again a composite strongly convex quadratic program of the same form as \eqref{eq:cvx_subprob3}, but in the primal space of $Dx$.
Specially, if $D =\Id$, the identity matrix, then \eqref{eq:dual_subprob1c} is in the primal space of $x$ as in \eqref{eq:cvx_subprob_k}.

%%% 3.6. Solution recovery
\subsection{Solution reconstruction for \eqref{eq:cvx_subprob3}}\label{subsec:solution_recovery}
{Recall that $\bar{z}^{k+1}$ denotes the} exact solution of \eqref{eq:dual_subprob1c}, then we can construct 
\myeq{eq:primal_sol_recovery}{
\bar{y}^{k+1} := y^k- \nabla^2{\varphi}(y^k)^{-1}\left(\nabla{\varphi_{\tau_{k+1}}}(y^k) + \bar{z}^{k+1}\right),  
}
as an exact solution of \eqref{eq:cvx_subprob3}.

Assume that we can only solve \eqref{eq:dual_subprob1c} up to a given accuracy $\delta \geq 0$.
In this case, we say that $z^{k+1}$ is a $\delta$-approximate solution to $\bar{z}^{k+1}$ of \eqref{eq:dual_subprob1c} if for any $\tilde{e}_k$ such that $\norm{\tilde{e}_k}_{y^k} \leq \delta$, we have
\myeq{eq:approx_sol3}{
\tilde{e}_k \in H(y^k)z^{k+1} - h_{\tau_{k+1}}(y^k) + \partial{\psi}(\tau_{k+1} z^{k+1}).
}
To guarantee \eqref{eq:approx_sol3}, we can apply inexact first-order methods to solve \eqref{eq:dual_subprob1c}, see, e.g., in \cite{Schmidt2011,villa2013accelerated}.

If $z^{k+1}$ satisfies \eqref{eq:approx_sol3}, then we can construct an approximate solution $y^{k+1}$ to $\bar{y}^{k+1}$ as
\myeq{eq:primal_sol_recovery_inexact}{
y^{k+1} := y^k - \nabla^2{\varphi}(y^k)^{-1}\left(\nabla{\varphi_{\tau_{k+1}}}(y^k) + z^{k+1}\right)  + \tilde{e}_k.
}
The following lemma shows a relation between $z^{k+1}$ of \eqref{eq:dual_subprob1c} and the approximate solution $y^{k+1}$ of \eqref{eq:cvx_subprob3}, whose proof is given in Appendix \ref{apdx:le:approx_pd_sol}.

%% Lemma 3.1.
\begin{lemma}\label{le:approx_pd_sol}
Let $z^{k+1}$ be a $\delta$-approximate solution to $\bar{z}^{k+1}$ of \eqref{eq:dual_subprob1c} in the sense of \eqref{eq:approx_sol3}. 
Then, $y^{k+1}$ constructed by \eqref{eq:primal_sol_recovery_inexact} is also a $\delta$-approximate solution to the true solution $\bar{y}^{k+1}$ of \eqref{eq:cvx_subprob3} such that $\Pc_k(y^{k+1}) - \Pc_k(\bar{y}^{k+1}) \leq \frac{\delta^2}{2}$.
\end{lemma}

\subsection{Primal solution recovery}\label{subsec:primal_solution_reconstruction}
Finally, we show how to recover an approximate primal solution $x^k$ of the original problem \eqref{eq:composite_cvx} from its dual approximate solution $y^k$.
Based on \eqref{eq:sol_of_com_cvx}, we show below that for an approximate solution $y^k$ to $y^{\star}$, the following point
\myeq{eq:app_sol_of_com_cvx}{
x^k = \nabla{f}^{\ast}(-D^{\top}y^k)
}
is an {approximate} solution to the true solution $x^{\star}$ of \eqref{eq:composite_cvx} as stated in the following theorem whose proof is given in Appendix \ref{apdx:le:xk_app_sol}.
In particular, if { $D$ is} an invertible matrix, then one can show that we can construct an approximate solution $x^{k+1}$  to $x^{\star}$ of \eqref{eq:composite_cvx} from $z^{k+1}$.

%%% Theorem 6.
\begin{theorem}\label{le:xk_app_sol}
Let $y^{\star}$ be an exact solution of  the dual problem \eqref{eq:dual_prob}. Then
\begin{itemize}
\item[$(\mathrm{a})$] 
$x^{\star}$ constructed by \eqref{eq:sol_of_com_cvx} is an exact solution of  \eqref{eq:composite_cvx}.

\item[$(\mathrm{b})$] Let $\set{y^k}$ be computed by \eqref{eq:primal_sol_recovery_inexact} and $\set{x^k}$ be given by \eqref{eq:app_sol_of_com_cvx} such that $\Vert y^k - y^{\star}\Vert_{y^{\star}} \!<\! 1$. Then
\vspace{-2ex}
\myeq{eq:xk_app_sol}{
\Vert x^k - x^{\star}\Vert_{x^{\star}} := \iprods{\nabla^2{f^{\ast}}(-D^{\top}y^{\star})^{-1}(x^k - x^{\star}), (x^k - x^{\star})}^{1/2} \leq \frac{\Vert y^k - y^{\star}\Vert_{y^{\star}}}{1 - \Vert y^k - y^{\star}\Vert_{y^{\star}}}.
}
Consequently, under the conditions of Theorem \ref{th:linear_convergence2b} or Theorem \ref{th:linear_convergence2}, the sequence $\set{x^k}$ converges linearly to the optimal solution $x^{\star}$ of \eqref{eq:composite_cvx}.

\item[$(\mathrm{c})$]
Let $z^{k+1}$ be an approximate solution of \eqref{eq:dual_subprob1c}. 
If $\Vert y^k - y^{\star}\Vert_{y^k} < 1$, then
\myeq{eq:norm_opt_sol_pd}{
{\!\!\!}\Vert z^{k+ 1} - Dx^{\star}\Vert^{\ast}_{y^k} \leq \tfrac{\Vert y^{\star} - y^k\Vert_{y^k}^2 }{1 - \Vert y^{\star} - y^k\Vert_{y^k}} + \Vert y^{k+1} - y^{\star}\Vert_{y^k} + \big(\tfrac{1}{\tau_{k+1}}-1\big)\Vert\xi^0\Vert^{\ast}_{y^k} + \Vert \tilde{e}_k\Vert_{y^k}.{\!\!\!}
}
Assume that we apply Algorithm~\ref{alg:A1} to solve the dual problem \eqref{eq:dual_prob} under the assumptions of Theorem \ref{th:linear_convergence2b} or Theorem \ref{th:linear_convergence2} and the choice $\delta_k \leq \frac{\lambda_k}{113}$.
If, in addition, $D$ is invertible, then  $x^{k+1} := D^{-1}z^{k+1}$ is an approximate solution to $x^{\star}$ of \eqref{eq:composite_cvx}.
 Moreover, $\big\{ \Vert x^{k+1} - x^{\star}\Vert^{\ast}_{y^k} \big\}$ converges linearly to zero.
\end{itemize}
\end{theorem}
%%% End of Lemma

From Theorem~\ref{le:xk_app_sol}, we can see that if $D$ is invertible, {then} we can directly use $x^{k+1} := D^{-1}z^{k+1}$ to approximate the solution $x^{\star}$ of \eqref{eq:composite_cvx}.
Otherwise, we can construct an approximate solution $x^k$ to $x^{\star}$ by using \eqref{eq:app_sol_of_com_cvx}, which requires one evaluation of $\nabla{f^{\ast}}$.

%%%%%%%%%%%%%%%%%%%%%%%%%%%%%%%%%%%
%%% 4. Applications to covariance estimation.
%%%%%%%%%%%%%%%%%%%%%%%%%%%%%%%%%%%
\subsection{Applications to covariance estimation}\label{sec:GL_app}
In this section, we apply Algorithm~\ref{alg:A1} and the primal-dual-primal method in Section \ref{sec:PDP_alg} to solve the regularized covariance estimation problem \eqref{eq:GL_prob} as in \cite{Friedman2008} and its least-squares extension in \cite{Kyrillidis2014}.

%%% 4.1. Regularized inverse covariance estimation
We recall the primal regularized covariance estimation problem given in \eqref{eq:GL_prob}.
Associated with \eqref{eq:GL_prob}, we can also consider its dual form:
\myeq{eq:GL_dual}{
\Psi^{\star} :=  \min_{Y}\Big\{ \Psi(Y) := -\log\det(Y + \Sigma) +  \psi^{\ast}(Y) ~\mid~ Y + \Sigma \succ 0 \Big\}.
}
Here, $\psi^{\ast}$ is the Fenchel conjugate of $\psi(X) := g(X)$. 
This problem again has the same form as \eqref{eq:composite_cvx}.
Instead of solving the primal problem \eqref{eq:GL_prob}, we apply Algorithm~\ref{alg:A1} to solve the dual problem \eqref{eq:GL_dual} and reconstruct a solution of \eqref{eq:GL_prob} from its dual.

%%% b. The main steps of algorithm.
\subsubsection{The main steps of the algorithm}
Given $Y_k$ such that $Y_k + \Sigma \succ 0$, we define $X_k := (Y_k + \Sigma)^{-1}$.
The main step of the algorithm is to solve the following subproblem
\myeq{eq:cvx_subprob2_app}{
{\!\!\!\!}Y_{k+1} \approx \bar{Y}_{k\!+\!1} := \argmin_Y\Big\{ \Pc_k(Y) \!:=\!  -\mathrm{trace}\big(\widehat{X}_k(Y \!-\! Y_k)\big) + \tfrac{1}{2}\trace{X_k(Y \!-\! Y_k)}^2 +  \tfrac{1}{\tau_{k\!+\!1}}\psi^{\ast}(Y) \Big\},{\!\!\!\!\!\!\!\!\!}
}
where $\widehat{X}_k := X_k - \Xi_k \equiv X_k - \big(\tfrac{1}{\tau_{k+1}} - 1\big)\Xi_0$ for a fixed $\Xi_0\in\partial{\psi^{\ast}(Y_0)}$.
As discussed in Section \ref{sec:PDP_alg}, instead of solving \eqref{eq:cvx_subprob2_app}, we look at its dual form
\myeq{eq:dual_subprob2_app}{
Z_{k+1} \approx \bar{Z}_{k+1} := \argmin_{X}\Big\{ \Qc_k(X) := - \mathrm{trace}\big(C_kX\big) + \tfrac{1}{2}\big(\mathrm{trace}\big((Y_k + \Sigma)X\big)^2\big) + \tfrac{1}{\tau_{k+1}}\psi(X) \Big\},
}
where $C_k := 2Y_k - \Xi_k + \Sigma$.
Once $Z_{k+1}$ is computed from \eqref{eq:dual_subprob2_app}, we can reconstruct $Y_{k+1}$ as follows:
\myeq{eq:primal_sol}{
Y_{k+1} := 2Y_k - \Xi_k +  \Sigma - (Y_k + \Sigma)Z_{k+1}(Y_k + \Sigma),
}
and compute an inexact Newton decrement
\myeq{eq:local_distancesb}{
\tilde{\lambda}_k := \big( p - 2\mathrm{trace}(W_k) + \mathrm{trace}(W_k^2) \big)^{1/2},~~\text{where}~~W_k := Z_{k+1}(Y_k + \Sigma).
}
Finally, when an $\varepsilon$-solution $\widetilde{Y}^{\star}$ of \eqref{eq:GL_dual} is computed (i.e. $\widetilde{Y}^{\star} := Y_{k_{\max}}$), we can reconstruct an approximate solution $\widetilde{X}^{\star}$ of the primal problem \eqref{eq:GL_prob} by taking $\widetilde{X}^{\star} := (\Sigma + \widetilde{Y}^{\star})^{-1}$. 
This computation requires the inverse of a symmetric positive definite matrix, which can be done efficiently by Cholesky decomposition.
However, as shown in {Theorem} \ref{le:xk_app_sol}, we can use $Z_{k+1}$ computed by \eqref{eq:dual_subprob2_app} to approximate the true solution $X^{\star}$.
This allows us to avoid the matrix inversion $(\Sigma + \widetilde{Y}^{\star})^{-1}$.

\subsubsection{The algorithm}
Putting together these steps, we obtain a new algorithmic variant for solving \eqref{eq:GL_prob} as presented in Algorithm~\ref{alg:A1b}.

%%%%%%%%%%%%%%%%%%%%%%%%%%%%%%%%%%%%%%%%%%
%%% + Algorithm 4.
%%%%%%%%%%%%%%%%%%%%%%%%%%%%%%%%%%%%%%%%%%
\begin{algorithm}[ht!]\caption{(\textit{An inexact primal-dual-primal homotopy proximal-Newton algorithm for \eqref{eq:GL_prob}})}\label{alg:A1b}
\begin{algorithmic}[1]
\State\textbf{Initialization:} A desired tolerance $\varepsilon > 0$, and an initial point  $Y_0$ such that $Y_0 + \Sigma \succ 0$. 
Evaluate a subgradient $\Xi_0\in\partial{\psi^{\ast}}(Y_0)$.
\vspace{0.5ex}
\State\textbf{Iteration:}~\textit{For $k = 0$ to $k_{\max}$, perform}
\vspace{0.5ex}
\State\hspace{0.70cm}\label{a2step:p1_step0} Update $\tau_{k+1}$ as in \eqref{eq:sigma_choice}.
\vspace{0.5ex}
\State\hspace{0.70cm}\label{a2step:p1_step1} Solve \eqref{eq:dual_subprob2_app} up to a tolerance $\delta_k \leq \bar{\delta}_{0} := 0.1\varepsilon$ to get $Z_{k+1}$.
\vspace{0.5ex}
\State\hspace{0.70cm}\label{a2step:p1_step2} Compute $D_k := Y_k +  \Sigma - (Y_k + \Sigma)Z_{k+1}(Y_k + \Sigma)$, and compute $\tilde{\lambda}_k$ as \eqref{eq:local_distancesb}.
\vspace{0.5ex}
\State\hspace{0.70cm}\label{a2step:p1_step3} If $\tilde{\lambda}_k \leq \varepsilon$ and $1 - \tau_{k+1} \leq \varepsilon$, then \textbf{terminate}.
\vspace{0.5ex}
\State\hspace{0.70cm}\label{a2step:p1_step4} If damped step is used, then compute $\alpha_k := \frac{\tilde{\lambda}_k  - \bar{\delta}_0}{\tilde{\lambda}_k (1 + \tilde{\lambda}_k  - \bar{\delta}_0)}$. Otherwise, set $\alpha_k := 1$.
\vspace{0.5ex}
\State\hspace{0.70cm}\label{a2step:p1_step5} Update $Y_{k+1} := Y_k + \alpha_k D_k$.
\vspace{0.5ex}
\State\hspace{0.20cm}\textit{End for $k$.}
\vspace{0.5ex}
\State\textbf{Output:}\label{a1b:p4_construct} Output $Y_k$ as an $\varepsilon$-solution of \eqref{eq:GL_dual} and $Z_{k+1}$ as an $\varepsilon$-solution of \eqref{eq:GL_prob}.
\end{algorithmic}
\end{algorithm}
%%% End of the algorithm.
%%%%%%%%%%%%%%%%%%%%%%%%%%%%%%%%%%%%%%%%%%

Let us highlight some new features of Algorithm~\ref{alg:A1b} {as} 
compared to existing methods in the literature, e.g., \cite{Friedman2008,Hsieh2011,hsieh2013big,Tran-Dinh2013b,Tran-Dinh2013a}.
\begin{itemize}

\item[(a)] Firstly, Algorithm~\ref{alg:A1b} deals with a general regularizer compared to \cite{Friedman2008,Hsieh2011,hsieh2013big}. 
When $g$ is the $\ell_1$-norm regularizer, we can apply coordinate descent methods as in \cite{Friedman2008,Hsieh2011,hsieh2013big} for solving \eqref{eq:dual_subprob2_app} to improve its practical performance.

\item[(b)] Secondly, Algorithm~\ref{alg:A1b} relies on Algorithm~\ref{alg:A1} to solve the dual problem \eqref{eq:GL_dual} instead of standard proximal-Newton methods.
It has a linear convergence rate  compared to the damped-step scheme  {which only has a} sublinear convergence rate {as shown}
in \cite{Tran-Dinh2013b,Tran-Dinh2013a}.

\item[(c)]  Thirdly, it does not require any linesearch or any additional assumption in our analysis to achieve a linear convergence rate.

\item[(d)]  Fourthly, the whole algorithm does not require any matrix inversion or Cholesky decomposition as long as we can solve the subproblem \eqref{eq:dual_subprob2_app} with a first order method.
This is an important feature for designing parallel and distributed variants of Algorithm~\ref{alg:A1b} as compared {to} \cite{hsieh2013big}.

\item[(e)] Finally, the subproblem~\eqref{eq:dual_subprob2_app} works on the original regularizer $g$ instead of the dual problem as in \cite{Tran-Dinh2013b}, which preserves the structure {such as sparsity on the iterates as promoted by the regularizer} $g$.
\end{itemize}

%%% 8. Numerical experiments.
\section{Numerical experiments}\label{sec:num_exp}
We provide some numerical experiments to illustrate our theoretical development.
Our experiments are implemented in Matlab 2018a running on a Dell Optiplex 9010, 3.4 GHz Intel Core i7-3770 with 16GB 1600 MHz DDR3 memory.

%%% 8.1. Homotopy vs. non-homotopy methods
\subsection{Lipschitz gradient and strongly convex models}\label{subsec:hom_vs_nonhom}

Now {we evaluate} the performance of the homotopy {proximal-Newton} scheme \eqref{eq:homotopy_method} by applying it to solve the following logistic regression problem with an elastic-net regularizer:
\myeq{eq:sparse_logistic_exam}{
F^{\star} := \min_{x\in\R^p}\Big\{ F(x) := \frac{1}{n}\sum_{i=1}^n\log\big(1 + \exp( - y_ia_i^{\top}x) \big) + \frac{\mu_f}{2}\Vert x\Vert^2 + \rho\Vert x\Vert_1 \Big\},
}
where $\mu_f > 0$ and $\rho > 0$ are two regularization parameters, and {$(a_i, y_i)
\in  \R^p\times \set{-1, 1}$,
$i=1,\ldots,n$ is  a given dataset. }
As shown in \cite{zou2005regularization}, {the}
elastic-net regularizer helps to remove variable limitation with more freedom than 
{the} classical LASSO model, and {it can also carter for groups of nonzero} variables.
Clearly, $f(x) := \frac{1}{n}\sum_{i=1}^n\log\big(1 + \exp( - y_ia_i^{\top}x) \big) + {\frac{\mu_f}{2}}\Vert x\Vert^2$ is $\mu_f$-strongly convex, and $L_f$-Lipschitz gradient continuous with $L_f := { \frac{1}{2n}}\norms{A}^2 + \mu_f$,  
where 
{$A^{\top} = [a_1,\ldots, a_n] \in \R^{p\times n}.$}
Moreover, the function $g(x) := \rho\norms{x}_1$ is $L_g$-Lipschitz continuous with $L_g := \rho$.
Hence, Assumption \ref{as:A3} of Theorem~\ref{th:linear_convergence1} is satisfied.

We implement Algorithm~\ref{alg:A1} to solve \eqref{eq:sparse_logistic_exam} and compare it with homotopy quasi-Newton variant,  standard proximal-gradient scheme \cite{Beck2009}, and the accelerated proximal-gradient method with line-search and {restart} \cite{Beck2009,Becker2011a,Su2014}. 
These methods are abbreviated as ``\texttt{HomoPN}'', ``\texttt{HomoQuasiPN}'', ``\texttt{PG}'', and ``\texttt{Ls-Rs-APG}'', respectively. 
We test these algorithms on several {binary} classification datasets \texttt{a1a}, \texttt{a9a}, \texttt{w1a}, \texttt{w8a}, \texttt{covtype.binary}, \texttt{news20.binary}, \texttt{rcv1.binary} and \texttt{real-sim} from \cite{CC01a}, and \texttt{mnist17} and \texttt{mnist38}  from the \texttt{mnist} dataset 
{where we choose the digits} $1$ and $7$, and the {digits} 3 and 8. 
The details of these dataset set is given in Table \ref{tbl:dataset_logistic}. 

Following \cite{defazio2014saga}, we set $\mu_f = \frac{1}{n}$. The 
{parameter $\rho$ for $\ell_1$-regularization is selected to produce about} $10$ percent of nonzeros coefficients. 
We should mention here that the subpropblem in the model \eqref{eq:sparse_logistic_exam} is an elastic-net regularized least-squares problem. 
For Algorithm 1 to be numerically efficient, it is crucial for us to solve those subproblems efficiently. 
Fortunately we can  adapt the highly efficient semismooth Newton augmented Lagrangian method in \cite{li2018highly} to solve the subproblems \eqref{eq:homotopy_method}.
We terminate the experiments when the relative gaps are less than a given tolerance $\varepsilon = 10^{-6}$, based on the KKT system of \eqref{eq:sparse_logistic_exam}. 
Moreover, for \texttt{PG} and \texttt{Ls-Rs-APG} methods, we set the maximum number of iterations at $2\times 10^4$. 
If any method does not achieve our desired accuracy after at most $2\times 10^4$ {iterations}, we use "---" to represent the result. 
Our final results are reported in Table \ref{tbl:perform_logistic}, where \texttt{iter} is the number of iterations, \texttt{time[s]} is the computational time in second, and \texttt{rgap} is the relative gap times $10^{-7}$. 
We {highlight that \eqref{eq:sparse_logistic_exam} can be solved effectively by a stochastic method to a modest level of accuracy} when the number of data points $n$ is large.
However, in our experiments, we only focus on relatively moderate datasets {as 
we are interested in solving the problems accurately to evaluate
the performance of Algorithm 1, and ignore the} comparison with stochastic methods.

\begin{table}[ht!]
\vspace{-3ex}
\begin{center}
\newcommand{\cell}[1]{{\!\!}#1{\!}}
\caption{The information of binary classification datasets used in our experiments.}\label{tbl:dataset_logistic}
\vspace{-1ex}
\resizebox{\textwidth}{!}{
\begin{tabular}{|l|c|c|c|c|c|c|c|c|c|c|}\hline
	\cell{~~~Dataset}      & \cell{a1a}      & \cell{a9a}    & \cell{covtype} & \cell{news20} & \cell{mnist17} & \cell{mnist38} & \cell{rcv1} & \cell{real-sim} & \cell{w1a} & \cell{w8a}   \\ \hline
	\cell{\#{samples}[$n$]}      & \cell{30956}  & \cell{16281} & \cell{581012} & \cell{19996}     & \cell{317402} & \cell{292363} & \cell{677399} & \cell{72309} & \cell{47272} &\cell{14951} \\ \hline
	\cell{\#features[$p$]} & \cell{123}     & \cell{122}      & \cell{54}         & \cell{1355191} & \cell{784}      & \cell{784}       & \cell{47236}   & \cell{20958} & \cell{300}      & \cell{300} \\ \hline
\end{tabular}}
\end{center}
\vspace{-4ex}
\end{table}

\begin{table}[ht!]
\vspace{-1ex}
\begin{center}
\newcommand{\cell}[1]{{\!\!}#1{\!}}
\newcommand{\cells}[1]{{\!\!}#1{\!\!}}
\caption{The performance and results of four algorithms on the logistic regression problem \eqref{eq:sparse_logistic_exam}.}\label{tbl:perform_logistic}
\vspace{-1ex}
\resizebox{\textwidth}{!}{
\begin{tabular}{|c| rrr | rrr | rrr | rrr | r | r |}
\hline
	\multirow{2}{*}{\cell{Datasets}} & \multicolumn{3}{c|}{\cell{\texttt{HomoPN}}} & \multicolumn{3}{c|}{\cell{\texttt{HomoQuasiPN}}} & \multicolumn{3}{c|}{\cell{\texttt{PG}}} & \multicolumn{3}{c|}{\cell{\texttt{Ls-Rs-APG}}} & \multirow{2}{*}{\cell{Sparsity}} & \multirow{2}{*}{\cell{$\rho~~~~$}} \\ \cline{2-13}
	&\cells{iter} & \cells{time[s]} & \cells{rgap} &\cells{iter} & \cells{time[s]} & \cells{rgap} &\cells{iter} & \cells{time[s]} & \cells{rgap} &\cells{iter} & \cells{time[s]} & \cells{rgap} &  & \\ \hline
	\cells{a1a}                       &   \cell{7} &\cell{     0.247} & \cell{0.3}      & \cell{  13} & \cell{  0.190}  & \cell{0.3}      & \cell{1049} & \cell{      12.264} & \cell{10.0}       & \cell{1088}    & \cell{     16.272}  & \cell{7.6}      & \cell{11.38\%}                   & \cell{1e-2}                    \\ \hline
	\cells{a9a}                       &  \cell{6}  &\cell{     0.153} & \cell{0.4}      & \cell{  13} & \cell{  0.169}  & \cell{1.1}      & \cell{1101}  & \cell{       6.849} & \cell{10.0}       & \cell{   651}   & \cell{       4.959}  & \cell{3.5}      & \cell{12.30\%}                   & \cell{1e-2}                    \\ \hline
	\cells{covtype}                 & \cell{12} &\cell{     3.621} & \cell{2.1}      & \cell{121} & \cell{  5.938}  & \cell{7.0}      &    ---   &             --- &  ---       & \cell{10508} & \cell{3997.766}   & \cell{9.8}      & 9.26\%                    & \cell{1e-5}                    \\ \hline
	\cells{mnist17}                 &  \cell{6}  &\cell{ 101.088} & \cell{0.3}      & \cell{  24} & \cell{25.030}  & \cell{6.2}      &    ---   &             --- &  ---       & \cell{     912} & \cell{2030.214}   & \cell{8.0}      & 10.08\%                   & \cell{2.5e-3}                  \\ \hline
	\cells{mnist38}                 &  \cell{4}  &\cell{ 115.304} & \cell{6.1}      & \cell{  14} & \cell{36.109}  & \cell{9.7}      &    ---  &             --- &   ---       & \cell{13392}  & \cell{4151.542}   & \cell{10.0}     & 11.86\%                   & \cell{5.5e-3}                  \\ \hline
	\cells{news20}                 &  \cell{5}  &\cell{   9.149}  & \cell{0.1}     & \cell{  13} & \cell{104.030}  & \cell{1.2}     & \cell{1307} & \cell{     71.579} & \cell{10.0}        & \cell{      90} & \cell{     57.225}   & \cell{3.4}       & \cell{11.04\%}                   & \cell{2.5e-6}                  \\ \hline
	\cells{rcv1}                       & \cell{11}  &\cell{  244.589} & \cell{7.7}      & \cell{  41} & \cell{782.247} & \cell{3.8}     & \cell{10291} & \cell{3254.890} & \cell{10.0}       & \cell{      409} & \cell{ 1222.426}  & \cell{5.0}      & \cell{10.14\%}                   & \cell{2.5e-6}                 \\ \hline
	\cells{real-sim}                 & \cell{6}  &\cell{      3.966} & \cell{7.7}      & \cell{  22} & \cell{  24.055} & \cell{6.5}     & \cell{    730} & \cell{   14.182}  & \cell{10.0}       & \cell{      103} & \cell{     14.255}  & \cell{9.0}     & \cell{11.25\%}                   & \cell{2e-5}                    \\ \hline
	\cells{w1a}                       & \cell{7}  &\cell{      0.620} & \cell{1.6}      & \cell{  19} & \cell{   0.426}  & \cell{6.7}     & \cell{    910} & \cell{    31.084} & \cell{10.0}       & \cell{      666} & \cell{     13.710}  & \cell{7.6}       & \cell{10.00\%}                   & \cell{1e-3}                    \\ \hline
	\cells{w8a}                       & \cell{7}  &\cell{      0.245} & \cell{1.4}      & \cell{  19} & \cell{   0.380}  & \cell{2.1}     & \cell{    891} & \cell{     9.172}  & \cell{10.0}       & \cell{      297} & \cell{       1.856} & \cell{6.6}        & \cell{10.00\%}                   & \cell{1e-3}                    \\ \hline
\end{tabular}}
\vspace{-1ex}
\end{center}
\end{table}

Table \ref{tbl:perform_logistic} shows that since $f$ is strongly convex and {has continuous} Lipschitz gradient, the standard proximal-gradient method \texttt{PG} 
{can solve} some of the problems  slightly better than the homotopy quasi-Newton method \texttt{HomoQuasiPN}, {particularly for some of the datasets where} the number of features is large (e.g., \texttt{news20} and \texttt{real-sim}). 
Note that since the problem \eqref{eq:sparse_logistic_exam} is strongly convex, both \texttt{PG} and \texttt{Ls-Rs-APG} have linear convergence rate.
\texttt{Ls-Rs-APG} outperforms \texttt{PG} in terms of iteration numbers as stated by the theory. 
But since we adopt {the} line search scheme, we found that the 
{times taken by these two methods in solving}
some of the problems {are} roughly at the same level. 
However, \texttt{HomoQuasiPN} is promising when the number of features is moderate but the number of observations is large, for instance {the} datasets \texttt{mnist17} and \texttt{mnist38}. 
The reason is that when the number of observations is large, the cost of computing the $H_k$ for \texttt{HomoPN} will be more expensive than that in \texttt{HomoQuasiPN}. 
Hence, with warm start, \texttt{HomoQuasiPN} will only need a few more iterations to converge but {with lower computing cost} in each iteration. 
In other situations, our \texttt{HomoPN} outperforms all other methods in terms of iteration numbers and {computation} time. 
Furthermore, our homotopy methods can often solve the problems more {accurately}. 
To conclude, our homotopy  methods are highly efficient for solving a varieties of large-scale datasets in logistic regression.

%%% 8.3. Non-Lipschitz gradient models.
\subsection{Self-concordant barrier models}\label{eq:experiment_design}
We illustrate the variant of Algorithm~\ref{alg:A1} in Theorem~\ref{th:linear_convergence2} {for} solving the following constrained convex problem {with a self-concordant barrier function $f$ arising
 from D-optimal experimental design}, see, e.g. \cite{harman2009approximate,lu2013computing}:
 \vspace{-0.5ex}
\myeq{eq:D-ex-design}{
F^{\star} := \min_{x\in\R^p}\Big\{ F(x) := -\log\det\Big(\sum_{i=1}^px_iA_i\Big) ~~\mid~~ \sum_{i=1}^px_i = 1, ~x\geq 0 \Big\},
}
where $A_i$ for $i=1,\cdots, p$ are $m\times m$ symmetric positive semidefinite matrices. 
If we define $f(x) := -\log\det\Big(\sum_{i=1}^px_iA_i\Big)$, and $g(x) := \delta_{\Delta_p}(x)$, where $\Delta_p$ is the standard simplex in $\R^p$, then we can reformulate \eqref{eq:D-ex-design} into \eqref{eq:composite_cvx}, and $f$ satisfies the assumptions of Theorem \ref{th:linear_convergence2}.

We implement Algorithm~\ref{alg:A1} to solve \eqref{eq:D-ex-design}, and compare it with the interior-point method in \cite{lu2013computing}, where their code is available online at \href{http://www.mypolyuweb.hk/~tkpong/OD_final_codes/}{http://www.mypolyuweb.hk/~tkpong/OD\_final\_codes/}. 
To approximately {solve the proximal-Newton subproblem
\eqref{eq:dual_subprob2_app},  we adapt the semi-smooth Newton-CG augmented Lagrangian method in \cite{li2018highly,zhao2010newton} to solve the composite convex  QP problem where the 
nonsmooth term is given by $g(x) = \delta_{\Delta_p}(x)$.}
We also compare our method with the multiplicative method proposed in \cite{harman2009approximate} and the interior-point method implemented in SDPT3-v.4.0 \cite{Toh2010}.
We abbreviate these solvers by \texttt{HomoPN}, \texttt{MUL}, \texttt{IP}, and \texttt{SDPT3}, respectively.
Unlike other well-established IP solvers, SDPT3 allows us to handle directly the log-determinant functions without reformulation or approximation.

We follow the same procedure as in \cite{lu2013computing} to generate the data and use the implementation of \texttt{MUL} from  \cite{lu2013computing}.
More precisely, we consider the following four design spaces:
\myeqn{\begin{array}{ll}
\chi_1 & := \set{ x_i = (e^{-s_i},s_ie^{-s_i},e^{-2s_i},s_ie^{-2s_i})^{\top},~1\leq i\leq p } \subset\R^4, \vspace{0.75ex}\\
\chi_2 & := \set{ x_i = (1,s_i,s_i^2,s_i^3)^{\top},~1\leq i\leq p } \subset\R^4, \vspace{0.75ex}\\
\chi_3& := \set{ x_{(i-1)\lceil \sqrt{p}\rceil+j} = (1,r_i,r_i^2,t_j,r_it_j)^{\top},~1\leq i,j\leq \lceil \sqrt{p}\rceil } \subset\R^5, \vspace{0.75ex}\\
\chi_4 & := \set{ x_i = (t_i,t_i^2,\sin(2\pi t_i), \cos(2\pi t_i))^{\top}, ~1\leq i\leq p } \subset\R^4,
\end{array}}
where $s_i = \frac{3i}{p}, r_i = \frac{2i}{p}-1$ and $t_i = \frac{i}{p}$. 

For each design space, we set $A_i := x_ix_i^{\top}$ for $i=1,\cdots, p$, with $p =10000,50000,100000$ for $\chi_1,\chi_2,\chi_4$, and $p = 10000,40000,90000$ for $\chi_3$. 
In this case, the problem dimension $m$ is $m = 4$ in $\chi_1$, $\chi_2$ and $\chi_4$ and $m = 5$ in $\chi_3$.
The performance of these four methods on $9$ problems of different sizes are reported in Table~\ref{tbl:perform_doptimal}, where (\texttt{\#Iterations}) {and} \texttt{Time[s]} 
denote the number of iterations and computational {time taken, respectively, 
and $F(x^k)$ is the approximate optimal objective value attained for} 
\eqref{eq:D-ex-design}.
 
\begin{table}[ht!]
\begin{center}
%\vspace{-3ex}
\caption{The performance of $4$ algorithms on the $D$-optimal experimental design problem \eqref{eq:D-ex-design}.}\label{tbl:perform_doptimal}
\vspace{-1ex}		
\resizebox{\textwidth}{!}{
	\begin{tabular}{@{} r | r |l|l|l|l|l|l|l|l@{}}
		\toprule
		\multicolumn{2}{l|}{Problem}          & \multicolumn{4}{c|}{(\texttt{\#Iterations}) \texttt{Time[s]}} & \multicolumn{4}{c}{Objective value $F(x^k)$} \\ \midrule
		$\chi_i$               &  $p~~~$    & \texttt{~HomoPN}  & \texttt{~MUL\cite{harman2009approximate}} & \texttt{~IP\cite{lu2013computing}} & \texttt{SDPT3} & \texttt{~HomoPN}  & \texttt{~MUL\cite{harman2009approximate}} & \texttt{~IP\cite{lu2013computing}} & \texttt{SDPT3} \\ \midrule
		1                      & 10000  & (7)0.088 &(2509)0.407&(127)0.262&0.399 &20.51196&20.51254&20.51195&20.51241\\ 
		1                      & 50000  & (7)0.368&(2510)1.050 &(122)0.998 &1.127 &20.50907 & 20.50981 &20.50907&20.50908  \\ 
		1                      & 100000 &(7)0.956 &(2855)3.073&(120)1.962&2.008& 20.50871&20.50943&20.50872&20.50883\\ \midrule
		2                      & 10000  &(7)0.057&(2493)0.571&(102)0.211& 0.368&0.410236&0.410745&0.410221& 0.410220\\ 
		2                      & 50000  &(6)0.266&(3278)3.122& (101)0.867&1.280& 0.409263&0.409964 &0.409267&0.409260 \\ 
		2                      & 100000 &(5)0.525 &(3910)9.361&(100)1.935&2.701&0.409143&0.409795 &0.409154&0.409145\\ \midrule
		3                      & 10000  &(5)0.048&(1619)0.428&(97)0.236&0.415&5.142670 &5.142919&5.142671& 5.142670\\  
		3                      & 40000  &(5)0.214& (2208)2.378&(97)0.784& 1.578&5.082114&5.082363&5.082119&5.082113\\  
		3                      & 90000  &(5)0.535&(2407)7.178& (95)1.774& 3.827& 5.062011&5.062261&5.062024&5.062011\\ \midrule
		4                      & 10000  &(6)0.069&(2512)0.488&(135)0.271& 0.325& 7.251897&7.252565&7.251890&7.251888\\ 
		4                      & 50000  & (6)0.402&(3413)2.870&(130)1.045&1.064&7.251892&7.252527&7.251895 &7.251889\\ 
		4                      & 100000 & (6)1.032& (4013)8.480&(128)2.554&2.294&7.251891&7.252461&7.251902&7.251888\\ \bottomrule
	\end{tabular}}
\end{center}	
\vspace{-2ex}	
\end{table}

In Table \ref{tbl:perform_doptimal}, the objective value $F(x^k)$ is rounded off to seven significant digits. 
We can see that our homotopy {method}, \texttt{HomoPN} outperforms {the} multiplicative algorithm \texttt{MUL} in terms of computational time, and 
{achieving} much smaller objective values in all the instances.
Our \texttt{HomoPN} also outperforms the interior point method (\texttt{IP}) and SDPT3 in terms of time and also gives slightly better objective values in  most instances.

To see the performance of our method compared to \texttt{MUL} and \texttt{IP}, in the following test, we extend the dimension of datasets in $\chi_1$, $\chi_2$, and $\chi_3$ as follows:
\myeqn{{\!\!\!}\begin{array}{ll}
{\chi_1(8)} &:= \set{ x_i = (e^{-s_i},s_ie^{-s_i},e^{-2s_i},s_ie^{-2s_i},e^{-3s_i},s_ie^{-3s_i},e^{-4s_i},s_ie^{-4s_i})^{\top},~~1\leq i\leq p } \subset \R^8, \vspace{0.75ex}\\
{\chi_2(10)} &:= \set{ x_i = (1,s_i,s_i^2,\cdots,s_i^9)^{\top},~~1\leq i\leq p } \subset\R^{10}, \vspace{0.75ex}\\
{\chi_3(10)} &:= \set{ x_{(i-1)\lceil \sqrt{p}\rceil+j} = ( 1,r_i,r_i^2,r_i^3,t_j,r_it_j,t_jr_i^2,t_j^2,t_j^3,r_it_j^2)^{\top},~~1\leq i,j\leq \lceil \sqrt{p}\rceil } \subset\R^{10}.
\end{array}{\!\!\!}}
We first run \texttt{MUL} up to $10,000$ iterations, and check whether the \texttt{HomoPN} and \texttt{IP} methods can solve the problem in terms of the function value. 
We terminate the algorithms when the objective value is less than the value procedured by \texttt{MUL}. 
When the \texttt{IP} method cannot achieve this objective value, we terminate it and record the time and objective value after at most $10,000$ iterations (this is a large number of iterations  for interior point methods). 
The computation results are {presented} in Table~\ref{tbl:perform_doptimal2}.
{For these larger dimensional} problems, \texttt{SDPT3} 
{fails to solve them to the required accuracy}, and we do not add this method to our experiments.

\begin{table}[ht!]
%\vspace{-2ex}
\newcommand{\cellb}[1]{{\color{blue}#1}}
\begin{center}
\caption{The performance of $3$ algorithms on \eqref{eq:D-ex-design} for $\chi_1(8)$, $\chi_2(10)$ and $\chi_3(10)$.}\label{tbl:perform_doptimal2}
\vspace{-1ex}		
%\resizebox{\textwidth}{!}{
%\begin{small}
\begin{tabular}{r|r|r|r|r|r|r|r}\toprule
		\multicolumn{2}{c|}{Problem} & \multicolumn{3}{c|}{\texttt{Time {[}s{]}}} & \multicolumn{3}{c}{Objective value $F(x^k)$}    \\ \midrule
		$\chi$        & $p~~~$             & \texttt{HomoPN}    &  \texttt{MUL\cite{harman2009approximate}} & \texttt{IP\cite{lu2013computing}} &  \texttt{HomoPN} & \texttt{MUL\cite{harman2009approximate}} & \texttt{IP\cite{lu2013computing}}  ~~\\ \midrule
		1             & 10000          & 0.203  & 9.172      & 57.830   & 92.552013  & 92.552097 & \cellb{96.377758} \\ %\midrule
		1             & 50000         & 1.134  & 48.288      & 264.658  & 92.541322 & 92.541593 & \cellb{96.373413} \\ %\midrule
		1             & 100000        & 2.589  & 97.433     & 468.258  & 92.539162  & 92.540231  & \cellb{96.372813} \\ \midrule
		2             & 10000         & 0.303  & 6.950     & 71.1327   & 18.424093  & 18.424563 & \cellb{22.387641} \\ %\midrule
		2             & 50000         & 1.555   & 43.340     & 229.223  & 18.416971 & 18.417293  & \cellb{22.388450} \\ %\midrule
		2             & 100000        & 2.862   & 84.389     & 543.588   & 18.416223  & 18.416356  & \cellb{22.389739} \\ \midrule
		3             & 10000         & 0.188    & 17.657     & 2.140    & 30.158301  & 30.158329 & {30.158325}  \\ %\midrule
		3             & 50000         & 0.822  & 72.969    & 240.569  & 29.956255  & 29.956513 & \cellb{37.689778} \\ %\midrule
		3             & 100000        & 2.217  & 165.105     & 556.789   & 29.889216 & 29.889565  & \cellb{37.692640}  \\ \bottomrule
\end{tabular}
%\end{small}%}
\end{center}	
%\vspace{-1ex}
\end{table}

From Table~~\ref{tbl:perform_doptimal2}, we can see that our algorithm outperforms {both the} \texttt{MUL} and \texttt{IP} methods in terms of {computation}
time, while the IP method cannot achieve the desired accuracy when 
{the dimension of the matrix variable}
$m$ is increased.
It also illustrates the ability of our method to solve the problems to an intermediate accuracy. 
The reason why the IP method cannot solve the problem is that when the dimension is large, the Newton system is very ill-condition, and {the solver cannot
handle the ill-conditioning difficultly satisfactorily.}
%So it is hard to move to next step, so the Ip method will not able to solve the problem at $\mu$ near zero. 
%This a common disadvantage for interior point method, and the reason why SDPT3 can not solve the problem neither. 

%%% 8.3. Non-Lipschitz gradient models.
\subsection{Non-Lipschitz gradient models}\label{eq:Poisson_exam}
Next, we illustrate  the variant of Algorithm~\ref{alg:A1} stated in Theorem  \ref{th:linear_convergence2b} under Assumption \ref{as:A4} through a non-Lipschitz gradient model.
We consider the following problem from Poisson regression \cite{ivanoff2016adaptive,jia2017sparse}:
\myeq{eq:p-reg-exam}{
F^{\star} := \min_{x\in\R^p}\Big\{ F(x) := \frac{1}{n}\sum_{i=1}^n\left(y_i\exp\left(-\tfrac{a_i^{\top}x}{2}\right) + \exp\left(\tfrac{a_i^{\top}x}{2}\right)\right) + \frac{\mu_f}{2}\norms{x}^2_2 +  \rho\norms{x}_1 \Big\},
}
where $a_i\in\R^p$, $y_i\in \R$ are given, and $\mu_f > 0$ and $\rho > 0$ are given regularization parameters.

Let $f(x) := \frac{1}{n}\sum_{i=1}^n\left(y_i\exp\left(-\tfrac{a_i^{\top}x}{2}\right) + \exp\left(\tfrac{a_i^{\top}x}{2}\right)\right) + \frac{\mu_f}{2}\norms{x}^2_2$.
Note that if $\mu_f > 0$, then by Proposition \ref{pro:self_con_class} 
{in Appendix A, $f$ is $M_f$-self-concordant} with $M_f := \frac{\max_i\set{\Vert a_i\Vert_2 }}{\sqrt{\mu_f}}$. 
Moreover, it is also $\mu_f$-strongly convex, and $g(x) := \rho\norms{x}_1$ is $L_g$-Lipschitz continuous with $L_g := \rho$.
Hence, problem \eqref{eq:p-reg-exam} satisfies Assumption \ref{as:A4}. Consequently, the results of Theorem \ref{th:linear_convergence2b} hold for this problem.

We implement Algorithm~\ref{alg:A1} to solve \eqref{eq:p-reg-exam} and compare {it} with a quasi-Newton method using BFGS with linesearch scheme and a limited-memory quasi-Newton method (L-BFGS) with both linesearch and restarting scheme. 
We name these three schemes \texttt{HomoPN}, \texttt{BFGS}, and \texttt{L-BFGS-LS-R}, respectively.
Similar to the above example, we {adapt} the semi-smooth Newton-CG augmented Lagrangian method  in \cite{li2018highly,zhao2010newton} to approximately 
{solve the convex composite quadratic programming subproblem \eqref{eq:cvx_subprob_k}
 in \texttt{HomoPN}.
 }

Note that since the gradient $\nabla{f}$  is not Lipschitz continuous, first-order methods such as proximal gradient-type methods are not applicable. 
Our experiment reveals that first-order methods such as accelerated  Barzilai-Borwein step-size  proximal gradient algorithms  often failed to converge due to the explosion of the estimated Lipschitz constant of $\nabla{f}$. 

We test three algorithms on $6$ datasets downloaded from the UCI dataset repository \cite{lichman2017uci} and Kaggle (\href{https://www.kaggle.com/datasets}{https://www.kaggle.com/datasets}).
Such datasets are used to predict {the number of communications} for a social post (news and facebook), the review score of a hotel or wine (vegas and wine), 
{the number of goals scored in a game season or the number of 911 calls} in a period of time.
The sizes of these datasets are given in Table~\ref{tbl:dataset_detail}.
Since these are raw datasets, to use them in our model \eqref{eq:p-reg-exam}, we perform a sequence of data-preprocessing routines to make our feature matrices well scaled by using methods {such as} one-hot encoding and min-max scaling. 
%
%For our Poisson regression model, we have performed a sequence of data-preprocessings to make our feature matrices well scaled by using methods like one-hot encoding and min-max scaling. Our tested datasets are found on UCI dataset repository \cite{lichman2017uci} and Kaggle (see \href{url}{https://www.kaggle.com/datasets}). We summarize the basic information of our six datasets, see in Table \ref{table:dataset_detail}. These datasets are used to predict a number of communication for a social post (news and facebook), the review score of a hotel or a wine (vegas and wine), a number of goals of a season of a game or a number of call to 911 in a period of time.

\begin{table}[ht!]
\begin{center}
\vspace{-3ex}
\caption{Poisson datasets information.}\label{tbl:dataset_detail}
\vspace{-1ex}
	\begin{tabular}{|c|c|c|c|c|c|c|}
		\hline
		Dataset     & vegas & games & news  & 911   & facbook & wine   \\ \hline
		\# {samples}    & 504   & 4940  & 39644 & 33059 & 99003   & 150930 \\ \hline
		\# features & 90    & 156   & 59    & 21    & 254     & 1154   \\ \hline
	\end{tabular}
\end{center}
\vspace{-2ex}
\end{table}

\begin{table}[ht!]
\begin{center}
\vspace{-3ex}
\caption{The performance and results of three algorithms on the Poisson model \eqref{eq:p-reg-exam}.}\label{tbl:perform_poisson}
\vspace{-1ex}
	\begin{tabular}{|c|c|c|c|c|c|c|c|}
		\hline
		\multirow{2}{*}{Datasets} & \multicolumn{3}{c|}{Time [s]}    & \multicolumn{3}{c|}{Number of iterations} & \multirow{2}{*}{$F(x^k)$} \\ \cline{2-7}
		& \texttt{HomoPN} & \texttt{BFGS}   & \texttt{L-BFGS-LS-R} & \texttt{HomoPN} & \texttt{BFGS} & \texttt{L-BFGS-LS-R} &                      \\ \hline
		vegas                     & 0.04 & 0.23   & 0.21        & 5    & 11   & 11          & 4.2695e+00           \\ \hline
		games                     & 0.12 & 1.34   & 0.99        & 6    & 22   & 22          & 4.5610e+00           \\ \hline
		news                      & 1.82 & 52.94  & 48.57       & 9    & 220  & 215         & 1.1447e+02           \\ \hline
		911                       & 0.10 & 0.87   & 0.55        & 6    & 15   & 15          & 8.2290e+00           \\ \hline
		facebook                  & 1.54 & 62.36  & 64.68       & 8    & 39   & 48          & 4.6651e+01           \\ \hline
		wine                      & 2.89 & 455.28 & 459.81      & 6    & 40   & 40          & 1.9819e+01           \\ \hline
	\end{tabular}
\end{center}
\vspace{-1ex}	
\end{table}

We test three algorithms: \texttt{HomoPN}, \texttt{BFGS}, and \texttt{L-BFGS-LS-R} on the datasets in Table \ref{tbl:dataset_detail}.
The computation results are reported in Table \ref{tbl:perform_poisson}, where $F(x^k)$ is the objective value of \eqref{eq:p-reg-exam}. 
We can see from the results that while all methods we {have} tested can solve the problem using real datasets (they all reached {almost the same} optimal function values), the performance of {the} two quasi-Newton methods are similar and our homotopy method outperforms them both in terms of computational times and number of iterations. 
Moreover, based on our observation, when the problem is hard (i.e. the linear system is very ill-conditioned), the number of iterations in the quasi-Newton-type methods increases rapidly. However, our homotopy method is relatively stable in terms of the number of iterations, as was indicated in Section \ref{subsec:global_linear_conv}. 
{To} conclude, our method is highly efficient for real datasets.

%%%% 8.2. Initial point independence of the new homotopy methods. 
\subsection{The initial point independence of Algorithm~\ref{alg:A1b}}\label{subsec:global_linear_conv}
Our goal in this example is to observe the independence of performance w.r.t. to the initial point $X_0$ in our inexact primal-dual-primal homotopy PN scheme, i.e. Algorithm~\ref{alg:A1b}, compared to standard proximal-Newton methods for sparse optimization.
We demonstrate this observation on the sparse inverse covariance estimation problem covered by \eqref{eq:GL_prob} which perfectly fits our assumptions.

As mentioned, in theory, we have shown that Algorithm~\ref{alg:A1b} can achieve a global linear convergence rate. 
In practice, however, this rate may {still be} slow. 
We instead update the homotopy parameter $\tau$ in Algorithm~\ref{alg:A1} by a longer step based on a linesearch such that the new {iterate} $X_k$ remains in $\dom{\phi}$.
This trick allows us to go faster from $\tau\approx 0$ to $\tau\approx 1$ within a few iterations instead of using the worst-case {factor}.
Then, we fix the homotopy parameter $\tau_k \approx 1$ and apply a few full-step proximal-Newton iterations to compute the desired solution (Stage 3 of Algorithm~\ref{alg:A1}).

We implement Algorithm~\ref{alg:A1b} and compare it with the primal proximal-Newton method in \cite{Tran-Dinh2013a}.
We abbreviate these methods by \texttt{HomoPN} and \texttt{Primal PN}, respectively.
We use a restarting accelerated proximal-gradient algorithm to approximate the proximal Newton direction.
We generate two problem instances using the same procedure as in \cite{Li2010} with $p= 500$ and $p=1000$, respectively.
In order to obtain a desired sparse solution of \eqref{eq:GL_prob}, we choose $\rho := 0.01$.
To see the affect of the initial point $X_0$ {on} the methods we generate two different initial points $X_0$.
\begin{itemize}
\item\textit{Case 1 (dense)}: $X_0 := \Sigma^{\dagger} + 10^{-6}\times \Id$, where $\Sigma^{\dagger}$ is the pseudo-inverse of $\Sigma$.
\item\textit{Case 2 (sparse)}:  $X_0 := \mathrm{diag}(1./\mathrm{diag}(\Sigma))$ a diagonal matrix. 
\end{itemize}
Figure \ref{fig:linear_convergence} shows the convergence of \texttt{HomoPN} and \texttt{Primal PN}  for these two cases: the sparse initial point $X_0$ and the dense initial point $X_0$.

\begin{figure}[ht!]
\vspace{-1ex}
\begin{center}
\includegraphics[width = 1\textwidth]{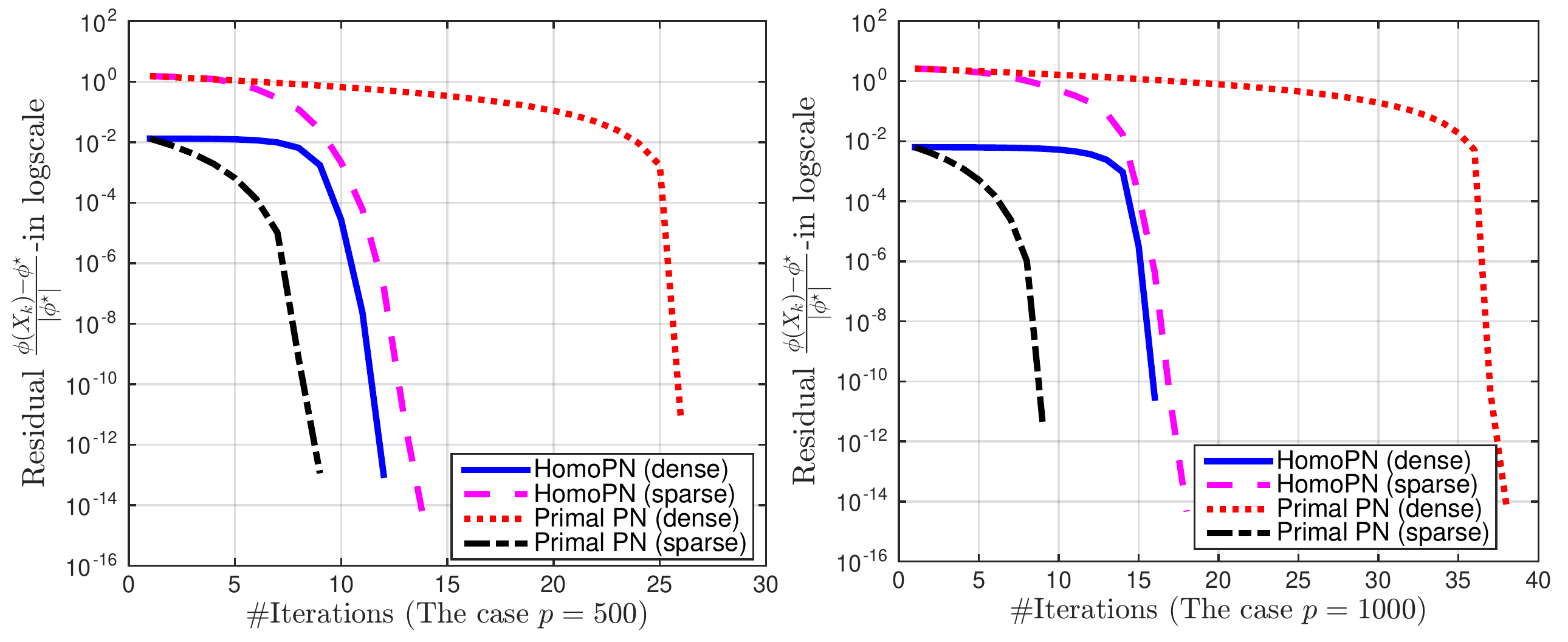}
\vspace{-3ex}
\caption{The convergence of the two algorithms. Left: For $p=500$, and Right: For $p = 1000$.}\label{fig:linear_convergence}
\end{center}
\vspace{-4ex}
\end{figure}

Figure~\ref{fig:linear_convergence} shows that our new algorithmic variant (\texttt{HomoPN}) {depends weakly} on the initial point $X_0$, while the primal proximal Newton method (\texttt{Primal PN}) in \cite{Tran-Dinh2013a} strongly depends on the choice of $X_0$.
The convergence rate of \texttt{HomoPN} is divided into two parts: the homotopy {part} with a linear rate, and the refining part with a quadratic rate.
 Figure \ref{fig:linear_convergence} shows {the} quadratic convergence on the relative objective residual $\frac{\phi(X_k) - \phi^{\star}}{\vert\phi^{\star}\vert}$ of both algorithms
 {when the iterates approach the optimal solution}, but \texttt{Primal PN (dense)} requires a large number of iterations to reach its quadratic convergence region.

%%%% 8.5. The performance of the primal-dual-primal  approach.
%\subsection{The performance of the primal-dual-primal  approach}\label{subsec:SICE_dual_approach}
%The third experiment is to compare our method and recent state-of-the-arts for solving the sparse inverse covariance estimation problem \cite{Friedman2008}.
%We choose two methods: the primal proximal Newton (primal PN) method in \cite{Tran-Dinh2013a} and \texttt{Quic} in \cite{Hsieh2011}.
%To the best of our knowledge, \texttt{Quic} is the most efficient method for solving this problem when the solution is sparse.
%
%\liang{We can illustrate this on the inverse covariance problem but with different regularizers than that of the $\ell_1$-norms}

%%% Acknowledgements.
\begin{footnotesize}
\section*{Acknowledgments}
This work is supported in part by the NSF grant, no. DMS-16-2044 (USA).
\end{footnotesize}
%%% End of Acknowledgments.

%%%% Appendix 1.1.
\appendix
\section{Self-concordant and generalized self-concordant functions}\label{sec:apdx:proofs}
\normalsize
Our central concept is the self-concordance and self-concordant barrier introduced by Nesterov and Nemirovskii \cite{Nesterov2004,Nesterov1994}, and extensions to generalized self-concordance studied in \cite{SunTran2017gsc}.

For a three-time continuously differentiable and convex function $f$,  let $\nabla^3{f}(x)[u]$ denote the third-order derivative along the direction $u$.
We denote this class of functions by $\Cc^3(\dom{f})$.

%% Definition 2.2.
\begin{definition}\label{de:gen_sel_con_def}
For $f : \dom{f}\subseteq \R^{p} \to \R$ in $\Cc^3(\dom{f})$, we say that:
\begin{itemize}
\item[$\mathrm{(a)}$] $f$ is $M_f$-self-concordant $($$M_f \geq 0$$)$ if for any $x\in\dom{f}$ and $u\in\R^p$, we have
\myeqn{
\vert \iprods{\nabla^3{f}(x)[u]u, u}\vert \leq M_f\iprods{\nabla^2{f}(x)u, u}^{3/2}.
}
If $M_f = 2$, then we say that $f$ is standard self-concordant.
\item[$\mathrm{(b)}$] $f$ is $(M_f, \kappa)$-generalized self-concordant if for any $x\in\dom{f}$ and $u, v\in\R^p$, we have
\myeqn{
\vert \iprods{\nabla^3{f}(x)[u]v, v}\vert \leq M_f\norms{v}_x^2\norms{u}_x^{\kappa-2}\norms{u}_2^{3-\kappa},
}
where if $0 \leq \kappa < 2$ or $\kappa > 3$, then we use {the} convention $\frac{0}{0} = 0$.
\item[$\mathrm{(c)}$]
$f$ is a $\nu_f$-self-concordant barrier $(\nu_f \geq 1)$ if $f$ is standard self-concordant with $\dom{f} = \intx{\Xc}$, $f(x)$ tends to $+\infty$ as $x$ approaches the boundary $\partial{\Xc}$ of $\Xc$, and
\myeqn{
\sup_{u \in\R^p} \left \{2\iprods{\nabla{f}(x), u} - \Vert u \Vert_{x}^2\right \} \leq 
\nu_f, ~~\forall x \in\dom{f}.
}
\item[$\mathrm{(d)}$] $f$ is {a} $\nu_f$-self-concordant \textit{logarithmically homogeneous  barrier} function of $\Xc$ if {it} is self-concordant barrier and $f(\tau x) = f(x) - \nu_f\log(\tau)$ for all $x\in\intx{\Xc}$ and $\tau > 0$.
\end{itemize}
\end{definition}

It is clear that affine and convex quadratic functions are standard self-concordant but not self-concordant barrier.
The logistic, Poisson, and DWD models discussed in the introduction are  generalized self-concordant with $\kappa\in [2,3]$, but not self-concordant.
Self-concordant barriers are often associated with convex sets such as cones and convex bodies.
Several simple sets are equipped with a  self-concordant barrier.
For instance, $f(x) := -\sum_{i=1}^p\log(x_i)$ is a $p$-self-concordant barrier of $\R^p_{+}$, $f(X) := -\log\det(X)$ is a $p$-self-concordant barrier of  $\Spd^p_{+}$, and $f(x, t) = -\log(t^2 - \iprods{x,x})$ is a $2$-self-concordant barrier of the Lorentz cone $\mathcal{L}_{p+1} := \set{(x, t) \in\R^p\times\R_{+} \mid \norm{x} \leq t}$.
The relation between self-concordant and generalized self-concordant is stated in the following proposition.
%%% Propostion 10.
\begin{proposition}\cite{SunTran2017gsc}\label{pro:self_con_class}
Let $f\in\Cc^3(\dom{f})$ be an $(M_f,\kappa)$-generalized self-concordant and $\kappa \in (0, 3]$ as defined in Definition~\ref{de:gen_sel_con_def}.
Then, if $f$ is strongly convex with a strong convexity parameter $\mu_f > 0$, then $f$ is $\widehat{M}_f$-self-concordant with $\widehat{M}_f := \frac{M_f}{(\sqrt{\mu_f})^{3-\kappa}}$.
\end{proposition}
%%% End of proposition.
Proposition~\ref{pro:self_con_class} shows that the regularized logistic, Poisson, and DWD models in the introduction become self-concordant due to the aid of a regularization term $\frac{\mu_f}{2}\norms{x}^2$.

Let $f$ be a $\nu_f$-self-concordant barrier of $\Xc$. 
Then
\myeq{eq:analytical_center}{
x^{\star}_f := \argmin_x\set{ f(x) \mid x\in\intx{\Xc} },
}
is referred to as the \textit{analytical center} of $\Xc$.
If $\Xc$ is bounded, then $x^{\star}_f$ exists and is unique. 
Otherwise, we often add an artificial bound or consider $f$ on a sublevel set of $F$ in \eqref{eq:composite_cvx} to guarantee the existence of $x_f^{\star}$.
Let $\theta_f := \nu_f + 2\sqrt{\nu_f}$ be {defined} for a general self-concordant barrier, and $\theta_f := 1$ be {defined} for a self-concordant logarithmically homogeneous barrier.
Then, $\Vert v \Vert_{x}^{*} \leq \theta_f\Vert v \Vert_{x_f^{\star}}^{*} $ and $\Vert x - x^{\star}_f\Vert_{x^{\star}_f} \leq \theta_f$ for $x\in\dom{f}$ and $v\in\R^p$ if $x^{\star}_f$ exists.

Let $\Kc$ be a proper, closed, pointed, and convex cone.
If $\Kc$ is endowed with a $\nu_f$-self-concordant logarithmically homogeneous  barrier function $f$, then its Legendre transformation \cite{Nesterov1994}:
\myeqn{
f^{*}(y) := \sup_{v}\set{ \iprods{y, v} - f(v) \mid v\in\Kc},
}
is also a $\nu_f$-self-concordant logarithmically homogeneous  barrier of the anti-dual cone $-\Kc^{*}$ of $\Kc$.
For instance, if $\Kc = \Spd_{+}^p$, then $\Kc^{*} = S_{+}^p = \Kc$ (self-dual cone).
A barrier function of $\Spd_{+}^p$ is $f(X) := -\log\det(X)$. Hence, $f^{*}(Y) = -p - \log\det(-Y)$ is a barrier function of $-\Kc^{*}$.

%%% B. Convergence analysis of the homotopy variable-metric methods.
\section{Convergence and iteration-complexity analysis of Algorithm \ref{alg:A1}}\label{apdx:convergence_analysis}
We break {up} our analysis into several lemmas and theorems.
The following  lemmas provide several key estimates for our proofs.

%%% Lemma 11.
\begin{lemma}\label{le:auxi_results}
Let $x^0 \in\dom{F}$  be given, $\xi^0\in\partial{g}(x^0)$, and $L_g$ be the Lipschitz constant of $g$.
Let $x^{\ast}_{\tau}$ and $x^{\ast}_{\hat{\tau}}$ be two solutions of  \eqref{eq:new_reparam_opt_cond} at $\tau \in (0,1]$ and $\hat{\tau}\in (0,1]$, respectively. 
Then {the following results hold.}
\begin{itemize}
\item[$\mathrm{(a)}$] 
If $f$ is $\mu_f$-strongly convex, then we have
\myeq{eq:lm11_bound_fx0_b1}{
\begin{array}{ll}
& \norms{x^0 - x^{\ast}_{\tau}}_2 \leq \tfrac{\norms{\nabla{f}(x^0) + \xi^0}_2}{\mu_f}~~~\text{and}\vspace{1ex}\\
& \Vert x^{\ast}_{\tau} - x^{\ast}_{\hat{\tau}}\Vert_2\leq \frac{\vert\tau - \hat{\tau}\vert}{\hat{\tau}\, \mu_f}\Vert \nabla{f}(x^{\ast}_{\tau}) + \xi^0\Vert_2 \leq \frac{\vert\tau - \hat{\tau}\vert}{\tau\,\hat{\tau}\, \mu_f}\left( L_g + \norms{\xi^0}_2\right).
\end{array}}
\item[$\mathrm{(b)}$] If $f$ is self-concordant, then we have
\myeq{eq:lm11_bound_fx0_b2}{
\begin{array}{ll}
& \tfrac{\norms{x^0 - x^{\ast}_{\tau}}_{x^{\ast}_{\tau}}}{1 ~+~ \norms{x^0 - x^{\ast}_{\tau}}_{x^{\ast}_{\tau}}} \leq  \norms{\nabla{f}(x^0) + \xi^0}_{x^{\ast}_{\tau}}^{\ast}~~\text{and}~~\vspace{1ex}\\
&\frac{\Vert x^{\ast}_{\tau} - x^{\ast}_{\hat{\tau}}\Vert_{x^{\ast}_{\tau}}}{1 ~+~ \Vert x^{\ast}_{\tau} - x^{\ast}_{\hat{\tau}}\Vert_{x^{\ast}_{\tau}}} \leq \frac{\abs{\tau - \hat{\tau}}}{\hat{\tau}}\Vert \nabla{f}(x^{\ast}_{\tau}) + \xi^0 \Vert_{x^{\ast}_{\tau}}^{\ast} \leq \frac{2L_g\abs{\tau - \hat{\tau}}}{\tau\hat{\tau}}.
\end{array}}
\item[$\mathrm{(c)}$]
Let $x^{\ast}_t$ and $x^{\ast}_{\hat{t}}$  be two solutions of  \eqref{eq:opt_cond_auxi} at $t \in (0, 1]$ and $\hat{t} \in (0, 1]$, respectively.
If $f$ is self-concordant, then 
\myeq{eq:lm11_bound_fx0_b4}{{\!\!\!}
\tfrac{\Vert x^{\ast}_t - x^0\Vert_{x^0}}{1 ~+~ \Vert x^{\ast}_t - x^0\Vert_{x^0}} \leq \vert t-1\vert\norms{\nabla{f}(x^0) + \xi^0}_{x^0}^{\ast} ~\text{and}~
\tfrac{\Vert x^{\ast}_t - x^{\ast}_{\hat{t}}\Vert_{x^{\ast}_t}}{1 ~+~ \Vert x^{\ast}_t - x^{\ast}_{\hat{t}}\Vert_{x^{\ast}_t}} \leq \vert t-\hat{t}\vert\norms{\nabla{f}(x^0) + \xi^0}_{x^{\ast}_t}^{\ast}.
{\!\!\!\!\!}}
\item[$\mathrm{(d)}$] 
Let $x^{\ast}_s$ be a solution of either \eqref{eq:new_reparam_opt_cond} or \eqref{eq:opt_cond_auxi}.
If $f$ is self-concordant, and $x^0$ and $s\in (0,1)$ are chosen such that $\norms{x^0 - x^{\ast}_s}_{x^0} \leq \gamma < 1$, then $\norms{x^0 - x^{\ast}_s}_{x^{\ast}_s} \leq \frac{\gamma}{1-\gamma}$.
\item[$\mathrm{(e)}$]
If $g$ is $\mu_g$-strongly convex, and $x^{\ast}_{\tau}$ is a solution of \eqref{eq:new_reparam_opt_cond}, then
\myeq{eq:lm11_bound_fx0_b3}{
\Vert x^0 - x^{\ast}_{\tau}\Vert_2 \leq \tfrac{\tau}{\mu_g}\Vert\nabla{f}(x^0) + \xi^0\Vert_2.
}
\end{itemize}
\end{lemma}

%%% The proof of Lemma 11.
\begin{proof}
(a)~From \eqref{eq:new_reparam_opt_cond}, we have $-\tau (\nabla{f}(x^{\ast}_{\tau}) + \xi^0) \in\partial{g}(x^{\ast}_{\tau}) - \xi^0$.
Using the monotonicity of $\partial{g}$ and $\xi^0\in\partial{g}(x^0)$, we have $\iprods{\nabla{f}(x^{\ast}_{\tau}) +   \xi^0, x^{\ast}_{\tau} - x^0} \leq 0$, which implies
\myeqn{
\iprods{\nabla{f}(x^{\ast}_{\tau}) - \nabla{f}(x^0), x^{\ast}_{\tau} - x^0} \leq \iprods{\nabla{f}(x^0) + \xi^0, x^0 - x^{\ast}_{\tau}}. 
}
If $f$ is $\mu_f$-strongly convex, then  using the Cauchy-Schwarz inequality, the last inequality leads to
\myeqn{
\mu_f\norms{x^0 - x^{\ast}_{\tau}}_2^2 \leq  \iprods{\nabla{f}(x^{\ast}_{\tau}) - \nabla{f}(x^0), x^{\ast}_{\tau} - x^0}  \leq \norms{\nabla{f}(x^0) + \xi^0}_2\norms{x^0 - x^{\ast}_{\tau}}_2, 
}
which can be simplified as $\norms{x^0 - x^{\ast}_{\tau}}_2 \leq \tfrac{\norms{\nabla{f}(x^0) + \xi^0}_2}{\mu_f}$.
This is exactly the first estimate of \eqref{eq:lm11_bound_fx0_b1}.

Using \eqref{eq:new_reparam_opt_cond} again with $\hat{\tau}$ we have $-\hat{\tau}( \nabla{f}(x^{\ast}_{\hat{\tau}}) +  \xi^0) \in\partial{g}(x^{\ast}_{\hat{\tau}}) - \xi^0$.
Combining this expression and $-\tau( \nabla{f}(x^{\ast}_{\tau}) + \xi^0) \in\partial{g}(x^{\ast}_{\tau}) - \xi^0$, and using the monotonicity of $\partial{g}$, we have
\myeqn{
\iprods{\tau\nabla{f}(x^{\ast}_{\tau}) - \hat{\tau}\nabla{f}(x^{\ast}_{\hat{\tau}}) + (\tau-\hat{\tau})\xi^0, x^{\ast}_{\tau} - x^{\ast}_{\hat{\tau}} } \leq 0.
}
Rearranging this inequality, we obtain
\myeqn{
\hat{\tau}\iprods{\nabla{f}(x^{\ast}_{\tau}) - \nabla{f}(x^{\ast}_{\hat{\tau}}), x^{\ast}_{\tau} - x^{\ast}_{\hat{\tau}} } \leq (\hat{\tau} -  \tau)\iprods{\nabla{f}(x^{\ast}_{\tau}) + \xi^0, x^{\ast}_{\tau} - x^{\ast}_{\hat{\tau}}}.
}
If $f$ is $\mu_f$-strongly convex, then we have $\iprods{\nabla{f}(x^{\ast}_{\tau}) - \nabla{f}(x^{\ast}_{\hat{\tau}}), x^{\ast}_{\tau} - x^{\ast}_{\hat{\tau}} } \geq \mu_f\norms{x^{\ast}_{\tau} - x^{\ast}_{\hat{\tau}} }_2^2$. 
Combining this estimate {with} the last inequality, and then using the Cauchy-Schwarz inequality, we obtain the first inequality of the second line of \eqref{eq:lm11_bound_fx0_b1}.

 Using again \eqref{eq:new_reparam_opt_cond}, we have $-\tau (\nabla{f}(x^{\ast}_{\tau}) + \xi^0) \in\partial{g}(x^{\ast}_{\tau}) - \xi^0$.
Since $g$ is $L_g$-Lipschitz continuous, the last expression leads to $\tau\Vert \nabla{f}(x^{\ast}_{\tau}) +  \xi^0\Vert_2 \leq L_g + \norms{\xi^0}_2$, which implies \eqref{eq:lm11_bound_fx0_b1}.
 
(b)~If $f$ is self-concordant, then we have 
\myeqn{
\tfrac{\norms{x^0 - x^{\ast}_{\tau}}_{x^{\ast}_{\tau}}^2}{1 ~+~ \norms{x^0 - x^{\ast}_{\tau}}_{x^{\ast}_{\tau}}} \leq \iprods{\nabla{f}(x^{\ast}_{\tau} - \nabla{f}(x^0), x^{\ast}_{\tau} - x^0}.
}
With a similar proof as in (a), we have 
\myeqn{\tfrac{\norms{x^0 - x^{\ast}_{\tau}}_{x^{\ast}_{\tau}}}{1 ~+~ \norms{x^0 - x^{\ast}_{\tau}}_{x^{\ast}_{\tau}}} \leq  \norms{\nabla{f}(x^0) + \xi^0}_{x^{\ast}_{\tau}}^{\ast},} 
which is the first line of \eqref{eq:lm11_bound_fx0_b2}.
The second line of \eqref{eq:lm11_bound_fx0_b2} is proved similarly {as that of} \eqref{eq:lm11_bound_fx0_b1} and using the fact that $\norms{\xi^0}_{x^{\ast}_{\tau}}^{\ast} \leq L_g$.

(c)~The proof of \eqref{eq:lm11_bound_fx0_b4} is very similar to the proof of \eqref{eq:lm11_bound_fx0_b2} and we omit it here.

(d)~Note that if $f$ is self-concordant, then we have $\norms{x^0 - x^{\ast}_s}_{x^{\ast}_s} \leq \frac{\norms{x^0 - x^{\ast}_s}_{x^0}}{1 - \norms{x^0 - x^{\ast}_s}_{x^0}}$ due to \cite[Theorem 4.1.5]{Nesterov2004} as long as $\norms{x^0 - x^{\ast}_s}_{x^0} < 1$.
If $\norms{x^0 - x^{\ast}_s}_{x^0} \leq \gamma < 1$, then the last inequality implies 
{that} $\norms{x^0 - x^{\ast}_s}_{x^{\ast}_s} \leq \frac{\gamma}{1-\gamma}$, which proves (d).

(e)~By the choice of $\xi^0$, we have $0 \in \partial{g}(x^0) - \xi^0$. 
From \eqref{eq:new_reparam_opt_cond}, we also have 
\myeqn{-\tau(\nabla{f}(x^{\ast}_{\tau}) - \nabla{f}(x^0)) -\tau(\nabla{f}(x^0) + \xi^0) \in \partial{g}(x^{\ast}_{\tau}) - \xi^0.}
Using the $\mu_g$-strong monotinicity of $\partial{g}$ and the monotonicity of $\nabla{f}$, we have 
\myeqn{
\iprods{\tau(\nabla{f}(x^0) + \xi^0), x^0 - x^{\ast}_{\tau}} \geq \mu_g\norms{x^{\ast}_{\tau} - x^0}_2^2.
}
By the Cauchy-Schwarz inequality, we obtain $\norms{x^0 - x^{\ast}_{\tau}}_2 \leq \frac{\tau}{\mu_g}\norms{\nabla{f}(x^0) + \xi^0}_2$, which is {the desired result in} \eqref{eq:lm11_bound_fx0_b3}.
\end{proof}
%%% End of the proof.

%%% Lemma 11b.
\begin{lemma}\label{le:contraction_sequence}
Let $M > 0$ and $q\in (0, 1)$ be two given constants. 
Let $\set{\tau_k}\subset (0, 1)$ be a given sequence such that $\tau_{k+1} \leq \tau_k + Mq^k\tau_k\tau_{k+1}$ for all $k\geq 0$.
Then
\myeq{eq:tau_estimates}{
\tau_k \leq \frac{\tau_0(1-q)}{1-q -\tau_0M(1-q^k)} ~~~~~\text{and}~~~~1 - \tau_k \geq \frac{(1-\tau_0)(1-q) - \tau_0M + \tau_0Mq^k}{1 - q - \tau_0M(1-q^k)},
}
as long as the denominators are well-defined.
\end{lemma}

%%% The proof of Lemma 11b.
\begin{proof}
Define $s_k := \frac{1}{\tau_k}$. Then, $\tau_{k+1} \leq \tau_k + Mq^k\tau_k\tau_{k+1}$ is equivalent to $s_{k+1} \geq s_k - Mq^k$.
By induction, we have $s_k \geq s_0 - M\sum_{i=0}^{k-1}q^i = s_0 - \frac{M(1-q^k)}{1-q}$.
Hence, we can show that $\tau_k \leq \frac{\tau_0(1-q)}{1-q -\tau_0M(1-q^k)}$. This estimate leads to
\myeqn{
1 - \tau_k \geq \frac{(1-\tau_0)(1-q) - \tau_0M + \tau_0Mq^k}{1 - q - \tau_0M(1-q^k)},
}
which proves \eqref{eq:tau_estimates}.
\end{proof}
%%% End of the proof.

%%% Lemma 12.
\begin{lemma}\label{le:contraction_of_prox_grad}
Assume that $f$ is $\mu_f$-strongly convex and $L_f$-smooth. {Suppose $0 < m \leq L < +\infty$ are given parameters} such that $\omega := \frac{1}{m}\sqrt{(L-2\mu_f)m + L_f^2} < 1$. 
Let $\set{x^k}$ be the sequence generated by \eqref{eq:homotopy_method} using $H_k\in\Sc^p_{++}$ such that $m\Id \preceq H_k\preceq L\Id$. 
Then
\myeq{eq:lm12_proof_est2}{
\norms{x^{k+1} - x^{\ast}_{\tau_{k+1}}}_2 \leq \omega\norms{x^k - x^{\ast}_{\tau_{k+1}}}_2.
\vspace{1ex}
}
\end{lemma}

%%% Begin of the proof of Lemma 12.
\begin{proof}
Using $\nabla{f}_{\tau_{k+1}}(\cdot) := \nabla{f}(\cdot) - (\frac{1}{\tau_{k+1}}-1)\xi^0$, $H = H_k$, and $\tau = \tau_{k+1}$ in \eqref{eq:fixed_point1}, we get $x^{\ast}_{\tau_{k+1}} = \prox_{\frac{1}{\tau_{k+1}}g}^{H_k}\big(x^{\ast}_{\tau_{k+1}} - H_k^{-1}\nabla{f_{\tau_{k+1}}}(x^{\ast}_{\tau_{k+1}})\big)$.
Combining this expression and \eqref{eq:homotopy_method}, then using the non-expansiveness of $\prox_{\frac{1}{\tau_{k+1}}g}^{H_k}$ in \eqref{eq:nonexpansiveness}, we can derive
{that}
\vspace{-0.75ex}
\begin{eqnarray}\label{eq:lm12_proof_est1}
\norms{x^{k+1} \!-\! x^{\ast}_{\tau_{k+1}}}_{H_k}^2 
& = & \big\Vert \prox_{\frac{1}{\tau_{k+1}}g}^{H_k}\left(x^k \!-\! H_k^{-1}\nabla{f_{\tau_{k\!+\!1}}}(x^k)\right)  - \prox^{H_k}_{\frac{1}{\tau_{k+1}}g}\big( x^{\ast}_{\tau_{k+1}} {\!\!\!}- H_k^{-1}\nabla{f_{\tau_{k\!+\!1}}}(x^{\ast}_{\tau_{k+1}}) \big) \big\Vert_{H_k}^2 \nonumber\\
&\leq& \big\Vert x^k - x^{\ast}_{\tau_{k+1}} - H_k^{-1}\big(\nabla{f_{\tau_{k+1}}}(x^k) -  \nabla{f_{\tau_{k+1}}}(x^{\ast}_{\tau_{k+1}})\big) \big\Vert_{H_k}^2 \nonumber\\
&= &\norms{x^k - x^{\ast}_{\tau_{k+1}}}_{H_k}^2 - 2\big\langle{x^k - x^{\ast}_{\tau_{k+1}}, \nabla{f_{\tau_{k+1}}}(x^k) -  \nabla{f_{\tau_{k+1}}}(x^{\ast}_{\tau_{k+1}})}\big\rangle \nonumber\\
& &+~ \big(\norms{\nabla{f_{\tau_{k+1}}}(x^k) -  \nabla{f_{\tau_{k+1}}}(x^{\ast}_{\tau_{k+1}})}_{H_k}^{\ast}\big)^2 \nonumber\\
&=& \norms{x^k - x^{\ast}_{\tau_{k+1}}}_{H_k}^2 - 2\iprods{x^k - x^{\ast}_{\tau_{k+1}}, \nabla{f}(x^k) -  \nabla{f}(x^{\ast}_{\tau_{k+1}})} \nonumber 
\\
&& +~ \big(\norms{\nabla{f}(x^k) -  \nabla{f}(x^{\ast}_{\tau_{k+1}})}_{H_k}^{\ast}\big)^2.
\vspace{-0.75ex}
\end{eqnarray}
Next, {from the fact that}  $m\Id \preceq H_k \preceq L\Id$ and the strong convexity and smoothness of $f$, we have
\myeqn{
\begin{array}{ll}
&\norms{\nabla{f}(x^k) -  \nabla{f}(x^{\ast}_{\tau_{k+1}})}_{H_k}^{\ast} \leq \frac{1}{\sqrt{m}}\norms{\nabla{f}(x^k) -  \nabla{f}(x^{\ast}_{\tau_{k+1}})}_2 
\leq \frac{L_f}{\sqrt{m}}\norms{x^k - x^{\ast}_{\tau_{k+1}}}_2,\vspace{1.5ex}\\
&\iprods{\nabla{f}(x^k) -  \nabla{f}(x^{\ast}_{\tau_{k+1}}), x^k - x^{\ast}_{\tau_{k+1}}} \geq \mu_f\norms{x^k - x^{\ast}_{\tau_{k+1}}}_2^2.
\end{array}
}
Substituting these estimates into \eqref{eq:lm12_proof_est1}, we {get}
\myeqn{
\norms{x^{k+1} - x^{\ast}_{\tau_{k+1}}}_{H_k}^2 \leq \norms{x^k - x^{\ast}_{\tau_{k+1}}}_{H_k}^2 - 2\mu_f\norms{x^k - x^{\ast}_{\tau_{k+1}}}_2^2 + \frac{L_f^2}{m}\norms{x^k - x^{\ast}_{\tau_{k+1}}}_2^2.
}
Using again $m\Id \preceq  H_k \preceq L\Id$, the last inequality leads to
\myeqn{
\begin{array}{ll}
m\norms{x^{k+1}  - x^{\ast}_{\tau_{k+1}}}_2^2 &\leq \norms{x^{k+1}  -  x^{\ast}_{\tau_{k+1}}}_{H_k}^2 \vspace{1ex}\\
& \leq L\norms{x^k - x^{\ast}_{\tau_{k+1}}}_2^2 - 2\mu_f\norms{x^k - x^{\ast}_{\tau_{k+1}}}_2^2 + \frac{L_f^2}{m}\norms{x^k - x^{\ast}_{\tau_{k+1}}}_2^2 \vspace{1.5ex}\\
&\leq \big(L - 2\mu_f + \frac{L_f^2}{m}\big)\norms{x^k - x^{\ast}_{\tau_{k+1}}}_2^2.
\end{array}
}
This estimate can be simplified as in \eqref{eq:lm12_proof_est2} with $\omega := \frac{1}{m}\big((L-2\mu_f)m + L_f^2\big)^{1/2} < 1$.
\end{proof}
%% End of the proof.

%%% Lemma 12.
\begin{lemma}\label{le:contraction_of_prox_nt}
Assume that $f$ is self-concordant.
Let $x^{k+1}$ be  generated by the inexact scheme \eqref{eq:homotopy_method2} up to a given accuracy $\delta_k \geq 0$.
Let  $\lambda_{k+1} := \Vert x^{k+1} -  x^{\ast}_{\tau_{k+1}}\Vert_{x^{\ast}_{\tau_{k+1}}}$ and $\hat{\lambda}_k := \Vert x^{k} -  x^{\ast}_{\tau_{k+1}}\Vert_{x^{\ast}_{\tau_{k+1}}}$ be defined by \eqref{eq:local_distances22}.
If $1 - 4\hat{\lambda}_k + 2\hat{\lambda}_k^2 > 0$, then we have
\myeq{eq:contraction_of_prox_nt}{
\lambda_{k+1} \leq \left(\frac{3 - 2\hat{\lambda}_k}{1 - 4\hat{\lambda}_k + 2\hat{\lambda}_k^2}\right) \hat{\lambda}_k^2 + \frac{\delta_k}{1 - \hat{\lambda}_k}.
}
\end{lemma}

%%% Begin of the proof of Lemma 13.
\begin{proof}
Note that $\nabla{f}_{\tau_{k+1}}(\cdot) := \nabla{f}(\cdot) - (\tfrac{1}{\tau_{k+1}} - 1)\xi^0$.
Using \eqref{eq:fixed_point1} with $H \equiv H^{\ast}_{k} :=  \nabla^2f(x^{\ast}_{\tau_{k+1}})$ and $\tau = \tau_{k+1}$, we have 
\myeq{eq:lm13_est1}{
x^{\ast}_{\tau_{k+1}} = \prox_{\frac{1}{\tau_{k+1}}g}^{H_{k}^{\ast}}\Big( x^{\ast}_{\tau_{k+1}} - {H^{\ast}_{k}}^{-1} \nabla{f_{\tau_{k+1}}}(x^{\ast}_{\tau_{k+1}}) \Big).
}
Define $H_{k} := \nabla^2{f}(x^k)$ and $e_k := {H_{k}^{*}}^{-1}\left(H_{k}^{*} - H_{k}\right)(\bar{x}^{k+1} - x^k)$.
Then, we can rewrite \eqref{eq:cvx_subprob_k} as
\myeq{eq:lm13_est2}{
\bar{x}^{k+1} := \prox_{\frac{1}{\tau_{k+1}}g}^{H^{\ast}_{k}}\left(x^k - {H^{\ast}_{k}}^{-1}\nabla{f_{\tau_{k+1}}}(x^k) + e_k\right).
}
Using \eqref{eq:nonexpansiveness}, we can derive from \eqref{eq:lm13_est1} and \eqref{eq:lm13_est2} that
\vspace{-0.75ex}
\begin{align}\label{eq:lm13_est3}
\Vert \bar{x}^{k+1} - x^{\ast}_{\tau_{k+1}} \Vert_{x^{\ast}_{\tau_{k+1}}} {\!\!\!\!\!}&= \Big\Vert \prox_{\frac{1}{\tau_{k+1}}g}^{H^{\ast}_{k}}\left(x^k - {H^{\ast}_{k}}^{-1}\nabla{f_{\tau_{k+1}}}(x^k) + e_k\right) \nonumber\\
&~~~~~ - \prox_{\frac{1}{\tau_{k+1}}g}^{H_{k}^{\ast}}\big( x^{\ast}_{\tau_{k+1}} - {H^{\ast}_{k}}^{-1} \nabla{f_{\tau_{k+1}}}(x^{\ast}_{\tau_{k+1}}) \big) \Big\Vert_{x^{\ast}_{\tau_{k+1}}} \nonumber\\
&\overset{\tiny\eqref{eq:nonexpansiveness}}{\leq} \big\Vert {H^{\ast}_{k}}^{-1} \big( \nabla{f_{\tau_{k+1}}}(x^k) - \nabla{f_{\tau_{k+1}}}(x^{\ast}_{\tau_{k+1}}) - H^{\ast}_{k}(x^k - x^{\ast}_{\tau_{k+1}})\big)  - e_k \big\Vert_{x^{\ast}_{\tau_{k+1}}} \nonumber\\
&\leq  \big\Vert  \nabla{f_{\tau_{k+1}}}(x^k) - \nabla{f_{\tau_{k+1}}}(x^{\ast}_{\tau_{k+1}}) - H^{\ast}_{k}(x^k - x^{\ast}_{\tau_{k+1}})\big\Vert_{x^{\ast}_{\tau_{k+1}}}^{\ast} {\!\!\!} + \Vert e_k \Vert_{x^{\ast}_{\tau_{k+1}}}.
\vspace{-0.75ex}
\end{align}
First, we estimate the following term of \eqref{eq:lm13_est3} using definition of $\nabla{f}_{\tau_{k+1}}$ and $H_k^{\ast}$ as
\myeqn{\begin{array}{ll}
A_k &:=  \big\Vert  \nabla{f_{\tau_{k+1}}}(x^k) - \nabla{f_{\tau_{k+1}}}(x^{\ast}_{\tau_{k+1}}) - H^{\ast}_{k}(x^k - x^{\ast}_{\tau_{k+1}}) \big\Vert_{x^{\ast}_{\tau_{k+1}}}^{\ast} \vspace{1ex}\\
& = \big\Vert  \nabla{f}(x^k) - \nabla{f}(x^{\ast}_{\tau_{k+1}}) -  \nabla^2f(x^{\ast}_{\tau_{k+1}})(x^k - x^{\ast}_{\tau_{k+1}})   \big\Vert_{x^{\ast}_{\tau_{k+1}}}^{\ast}.
\end{array}}
By the self-concordance of $f$, with the same proof as in \cite[Theorem 4.1.14]{Nesterov2004}, we can derive {that}
\myeqn{
\begin{array}{ll}
A_k & \leq \frac{\Vert x^k - x^{\ast}_{\tau_{k+1}} \Vert_{x^{\ast}_{\tau_{k+1}}}^2}{1 {~}-{~} \Vert x^k - x^{\ast}_{\tau_{k+1}}\Vert_{x^{\ast}_{\tau_{k+1}}}}.
\end{array}
}
Next, using \cite[Corollary 4.1.4]{Nesterov2004}, we estimate $\Vert e_k\Vert_{x^{\ast}_{\tau_{k+1}}}$ as follows:
\myeqn{
\begin{array}{ll}
\Vert e_k \Vert_{x^{\ast}_{\tau_{k+1}}} &=  \big\Vert \big(H_{k}^{*} - H_{k}\big)(\bar{x}^{k+1} - x^k)\big\Vert_{x^{\ast}_{\tau_{k+1}}}^{\ast} \vspace{1ex}\\
& = \big\Vert \big(\nabla^2{f}(x^{\ast}_{\tau_{k+1}}) - \nabla^2{f}(x^{k})\big)(\bar{x}^{k+1} - x^k) \big\Vert_{x^{\ast}_{\tau_{k+1}}}^{\ast} \vspace{1ex}\\
&\leq \Bigg[\frac{2 \Vert x^k -~ x^{\ast}_{\tau_{k+1}} \Vert_{x^{\ast}_{\tau_{k+1}}} - ~ \Vert x^k -~ x^{\ast}_{\tau_{k+1}} \Vert_{x^{\ast}_{\tau_{k+1}}}^2 }{\big(1 {~} - {~} \Vert x^k -~ x^{\ast}_{\tau_{k+1}} \Vert_{x^{\ast}_{\tau_{k+1}}} \big)^2} \Bigg] \cdot \Vert \bar{x}^{k+1} -~ x^k \Vert_{x^{\ast}_{\tau_{k+1}}},
\end{array}
}
provided that $\Vert x^k - x^{\ast}_{\tau_{k+1}} \Vert_{x^{\ast}_{\tau_{k+1}}} {\!\!\!}< 1$.
Combining these two estimates, we can derive from \eqref{eq:lm13_est3} that
\myeqn{{\!\!\!}
\Vert \bar{x}^{k+1} {\!\!}- x^{\ast}_{\tau_{k+1}} \Vert_{x^{\ast}_{\tau_{k+1}}} \leq \tfrac{\Vert x^k - x^{\ast}_{\tau_{k+1}} \Vert_{x^{\ast}_{\tau_{k+1}}}^2}{1 - \Vert x^k - x^{\ast}_{\tau_{k+1}}\Vert_{x^{\ast}_{\tau_{k+1}}}} + \tfrac{\big[ 2 \Vert x^k - x^{\ast}_{\tau_{k+1}} \Vert_{x^{\ast}_{\tau_{k+1}}} -~ \Vert x^k - x^{\ast}_{\tau_{k+1}} \Vert_{x^{\ast}_{\tau_{k+1}}}^2 \big] }{\big(1 - \Vert x^k - x^{\ast}_{\tau_{k+1}} \Vert_{x^{\ast}_{\tau_{k+1}}}\big)^2} \cdot \Vert \bar{x}^{k+1} - x^k \Vert_{x^{\ast}_{\tau_{k+1}}}.
{\!\!\!}}
Note that $\Vert \bar{x}^{k+1} - x^k \Vert_{x^{\ast}_{\tau_{k+1}}} \leq \Vert \bar{x}^{k+1} - x^{\ast}_{\tau_{k+1}} \Vert_{x^{\ast}_{\tau_{k+1}}} +~ \Vert x^k - x^{\ast}_{\tau_{k+1}} \Vert_{x^{\ast}_{\tau_{k+1}}}$ by the triangle inequality.
Using this {in} the last inequality, and rearranging the result, we obtain
\myeqn{
\Vert \bar{x}^{k+1} - x^{\ast}_{\tau_{k+1}} \Vert_{x^{\ast}_{\tau_{k+1}}} \leq \Bigg(\tfrac{3 {~}-{~} 2\Vert x^k - x^{\ast}_{\tau_{k+1}} \Vert_{x^{\ast}_{\tau_{k+1}}}}{1 {~} - {~} 4\Vert x^k - x^{\ast}_{\tau_{k+1}} \Vert_{x^{\ast}_{\tau_{k+1}}} {~} + {~~} 2\Vert x^k - x^{\ast}_{\tau_{k+1}} \Vert_{x^{\ast}_{\tau_{k+1}}}^2}\Bigg)\Vert x^k - x^{\ast}_{\tau_{k+1}} \Vert_{x^{\ast}_{\tau_{k+1}}}^2.
}
Moreover, using the self-concordance of $f$ in \cite[Theorem 4.1.5]{Nesterov2004}, we can also derive {that}
\myeqn{
\begin{array}{ll}
\Vert x^{k+1} - x^{\ast}_{\tau_{k+1}} \Vert_{x^{\ast}_{\tau_{k+1}}}  &\leq \Vert \bar{x}^{k+1} - x^{\ast}_{\tau_{k+1}} \Vert_{x^{\ast}_{\tau_{k+1}}} + \Vert x^{k+1} - \bar{x}^{k+1}\Vert_{x^{\ast}_{\tau_{k+1}}} \vspace{1ex}\\
&\leq \Vert \bar{x}^{k+1} - x^{\ast}_{\tau_{k+1}} \Vert_{x^{\ast}_{\tau_{k+1}}} + \tfrac{\Vert x^{k+1} - \bar{x}^{k+1} \Vert_{x^k}}{1 {~}-{~} \Vert x^k -  x^{\ast}_{\tau_{k+1}} \Vert_{x^{\ast}_{\tau_{k+1}}} }.
\end{array}
}
Combining the last two inequalities and using the definition of $\lambda_{k+1}$, $\hat{\lambda}_k$, and $\delta_k$, we  obtain
\myeqn{ 
\lambda_{k+1} \leq \left(\frac{3 - 2\hat{\lambda}_k}{1 - 4\hat{\lambda}_k + 2\hat{\lambda}_k^2}\right) \hat{\lambda}_k^2 + \frac{\delta_k}{1 - \hat{\lambda}_k},
}
which is exactly \eqref{eq:contraction_of_prox_nt}.
Here, the right-hand side is well-defined since $1 - 4\hat{\lambda}_k + 2\hat{\lambda}_k^2 > 0$.
\end{proof}
%%% End of the proof of Lemma 13.

%%% B.1. The proof of Theorem 4.
\subsection{Proof of Theorem~\ref{th:linear_convergence1}: Global linear convergence for smooth and strongly convex function $f$}\label{subsec:linear_convergence1}
From \eqref{eq:lm12_proof_est2} of Lemma~\ref{le:contraction_of_prox_grad}, we can write
\myeqn{ 
\norms{x^{k+1} - x^{\ast}_{\tau_{k+1}}}_2 \leq \omega\norms{x^k - x^{\ast}_{\tau_{k+1}}}_2.
}
Using the estimate \eqref{eq:lm11_bound_fx0_b1} of Lemma~\ref{le:auxi_results} with $\tau = \tau_{k+1}$ and $\hat{\tau} = \tau_k$, we obtain
\myeqn{ 
\norms{x^{\ast}_{\tau_{k+1}} - x^{\ast}_{\tau_k}}_2 \leq \frac{\vert \tau_{k+1} - \tau_k\vert}{\mu_f\tau_k\tau_{k+1}}\left(L_g + \norms{\xi^0}_2\right).
}
Combining these two estimates and  using the triangle inequality, we can derive that
\myeq{eq:th4_proof_est4}{
\norms{x^{k+1}  - x_{\tau_{k +1}}^{\ast}}_2 \leq \omega \norms{x^k - x^{\ast}_{\tau_{k}}}_2 + \omega\left(L_g + \norms{\xi^0}_2\right) \frac{\Delta\tau_k}{\mu_f\tau_k\tau_{k+1}},
} 
where $\omega := \frac{1}{m}\sqrt{(L-2\mu_f)m + L_f^2} < 1$, and $\Delta{\tau_k} := \vert\tau_{k+1} - \tau_k\vert$.
In addition, by using \eqref{eq:lm11_bound_fx0_b1} of Lemma~\ref{le:auxi_results} with $\tau = \tau_0$, we have
\myeq{eq:th4_proof_est5}{
\norms{x^0  - x_{\tau_0}^{\ast}}_2 \leq C := \frac{\norms{\nabla{f}(x^0) + \xi^0}_2}{\mu_f}.
}
Now, let us assume that $\norms{x^k - x^{\ast}_{\tau_{k}}}_2 \leq C \sigma^k$ for some $\sigma \in (\omega, 1]$ and {$C$ is}  given in \eqref{eq:th4_proof_est5}.
In order to get $\norms{x^{k+1} - x^{\ast}_{\tau_{k+1}}}_2 \leq C \sigma^{k+1}$, from \eqref{eq:th4_proof_est4}, we impose the following condition:
\myeqn{
\omega C \sigma^k + \omega(L_g + \norms{\xi^0}_2)\frac{\Delta{\tau}_k}{\mu_f\tau_k\tau_{k+1}} \leq C \sigma^{k+1}.
}
Since $\tau_k < \tau_{k+1}$ and $\sigma \in (\omega, 1]$, this inequality is equivalent to 
\myeqn{
\tau_{k+1} \leq \tau_k + \tfrac{C(\sigma - \omega)\mu_f}{\omega(L_g + \norms{\xi^0}_2)}\sigma^k\tau_k\tau_{k+1}. 
}
{Define} $C_1 := \frac{C(\sigma - \omega)\mu_f}{\omega(L_g + \norms{\xi^0}_2)} = \frac{(\sigma - \omega)\norms{\nabla{f}(x^0) + \xi^0}_2}{\omega(L_g + \norms{\xi^0}_2)} > 0$.
Then, the last inequality leads to
\myeq{eq:cond_on_sigma}{
\tau_{k+1} \leq \tau_k + C_1\sigma^k\tau_k\tau_{k+1}.
}
Using Lemma~\ref{le:contraction_sequence} with $M := C_1$ and $q := \sigma$, we get
%{Now} define $s_k := \frac{1}{\tau_k}$.
%Then, \eqref{eq:cond_on_sigma} is equivalent to $s_{k+1} \geq s_k - C_1\sigma^k$.
%By induction, we have $s_k \geq s_0 - C_1\sum_{i=0}^{k-1}\sigma^i = s_0 - \frac{C_1(1-\sigma^k)}{1-\sigma}$.
%Hence, we can show that $\tau_k \leq \frac{\tau_0(1-\sigma)}{1-\sigma -\tau_0C_1(1-\sigma^k)}$. This estimate is equivalent to
\myeqn{
1 - \tau_k \geq \frac{(1-\tau_0)(1-\sigma) - \tau_0C_1 + \tau_0C_1\sigma^k}{1 - \sigma - \tau_0C_1(1-\sigma^k)}.
}
Let us choose $\tau_0$ and $\sigma$ such that $(1-\tau_0)(1-\sigma) - \tau_0C_1 = 0$.
Using the formula of $C_1$ above, this condition is equivalent to
\myeq{eq:cond_on_param}{
\frac{(1-\tau_0)(1-\sigma)}{\tau_0} = \frac{(\sigma - \omega)\norms{\nabla{f}(x^0) + \xi^0}_2}{\omega(L_g + \norms{\xi^0}_2)}.
}
Let us fix $\tau_0 \in (0, 1)$. 
Then \eqref{eq:cond_on_param} shows that we can choose $\sigma :=  \frac{1-\tau_0 + \tau_0\omega\Gamma}{1-\tau_0 + \tau_0\Gamma} \in (\omega, 1)$ as shown in \eqref{eq:quantities1} to guarantee \eqref{eq:cond_on_param}, where $\Gamma := \frac{\norms{\nabla{f}(x^0) + \xi^0}_2}{\omega(L_g + \norms{\xi^0}_2)}$.

With the choice of $\tau_0$ and $\sigma$, we have $0 \leq 1 - \tau_k \leq \frac{(1-\tau_0)\sigma^k}{\tau_0 + (1-\tau_0)\sigma^k}$ and  $0 < \frac{(1-\tau_0)\sigma^k}{\tau_0 + (1-\tau_0)\sigma^k} < 1$.
If we update $1-\tau_k := \frac{(1-\tau_0)\sigma^k}{\tau_0 + (1-\tau_0)\sigma^k}$ as shown in \eqref{eq:sigma_k}, then the condition \eqref{eq:cond_on_sigma} holds.
Moreover,  $0 \leq 1 - \tau_k \leq \frac{(1-\tau_0)\sigma^k}{\tau_0}$, where the right-hand side $\frac{(1-\tau_0)\sigma^k}{\tau_0}$ converges to zero as $k$ tends to $+\infty$.

Finally, since $x^{\star} = x^{\ast}_1$, by the triangle inequality, we have
\myeqn{
\begin{array}{ll}
\norms{x^k - x^{\star}}_2 &\leq \norms{x^k - x^{\ast}_{\tau_k}}_2 + \norms{x^{\ast}_{\tau_k} - x^{\ast}_1}_2 \leq C \sigma^k + \omega(L_g + \norms{\xi^0}_2)\frac{(1-\tau_k)}{\mu_f\tau_k\tau_{k+1}}\vspace{1ex}\\
&\leq C \sigma^k + \omega(L_g + \norms{\xi^0}_2)\frac{(1-\tau_0)\sigma^k}{\mu_f\tau_0^3} \vspace{1ex}\\
&\leq  \Big[C +  \frac{\omega(L_g + \norms{\xi^0}_2)(1-\tau_0)}{\tau_0^3\mu_f} \Big] \sigma^k,
\end{array}
}
Hence, we obtain $\norms{x^k - x^{\star}}_2 \leq \hat{C}\sigma^k$ with $\hat{C} := C  +  \frac{\omega(L_g + \norms{\xi^0}_2)(1-\tau_0)}{\tau_0^3\mu_f}$.
This is exactly the last conclusion of Theorem \ref{th:linear_convergence1}.
\Eproof
%%% End of the proof.

%%% The proof of Theorem~5b.
\subsection{Proof of Theorem~\ref{th:linear_convergence2b}: Linear convergence under Assumption~\ref{as:A4}}\label{apdx:th:linear_convergence2b}
Let $\lambda_k$ and $\hat{\lambda}_k$ be defined by \eqref{eq:local_distances22} and  let $\Delta_k^{\ast} := \Vert x^{\ast}_{\tau_{k+1}} - x^{\ast}_{\tau_k}\Vert_{x_{\tau_{k+1}}^{\ast}}$ be the local distance between the true solutions $x^{\ast}_{\tau_{k+1}}$ and $x^{\ast}_{\tau_k}$ of \eqref{eq:new_reparam_opt_cond}.
We divide the proof of Theorem~\ref{th:linear_convergence2b}  into the following steps:

%%% Step 1: Upper bound on \lambda_{k+1}.
\textit{Step 1: Upper bound on $\lambda_{k+1}$:}
Recall \eqref{eq:contraction_of_prox_nt} from Lemma~\ref{le:contraction_of_prox_nt} as follows:
\myeq{eq:th100_est1}{
\lambda_{k+1} \leq \left(\frac{3 - 2\hat{\lambda}_k}{1 - 4\hat{\lambda}_k + 2\hat{\lambda}_k^2}\right) \hat{\lambda}_k^2 + \frac{\delta_k}{1 - \hat{\lambda}_k}.
}
By the triangle inequality and \cite[Theorem 4.1.5]{Nesterov2004}, we can also derive {that}
\myeqn{
\hat{\lambda}_k = \Vert x^k - x^{\ast}_{\tau_{k+1}}\Vert_{x^{\ast}_{\tau_{k+1}}} \leq {~} \Delta_k^{\ast} +  \tfrac{\Vert x^k - x^{\ast}_{\tau_k}\Vert_{x^{\ast}_{\tau_k}}}{1 {~} - {~} \Delta^{\ast}_k}  = \Delta^{\ast}_k + \tfrac{\lambda_k}{1 {~} - {~} \Delta_k^{\ast}}.
}
{Now}, consider the function $\psi(t) := \frac{3 - 2t}{1 - 4t + 2t^2}$ on $[0, 0.1145]$. 
It is straightforward to numerically check that {$3\leq \psi(t) \leq 5$}.
Hence, we can overestimate \eqref{eq:th100_est1} as follows:
\myeqn{
\lambda_{k+1} \leq 1.13\delta_k + 5\hat{\lambda}_k^2,~~~\text{for all}~ \hat{\lambda}_k \in [0, 0.1145].
}
Using the fact that $\hat{\lambda}_k \leq \frac{\lambda_k}{1 - \Delta_k^{\ast}} + \Delta_k^{\ast}$ above, we overestimate again this inequality as
\myeqn{
\lambda_{k+1} \leq 1.13\delta_k + 5\left(\frac{\lambda_k}{1 - \Delta_k^{\ast}} + \Delta_k^{\ast}\right)^2,
}
where we require $\frac{\lambda_k}{1 - \Delta_k^{\ast}} + \Delta_k^{\ast} \leq 0.1145$.

Now, if we impose $\lambda_k \leq 0.05$, then $\frac{\lambda_k}{1 - \Delta_k^{\ast}} + \Delta_k^{\ast} \leq 0.1145$ implies that $\Delta_k^{\ast} \leq 0.055$.
Therefore, if we choose $\delta_k := \tfrac{\lambda_k}{113}$, then we can further simplify the above inequality to get
\myeq{eq:main_est}{
\lambda_{k+1} \leq \tfrac{1}{100}\lambda_k + 5\left( \tfrac{10\lambda_k}{9}+ \Delta_k^{\ast}\right)^2,
}
as long as $\lambda_k \leq 0.05$ and $\Delta_k^{\ast} \leq 0.055$.

%%% Step 2: Upper bound on Delta_k.
\textit{Step 2: Upper bound on $\Delta^{\ast}_k$:}
Let us assume that $\lambda_k \leq 0.05\sigma^k$ for some $\sigma \in (0, 1]$.
In order to guarantee $\lambda_{k+1} \leq 0.05\sigma^{k+1}$, using the last inequality \eqref{eq:main_est}, we have to impose the following condition 
\myeqn{
\tfrac{0.05\sigma^k}{100} + 5\left( \frac{0.5\sigma^k}{9}+ \Delta_k^{\ast} \right)^2 \leq 0.05\sigma^{k+1}.
}
{Provided that $\sigma > \frac{25}{81}+0.01 = 0.318642$,
the previous
condition is equivalent to}
\myeq{eq:main_est2}{
\Delta_k^{\ast} \leq  \Delta_k := \sqrt{\frac{(\sigma - 0.01)\sigma^k}{100}} - \frac{\sigma^k}{18} = {C_k} \sqrt{\sigma}^k,
}
{
where $C_k :=  \frac{1}{10}\sqrt{\sigma - 0.01} - \frac{1}{18}\sqrt{\sigma}^k > 0$.}

%%% Step 3: Upper bound on Delta_k.
\textit{Step 3: Update rule of $\tau_k$:}
Using \eqref{eq:lm11_bound_fx0_b2} with $\tau := \tau_{k+1}$ and $\hat{\tau} := \tau_k$, we obtain
\myeq{eq:th100_est2c}{
\frac{\Delta_k^{\ast}}{1 + \Delta_k^{\ast}} \leq \frac{2L_g\abs{\tau_{k+1} - \tau_k}}{\tau_k\tau_{k+1}}.
}
If we update $\tau_k$ as  \eqref{eq:sigma_choice_2b}, then we can see that
\myeq{eq:cond1000a}{
\frac{2L_g(\tau_{k+1} - \tau_k)}{\tau_k\tau_{k+1}} = \frac{C_k\sqrt{\sigma}^k}{1 + C_k\sqrt{\sigma}^k} \overset{\tiny\eqref{eq:th100_est2c}}{\geq} \frac{\Delta^{\ast}_k}{1 + \Delta^{\ast}_k}.
}
This guarantees \eqref{eq:main_est2}.

%%% Step 4: Find condition on $\sigma$ and $tau_0$.
\textit{Step 4: Conditions on $\sigma$ and $\tau_0$:}
{Define} $\bar{C}_0 := \frac{2L_g(1-\tau_0)(1 - \sqrt{\sigma})}{\tau_0}$ and choose $\sigma \in (0.318642, 1]$ such that
\myeq{eq:cond1000}{
\bar{C}_0\sqrt{\sigma}^k  \leq \frac{C_k\sqrt{\sigma}^k}{1 + C_k\sqrt{\sigma}^k}  =  \frac{2L_g(\tau_{k+1} - \tau_k)}{\tau_k\tau_{k+1}}. 
}
{Note that the above} condition holds if
\myeq{eq:cond_of_sigma}{
 \frac{\bar{C}_0}{1 - \bar{C}_0\sqrt{\sigma}^k} \leq \left( \frac{1}{10}\sqrt{\sigma - 0.01} - \frac{1}{18}\sqrt{\sigma}^k \right).
}
{In turn, the  condition \eqref{eq:cond_of_sigma} holds} if we choose $\sigma \in (0.318642, 1)$ such that $\frac{\bar{C}_0}{1-\bar{C}_0} \leq \frac{\sqrt{\sigma - 0.01}}{10} - \frac{1}{18}$.
Using {the explicit expression of} $\bar{C}_0 $, we obtain {that} 
\myeqn{
\frac{2L_g(1-\tau_0)(1- \sqrt{\sigma})}{\tau_0 - 2L_g(1-\tau_0)(1-\sigma)} \leq \frac{\sqrt{\sigma - 0.01}}{10} - \frac{1}{18}, 
}
which is exactly the condition~\eqref{eq:param_cond2}.
It is not hard to show that this { inequality} always has a solution $\sigma \in (0.318642, 1]$.
Moreover,  since $\Delta_k^{\ast} \leq 0.055$, from \eqref{eq:main_est2}, we also require $C_k \leq 0.055$ for all $k\geq 0$. 
{It is easy to check numerically that this condition  is always satisfied for}
$\sigma \in  (0.318642, 1]$.

%%% Step 5: Bound on tau_k
\textit{Step 5: Bound on $\tau_k$:}
Next, using \eqref{eq:cond1000}, the definition of $\bar{C}_0$, and Lemma~\ref{le:contraction_sequence} with $M := \bar{C}_0$ and $q := \sqrt{\sigma}$, similar to the proof of Theorem~\ref{th:linear_convergence1}, we can show that 
\myeq{eq:cond200}{
\tau_k \geq 1 - \frac{(1-\tau_0)\sqrt{\sigma}^k}{\tau_0 + (1-\tau_0)\sqrt{\sigma}^k} \geq 1 - \left(\frac{1-\tau_0}{\tau_0}\right)\sqrt{\sigma}^k. 
}
Hence $0 \leq 1 - \tau_k \leq \big(\frac{1-\tau_0}{\tau_0}\big)\sqrt{\sigma}^k$.

%%% Step 6: Convergence of  xk.
\textit{Step 6: Linear convergence rate of $\set{x^k}$:}
Finally, note that $x^{\star} = x^{\ast}_1$ and $\tau_{k+1} = 1$, we can show from \eqref{eq:th100_est2c} and \eqref{eq:cond200} that
\myeqn{ 
\frac{\Vert x^{\ast}_1 - x^{\ast}_{\tau_k}\Vert_{x^{\ast}_1}}{1 + \Vert x^{\ast}_1 - x^{\ast}_{\tau_k}\Vert_{x^{\ast}_1} } \leq \frac{2L_g(1 - \tau_k)}{\tau_k} \overset{\tiny\eqref{eq:cond200}}{=} \bar{C}_0\sqrt{\sigma}^k.
}
Since in \eqref{eq:cond_of_sigma} we have chosen $\bar{C}_0$ such that $\bar{C}_0\sqrt{\sigma}^k < 1$, this inequality implies {that}
\myeq{eq:th203_est2b}{
\Delta_{1k}^{\ast} := \Vert x^{\ast}_1 - x^{\ast}_{\tau_k}\Vert_{x^{\ast}_1} \leq \frac{\bar{C}_0\sqrt{\sigma}^k}{1 - \bar{C}_0\sqrt{\sigma}^k} \leq \frac{\bar{C}_0\sqrt{\sigma}^k}{1-\bar{C}_0} = C_1\sqrt{\sigma}^k,
}
where $C_1 :=  \frac{\bar{C}_0}{1 -  \bar{C}_0} > 0$.
Using the triangle inequality, the self-concordance of $f$, and \eqref{eq:th203_est2b}, we 
can derive {that}
\myeqn{
\Vert x^k - x^{\star}\Vert_{x^{\star}} \leq \Vert x^k - x^{\ast}_{\tau_k}\Vert_{x^{\ast}_1} + \Vert x^{\ast}_{\tau_k} - x^{\ast}_1\Vert_{x^{\ast}_1}  \leq \tfrac{\Vert x^k - x^{\ast}_{\tau_k}\Vert_{x^{\ast}_{\tau_k}} }{1 - \Delta_{1k}^{\ast}} + \Delta_{1k}^{\ast}  \leq \tfrac{0.05\sigma^k}{1 - C_1\sqrt{\sigma}^k} + C_1\sqrt{\sigma}^k \leq \hat{C}\sqrt{\sigma}^k,
}
where $\hat{C} := \frac{0.05}{1-C_1} + C_1 = \frac{0.05(1-\bar{C}_0)}{1 - 2\bar{C}_0} + \frac{\bar{C}_0}{1-\bar{C}_0}$.
This inequality shows that $\set{x^k}$ converges to a solution $x^{\star}$ of \eqref{eq:composite_cvx} at a linear rate.
%The last statement of the theorem follows from this conclusion and Lemma~\ref{le:bound_subgrad}.
\Eproof
%%% End of the proof

%%% The proof of Theorem~5.
\subsection{Proof of Theorem~\ref{th:linear_convergence2}: Linear convergence 
when $f$ is a self-concordant barrier}\label{apdx:th:linear_convergence2}
Let $\Delta_k^{\ast} := \Vert x^{\ast}_{\tau_{k+1}} - x^{\ast}_{\tau_k}\Vert_{x_{\tau_{k+1}}^{\ast}}$ as defined in Theorem~\ref{th:linear_convergence1}.
Using \eqref{eq:lm11_bound_fx0_b2} with $\tau := \tau_{k+1}$ and $\hat{\tau} := \tau_k$, we get
\myeqn{ 
\frac{\Delta_k^{\ast}}{1 + \Delta_k^{\ast}} \leq \frac{\abs{\tau_{k+1} - \tau_k}}{\tau_k}\Big[ \Vert \nabla{f}(x^{\ast}_{\tau_{k+1}})\Vert_{x^{\ast}_{\tau_{k+1}}}^{\ast} + \Vert \xi^0 \Vert_{x^{\ast}_{\tau_{k+1}}}^{\ast}\Big].
}
Since $\Vert \xi^0 \Vert_{x^{\ast}_{\tau_{k+1}}}^{\ast} \leq \bar{c}_0 := \theta_f\Vert \xi^0\Vert^{\ast}_{x^{\star}_f}$ and $\Vert \nabla{f}(x^{\ast}_{\tau_{k+1}})\Vert_{x^{\ast}_{\tau_{k+1}}}^{\ast} \leq \sqrt{\nu_f}$ due to the self-concordant barrier 
{property} of $f$, 
the last inequality leads to 
\myeq{eq:th100_est2}{
\frac{\Delta_k^{\ast}}{1 + \Delta_k^{\ast}} \leq \frac{\abs{\tau_{k+1} - \tau_k}}{\tau_k}(\sqrt{\nu_f} + \bar{c}_0).
}
Since we want to guarantee $\norms{x^k - x^{\ast}_{\tau_k}}_{x^{\ast}_{\tau_k}} \leq \beta\sigma^k$ for  $\beta = 0.05$ and $\sigma \in (0, 1]$ as in Theorem~\ref{th:linear_convergence2b}, similar to the proof of Theorem~\ref{th:linear_convergence2b}, we can choose 
\myeq{eq:cond_Delta_star}{
\Delta_k^{\ast} \leq \Delta_k := \left( \frac{1}{10}\sqrt{\sigma - 0.01} - \frac{1}{18}\sqrt{\sigma}^k \right)\sqrt{\sigma}^k \equiv C_k\sqrt{\sigma}^k,
}
where $C_k := \frac{1}{10}\sqrt{\sigma - 0.01} - \frac{1}{18}\sqrt{\sigma}^k$.
Here, $C_k > 0$ for $k\geq 0$ if $\sigma \in (0.318642, 1]$.

Now, from the update rule  \eqref{eq:sigma_choice} of Theorem~\ref{th:linear_convergence2}, we have
\myeqn{
\frac{\abs{\tau_{k+1} - \tau_k}}{\tau_k}(\sqrt{\nu_f} + \bar{c}_0) = \frac{C_k\sqrt{\sigma}^k}{1 + C_k\sqrt{\sigma}^k} \geq \frac{\Delta_k^{\ast}}{1 + \Delta_k^{\ast}}.
}
This inequality shows that \eqref{eq:cond_Delta_star} automatically holds.

Let $a_k := \frac{C_k\sqrt{\sigma}^k}{(\sqrt{\nu_f} + \bar{c}_0)(1 + C_k\sqrt{\sigma}^k)} \in (0, 0.052133]$.
It is easily to show that $a_k \geq \tfrac{C_0\sqrt{\sigma}^k}{(\sqrt{\nu_f} + \bar{c}_0)(1 + C_0\sqrt{\sigma}^k)}$, where $C_0 := \frac{\sqrt{\sigma - 0.01}}{10} - \frac{1}{18} > 0$.
By induction, we get $\tau_k = \tau_0\prod_{i=0}^k(1+a_i)$.
Using an elementary inequality $\prod_{i=0}^k(1+a_i) \geq 1 + \sum_{i=0}^ka_i$, we have 
\myeqn{
\tau_k \geq \tau_0 + \frac{C_0\tau_0}{(1 + C_0)(\sqrt{\nu_f}+\bar{c}_0)}\sum_{i=0}^k\sqrt{\sigma}^i = \tau_0 + \frac{\tau_0C_0(1-\sqrt{\sigma}^k)}{(1 + C_0)(\sqrt{\nu_f}+\bar{c}_0)(1-\sqrt{\sigma})}.
}
Hence, for any $\tau_0 \in (0, 1)$, if we choose $\sigma \in (0.318642, 1]$ such that $\sigma \geq \big(1 - \frac{\tau_0C_0}{(1-\tau_0)(1 + C_0)(\sqrt{\nu_f} + \bar{c}_0)}\big)^2$, then we have $1 - \tau_k \leq \frac{C_0\sqrt{\sigma}^k}{(1+ C_0)(\sqrt{\nu_f} + \bar{c}_0)(1-\sqrt{\sigma})}$.

Finally, note that $x^{\star} = x^{\ast}_1$, we can show from \eqref{eq:th100_est2} that
\myeqn{ 
\frac{\Vert x^{\ast}_1 - x^{\ast}_{\tau_k}\Vert_{x^{\ast}_1}}{1 + \Vert x^{\ast}_1 - x^{\ast}_{\tau_k}\Vert_{x^{\ast}_1} } \leq \left(\frac{1-\tau_k}{\tau_k}\right)\left(\sqrt{\nu_f} + \bar{c}_0\right) \leq \frac{C_0\sqrt{\sigma}^k}{\tau_0(1+C_0)(1-\sqrt{\sigma})}.
}
This shows that
\myeq{eq:th203_est2}{
\Delta_{1k}^{\ast} := \Vert x^{\ast}_1 - x^{\ast}_{\tau_k}\Vert_{x^{\ast}_1} \leq 
{\widetilde{C}_1}\sqrt{\sigma}^k,
}
where ${\widetilde{C}_1} := \frac{C_0}{\tau_0(1 + C_0)(1-\sqrt{\sigma}) - C_0} > 0$.
Next, using the triangle inequality, the self-concordance of $f$, and \eqref{eq:th203_est2}, we can derive
\myeqn{
\Vert x^k - x^{\star}\Vert_{x^{\star}} \leq \Vert x^k - x^{\ast}_{\tau_k}\Vert_{x^{\ast}_1} + \Vert x^{\ast}_{\tau_k} - x^{\ast}_1\Vert_{x^{\ast}_1}  \leq \frac{\Vert x^k - x^{\ast}_{\tau_k}\Vert_{x^{\ast}_{\tau_k}} }{1 - \Delta_1^{\ast}} + \Delta_1^{\ast}  \overset{\tiny\eqref{eq:th203_est2}}{\leq} \frac{\beta\sigma^k}{1 - \widetilde{C}_1\sqrt{\sigma}^k} + {\widetilde{C}_1} \sqrt{\sigma}^k.
}
This inequality shows that $\set{x^k}$ converges to a solution $x^{\star}$ of \eqref{eq:composite_cvx} at a linear rate.
\Eproof
%%% End of the proof.

%%% The proof for the choice of tau0.
\subsection{Proof of  \eqref{eq:choice_of_tau0}: The choice of $\tau_0$}\label{apdx:eq:choice_of_tau0}
From \eqref{eq:lm11_bound_fx0_b3}, we have $\norms{x^0 - x^{\ast}_{\tau_0}}_2 \leq \frac{\tau_0}{\mu_g}\norms{\nabla{f}(x^0) + \xi^0}_2$.
Let $\lambda_{\max}(\nabla^2{f}(x^0))$ be the maximum eigenvalue of $\nabla^2{f}(x^0)$.
Hence, if we assume that 
\myeqn{
\gamma := \lambda_{\max}(\nabla^2{f}(x^0))^{1/2}\frac{\tau_0}{\mu_g}\norms{\nabla{f}(x^0) + \xi^0}_2 < 1, 
}
then we have 
\myeqn{
\norms{x^0 - x^{\ast}_{\tau_0}}_{x^0} \leq \lambda_{\max}(\nabla^2{f}(x^0))^{1/2}\norms{x^0 - x^{\ast}_{\tau_0}}_2 \leq \gamma < 1.
}
Since 
\myeqn{
\norms{x^0 - x^{\ast}_{\tau_0}}_{x^{\ast}_{\tau_0}} \leq \frac{\norms{x^0 - x^{\ast}_{\tau_0}}_{x^0}}{1 - \norms{x^0 - x^{\ast}_{\tau_0}}_{x^0}} \leq \frac{\gamma}{1-\gamma}
}
as long as $\norms{x^0 - x^{\ast}_{\tau_0}}_{x^0} < 1$, to guarantee 
{that} $\norms{x^0 - x^{\ast}_{\tau_0}}_{x^{\ast}_{\tau_0}} \leq \beta$, we impose $\frac{\gamma}{1-\gamma} \leq \beta$.
The last condition is equivalent to $\gamma \leq \frac{\beta}{1+\beta}$, which shows that 
\myeqn{
\tau_0 \leq \frac{\beta\mu_g}{(1+\beta)\lambda_{\max}(\nabla^2{f}(x^0))^{1/2}\norms{\nabla{f}(x^0) + \xi^0}_2}. 
}
This condition is  exactly \eqref{eq:choice_of_tau0}.
\Eproof
%%% End of the proof.

%%% The proof of Theorem 7. Finding an inital point.
\subsection{Proof of Theorem \ref{th:initial_point}: Finding an initial point}\label{apdx:th:initial_point}
Let $\lambda_j := \Vert x^j - x^{\ast}_{t_j}\Vert_{x^{\ast}_{t_j}}$ and 
$\Delta^{\ast}_j := \Vert x^{\ast}_{t_{j+1}} - x^{\ast}_{t_j}\Vert_{x^{\ast}_{t_{j+1}}}$
Assume that $\lambda_j \leq \beta$, from \eqref{eq:main_est}, to guarantee $\lambda_{j+1} \leq \beta$, we impose {the condition} 
\myeqn{
\frac{\beta}{100} + 5\left(\frac{10\beta}{9} + \Delta_j^{\ast}\right)^2 \leq \beta.
}
This condition holds if $\Delta_j^{\ast} \leq \Theta := \sqrt{\frac{99\beta}{500}} - \frac{10\beta}{9}$.
Since $\beta \leq 0.1145 < 0.16038$, we have $\Theta := \sqrt{\frac{99\beta}{500}} - \frac{10\beta}{9} > 0$.
Using \eqref{eq:lm11_bound_fx0_b4} with $t := t_{j+1}$ and $\hat{t} := t_j$, we obtain
{that}
\myeqn{
\frac{\Delta_j^{\ast}}{1 + \Delta_j^{\ast}} \leq (t_j - t_{j+1})\norms{\nabla{f}(\hat{x}^0) + \hat{\xi}^0}_{x^{\ast}_{t_{k+1}}}^{\ast}.
}
Since $f$ is $\mu_f$-strongly convex, we have $\norms{\nabla{f}(x^0) + \xi^0}_{x^{\ast}_{t_{k+1}}}^{\ast} \leq \frac{1}{\sqrt{\mu_f}}\norms{\nabla{f}(\hat{x}^0) + \hat{\xi}^0}_2$.
Hence, we {get} $\frac{\Delta_j^{\ast}}{1 + \Delta_j^{\ast}} \leq (t_j - t_{j+1})\frac{\norms{\nabla{f}(\hat{x}^0) + \hat{\xi}^0}_2}{\sqrt{\mu_f}}$.

To guarantee this condition, we impose $\frac{\Theta}{1 + \Theta} = M_0(t_j - t_{j+1})$, where $M_0 := \frac{\norms{\nabla{f}(\hat{x}^0) + \hat{\xi}^0}_2}{\sqrt{\mu_f}}$.
This allows us to update $t_j$ {to} $t_{j+1} := t_j - \frac{\Theta}{L_g(1 + \Theta)}$.
If we start from $t_0 \approx 1$, then we have $t_j := t_0 - \frac{\Theta}{M_0(1 + \Theta)}(j+1)$.
Hence, $t_j = 0$ if $j+1 \geq \frac{t_0(1 + \Theta_0)}{\Theta}$.
This shows that after $j_{\max} := \left\lfloor \frac{t_0M_0(1 + \Theta)}{\Theta}\right\rfloor$ iterations, we obtain $\hat{x}^{j_{\max}}$ such that $\norms{\hat{x}^{j_{\max}} - x^{\ast}_{\tau_0}}_{x^{\ast}_{\tau_0}} \leq \beta$.

Finally, we show how to choose $t_0$ such that $\norms{\hat{x}^0 - x^{\ast}_{t_0}}_{x^{\ast}_{t_0}} \leq \beta$. 
Using \eqref{eq:lm11_bound_fx0_b4} with $t := t_0$, if $(1-t_0)\norms{\nabla{f}(\hat{x}^0) + \hat{\xi}^0}_{\hat{x}^0}^{\ast} < \frac{1}{2}$, {then} we have
\myeq{eq:th7_proof3}{
\norms{\hat{x}^0 - x^{\ast}_{t_0}}_{x^0} \leq \frac{(1-t_0)\norms{\nabla{f}(\hat{x}^0) + \hat{\xi}^0}_{\hat{x}^0}^{\ast}}{1 - (1-t_0)\norms{\nabla{f}(\hat{x}^0) + \hat{\xi}^0}_{\hat{x}^0}^{\ast}} < 1.
}
By using Lemma~\ref{le:auxi_results}(e), if $\norms{\hat{x}^0 - x^{\ast}_{t_0}}_{\hat{x}^0} < 1$, then we have $\norms{\hat{x}^0 - x^{\ast}_{t_0}}_{x^{\ast}_{t_0}} \leq \frac{\norms{\hat{x}^0 - x^{\ast}_{t_0}}_{\hat{x}^0}}{1-\norms{\hat{x}^0 - x^{\ast}_{t_0}}
_{\hat{x}^0}}$. To guarantee {that} $\norms{\hat{x}^0 - x^{\ast}_{t_0}}_{x^{\ast}_{t_0}} \leq \beta$, we {require} $\norms{\hat{x}^0 - x^{\ast}_{t_0}}_{\hat{x}^0} \leq \frac{\beta}{1+\beta}$.
Combing this estimate and \eqref{eq:th7_proof3}, we can show that $\norms{\hat{x}^0 - x^{\ast}_{t_0}}_{x^{\ast}_{t_0}} \leq \beta$ if $1-t_0\geq \frac{\beta}{(1+2\beta)\norms{\nabla{f}(\hat{x}^0) + \hat{\xi}^0}_{\hat{x}^0}^{\ast}}$.
Therefore, we can choose
\myeqn{
t_0 := \begin{cases} 1 - \frac{\beta}{(1+2\beta)\norms{\nabla{f}(\hat{x}^0) + \hat{\xi}^0}_{\hat{x}^0}^{\ast}}, ~~&\text{if}~~ \norms{\nabla{f}(\hat{x}^0) + \hat{\xi}^0}_{\hat{x}^0}^{\ast} > \frac{1+2\beta}{\beta},\\
1 & \text{otherwise},
\end{cases}
}
to guarantee that $\norms{\hat{x}^0 - x^{\ast}_{t_0}}_{x^{\ast}_{t_0}} \leq \beta$.
\Eproof
%%% End of the proof.

%% The proof of Lemma 3.6.
\subsection{Proof of Lemma \ref{le:approx_pd_sol}: Approximate solution of the dual problem}\label{apdx:le:approx_pd_sol}

{Define  the} quadratic function
\myeq{eq:q_k}{
q_k(y) := \iprods{\nabla{\varphi_{\tau_{k+1}}}(y^k), y - y^k} + \frac{1}{2}\iprods{\nabla^2{\varphi}(y^k)(y - y^k), y - y^k}.
}
It is easy to show that 
\myeq{eq:lm36_est1}{
\begin{array}{ll}
q_k(y^{k+1}) - q_k(\bar{y}^{k+1}) &=  \iprods{\nabla{\varphi_{\tau_{k+1}}}(y^k) + \nabla^2{\varphi}(y^k)(\bar{y}^{k+1} - y^k), y^{k+1} -  \bar{y}^{k+1}} \vspace{1ex}\\
& +~ \tfrac{1}{2}\iprods{\nabla^2{\varphi}(y^k)(y^{k+1} - \bar{y}^{k+1}), y^{k+1} - \bar{y}^{k+1}}.
\end{array}}
{By using} \eqref{eq:approx_sol3} and  \eqref{eq:primal_sol_recovery_inexact}, we have $y^{k+1} \in  \partial{\psi}(\tau_{k+1}z^{k+1})$, which implies $z^{k+1} \in \frac{1}{\tau_{k+1}}\partial{\psi^{\ast}}(y^{k+1})$.
Hence, by the convexity of $\psi^{\ast}$, we have $\frac{1}{\tau_{k+1}}\big(\psi^{\ast}(y^{k+1}) - \psi^{\ast}(\bar{y}^{k+1})\big)  \leq \iprods{z^{k+1}, y^{k+1} - \bar{y}^{k+1}}$.
Using this inequality and \eqref{eq:lm36_est1}, if we define $\Pc_k(y) := q_k(y) + \frac{1}{\tau_{k+1}}\psi^{\ast}(y)$, then we have
\myeqn{
\begin{array}{ll}
\Pc_k(y^{k+1}) {\!\!\!}&-~ \Pc_k(\bar{y}^{k+1}) \leq  \tfrac{1}{2}\Vert  y^{k+1} - \bar{y}^{k+1}\Vert_{y^k}^2 \vspace{1ex}\\
&+~ \iprods{z^{k+1} + \nabla{\varphi_{\tau_{k+1}}}(y^k) + \nabla^2{\varphi}(y^k)(\bar{y}^{k+1} - y^k), y^{k+1} - \bar{y}^{k+1}}. 
\end{array}}
Next, using \eqref{eq:primal_sol_recovery_inexact}, we have $z^{k+1} =  - \nabla{\varphi_{\tau_{k+1}}}(y^k) - \nabla^2{\varphi}(y^k)(y^{k+1} - y^k - \tilde{e}_k)$ for $\tilde{e}_k \in\R^p$.
Substituting this expression into the last inequality, we obtain
\myeqn{ 
\begin{array}{ll}
\Pc_k(y^{k\!+\!1}) - \Pc_k(\bar{y}^{k\!+\!1}) &\leq  \tfrac{1}{2}\Vert  y^{k+1} \!-\! \bar{y}^{k+1}\Vert_{y^k}^2  + \iprods{ \nabla^2{\varphi}(y^k)(\bar{y}^{k+1} - y^{k+1} + \tilde{e}_k, y^{k\!+\!1} \!-\! \bar{y}^{k\!+\!1}} \vspace{1ex}\\
&= -\tfrac{1}{2}\Vert  y^{k+1} \!-\! \bar{y}^{k+1}\Vert_{y^k}^2   +  \iprods{ \nabla^2{\varphi}(y^k)\tilde{e}_k, y^{k+1} - \bar{y}^{k+1}} \vspace{1ex}\\
&\leq \frac{1}{2}\Vert \tilde{e}_k\Vert_{y^k}^2.
\end{array}
}
The last inequality implies that if $\norm{\tilde{e}_k}_{y^k} \leq \delta$, then $\Pc_k(y^{k+1}) - \Pc_k(\bar{y}^{k+1}) \leq \frac{\delta^2}{2}$. This completes the proof.
\Eproof

%%% The proof of Lemma 7.
\subsection{Proof of Theorem~\ref{le:xk_app_sol}: An approximate solution for $x^{\star}$ of \eqref{eq:composite_cvx}}\label{apdx:le:xk_app_sol}
The statement (a) is trivial. 
We now prove (b) and (c) as follows.

(b)~From \eqref{eq:sol_of_com_cvx}, \eqref{eq:app_sol_of_com_cvx}, and the definition of $\varphi$ in \eqref{eq:grad_hess_of_varphi}, we can show that
\myeqn{
{\!\!\!}\begin{array}{ll}
\Vert x^k - x^{\star}\Vert_{x^{\star}} {\!\!\!\!\!}&:= \iprods{\nabla^2{f^{\ast}}(-D^{\top}y^{\star})^{-1}(x^k - x^{\star}), (x^k - x^{\star})}^{1/2} \vspace{1ex}\\
&= \iprods{ \nabla^2{f^{\ast}}(-D^{\top}y^{\star})^{-1}(\nabla{f^{\ast}}(-D^{\top}y^k) \!-\! \nabla{f^{\ast}}(-D^{\top}y^{\star})), \nabla{f^{\ast}}(-D^{\top}y^k) \!-\! \nabla{f^{\ast}}(-D^{\top}y^{\star}) }.
\end{array}{\!\!\!}
}
Let  $G_k :=  \int_0^1\nabla^2{f^{\ast}}(-D^{\top}{y^{\star}} -  s D^{\top}(y^k - y^{\star})ds$.
By the mean-value theorem and \cite[Corollary 4.1.4]{Nesterov2004}, we can derive from the above expression that
\myeqn{
\begin{array}{ll}
\Vert x^k - x^{\star}\Vert_{x^{\star}}^2 &= {(y^k - y^{\star})^\top}DG_k^{\top}\nabla^2{f^{\ast}}(-D^{\top}y^{\star})^{-1}G_kD^{\top}(y^k - y^{\star}) \vspace{1ex}\\
&\leq \frac{ \iprods{D\nabla^2{f^{\ast}}(-D^{\top}y^{\star})D^{\top}(y^k -y^{\star}), y^k -y^{\star}}}{\left( 1 - \iprods{D\nabla^2{f^{\ast}}(-D^{\top}y^{\star})D^{\top}(y^k -y^{\star}), y^k -y^{\star}}^{1/2}\right)} \vspace{1ex}\\
& = \frac{\iprods{\nabla^2{\varphi}(y^{\star})(y^k - y^{\star}), y^k - y^{\star}} }{\big( 1 - \iprods{\nabla^2{\varphi}(y^{\star})(y^k - y^{\star}), y^k - y^{\star}}^{1/2} \big)^2} \vspace{1ex}\\
&= \frac{\Vert y^k - y^{\star}\Vert_{y^{\star}}^2}{ \left( 1- \Vert y^k - y^{\star}\Vert_{y^{\star}}\right)^2}.
\end{array}
}
Here, we note that $\nabla^2{\varphi}(y^{\star}) = D\nabla^2{f^{\ast}}(-D^{\top}y^{\star})D^{\top}$.
This inequality leads to \eqref{eq:xk_app_sol} provided that $ \Vert y^k - y^{\star}\Vert_{y^{\star}} < 1$.

Since we apply \eqref{eq:homotopy_method4} to solve the dual problem \eqref{eq:dual_prob}, by Theorem \ref{th:linear_convergence2b} or Theorem \ref{th:linear_convergence2}, the sequence $\set{y^k}$ satisfies $\Vert y^k - y^{\star}\Vert_{y^{\star}} \leq \hat{C}\sqrt{\sigma}^k$ for given constants $\hat{C} > 0$ and $\sigma \in (0, 1)$.
Combining this relation and \eqref{eq:xk_app_sol}, we can show that $\set{x^k}$ converges linearly to $x^{\star}$ an optimal solution of \eqref{eq:composite_cvx}.
%%%% End of the proof.

(c)~First, since $\nabla{\varphi}_{\tau}(y) =  D\nabla{f^{\ast}}(-D^{\top}y) - \big(\frac{1}{\tau} - 1\big)\xi^0$, using \eqref{eq:sol_of_com_cvx} we can  write
\myeqn{
Dx^{\star} = D\nabla{f^{\ast}}(-D^{\top}y^{\star}) =  -\nabla{\varphi_{\tau_{k+1}}}(y^{\star}) + \big(\tfrac{1}{\tau_{k+1}} - 1\big)\xi^0.
}
Next, from \eqref{eq:primal_sol_recovery_inexact} we can write 
\myeqn{
z^{k+1} = \nabla^2{\varphi}(y^k)(y^{k+1} - y^k - \tilde{e}_k) - \nabla{\varphi_{\tau_{k+1}}}(y^k).
}
Combining these expressions, we can further estimate 
\myeqn{
\begin{array}{lcl}
\Vert z^{k+ 1}{\!} - Dx^{\star}\Vert^{\ast}_{y^k} {\!\!\!} &=& \Vert \nabla^2{\varphi}(y^k)(y^{k\!+\!1} {\!\!}-\! y^k \!-\! \tilde{e}_k) - \nabla{\varphi_{\tau_{k\!+\!1}}}(y^k) + \nabla{\varphi_{\tau_{k\!+\!1}}}(y^{\star}) - \big(\tfrac{1}{\tau_{k\!+\!1}} - 1\big)\xi^0\Vert_{y^k}^{\ast}\vspace{1ex}\\
&\leq& \Vert \nabla{\varphi_{\tau_{k+1}}}(y^{\star}) - \nabla{\varphi_{\tau_{k+1}}}(y^k) - \nabla^2{\varphi}(y^k)(y^{\star} - y^k)\Vert^{\ast}_{y^k}\vspace{1ex}\\
&&+ \Vert y^{k+1} - y^{\star}\Vert_{y^k} +  \big(\tfrac{1}{\tau_{k+1}} - 1\big)\Vert\xi^0\Vert^{\ast}_{y^k} + \Vert \tilde{e}_k\Vert_{y^k} \vspace{1ex}\\
&& \frac{\Vert y^{\star} - y^k\Vert_{y^k}^2 }{1 - \Vert y^{\star} - y^k\Vert_{y^k}} + \Vert y^{k+1} - y^{\star}\Vert_{y^k} + \big(\tfrac{1}{\tau_{k+1}} - 1\big)\Vert\xi^0\Vert^{\ast}_{y^k} + \Vert \tilde{e}_k\Vert_{y^k},
\end{array}
}
provided that $\Vert y^{\star} - y^k\Vert_{y^k} < 1$.
Here, the last inequality follows from the self-concordance of $\varphi_{\tau}$ with similar proof as  \cite[Theorem 1]{TranDinh2016c}.
This proves \eqref{eq:norm_opt_sol_pd}.

Since we apply \eqref{eq:homotopy_method4} to solve the dual problem \eqref{eq:dual_prob}, by Theorem \ref{th:linear_convergence2b} or Theorem \ref{th:linear_convergence2}, we have $\Vert y^k - y^{\star}\Vert_{y^k} \leq \hat{C}\sqrt{\sigma}^k$ and $\Vert y^{k+1} - y^{\star}\Vert_{y^k} \leq \hat{C}\sqrt{\sigma}^{k+1}$ for some constants $\hat{C} > 0$ and $\sigma \in (0, 1)$.
In addition, $0 \leq \frac{1}{\tau_{k+1}} - 1 = \frac{1-\tau_{k+1}}{\tau_{k+1}} \leq \frac{1 - \tau_{k+1}}{\tau_0} \leq C_1\sqrt{\sigma}^{k+1}$ for some  constant $C_1 > 0$.
Using these bounds and $\Vert \tilde{e}_k\Vert_{y^k} \leq \delta_k$ in \eqref{eq:norm_opt_sol_pd}, we can show that
\myeqn{
\Vert z^{k+1} - Dx^{\star}\Vert^{\ast}_{y^k} \leq \frac{\hat{C}^2\sigma^k}{1 - \hat{C}\sqrt{\sigma}^k} + \hat{C}\sqrt{\sigma}^{k+1} + C_1\Vert\xi^0\Vert_{y^k}^{\ast}\sqrt{\sigma}^{k+1} + \delta_k.
}
Under the conditions of Theorem \ref{th:linear_convergence2b} or Theorem \ref{th:linear_convergence2}, we have $\Vert\xi^0\Vert_{y^k}^{\ast}  = \iprods{\nabla^2{\varphi}(y^k)^{-1}\xi^0, \xi^0}^{1/2} \leq M$ for some $M > 0$ and $\delta_k \leq \frac{\beta}{113}\sigma^k$ for  $\beta = 0.05$, then we can easily show that 
\myeqn{
\Vert z^{k+1} -  Dx^{\star}\Vert^{\ast}_{y^k} \leq \left(\tfrac{\hat{C}^2}{1-\hat{C}} + \hat{C} + C_1M + \tfrac{\beta}{113}\right) \sqrt{\sigma}^k, 
}
provided that $\hat{C} < 1$. 
If $D$ is invertible, then we can show that $x^{k+1} := D^{-1}z^{k+1}$ is an approximate solution to $x^{\star}$ of \eqref{eq:composite_cvx}.
Moreover, $\norms{x^k - x^{\star}}_{y^k}$ converges to zero at a linear rate.
\Eproof
%% End of the proof.

%%%%%%%%%%%%%%%%%%%%%%%%%%%%%%%%%%%%%%%%%%%%%%%%%
%+ References.
%%%%%%%%%%%%%%%%%%%%%%%%%%%%%%%%%%%%%%%%%%%%%%%%%
\bibliographystyle{siamplain}
%\bibliography{/Users/quoctd/Dropbox/E-Books/tran_bibtex_new}

\end{document}